\pdfoutput=1 
\RequirePackage[l2tabu, orthodox]{nag} 

\newcommand{\nocontentsline}[3]{}
\newcommand{\tocless}[2]{\bgroup\let\addcontentsline=\nocontentsline#1{#2}\egroup}

\documentclass[12pt, reqno]{amsart}
\usepackage[breakable,skins]{tcolorbox}
\newtcolorbox{tbox}[1][]{%
    breakable,
    enhanced,
    colframe=white,
    coltitle=white,
    #1
}

\usepackage{amssymb, amsmath, amsthm}
\usepackage[margin=3cm]{geometry}

\usepackage{quiver}
\usepackage{graphicx} 
\usepackage[linktoc=page]{hyperref}
\usepackage{amssymb}
\usepackage{amsmath, amsthm, amssymb, thmtools}
\usepackage[inline]{asymptote} 
\usepackage{tikz}
\usepackage{tikz-cd}
\usepackage{hyperref}
\usepackage{listings}
\usepackage{mathrsfs}
\usepackage{float}
\usepackage{enumitem}

\newtheorem{theo}{Theorem}[section]
\newtheorem{proposition}[theo]{Proposition}
\newtheorem{theorem}[theo]{Theorem}
\newtheorem{question}[theo]{Question}
\newtheorem{claim}[theo]{Claim}
\newtheorem{corollary}[theo]{Corollary}
\newtheorem{lemma}[theo]{Lemma}
\newtheorem{lemdef}[theo]{Lemma/Definition}

\theoremstyle{definition}
\newtheorem{example}[theo]{Example}
\newtheorem{definition}[theo]{Definition}
\newtheorem{remark}[theo]{Remark}

\newtheorem*{convention*}{Convention}
\newtheorem{computation}[theo]{Computation}

\newcommand{\defi}[1]{\textsf{#1}} 

\newlist{thmcases}{enumerate}{1}
\setlist[thmcases]{
  label=\textbf{\upshape Case~\thetheorem.\arabic*},
  leftmargin=*,
  ref={\thetheorem.\arabic*}}

\newcommand{\tr}{\widetilde{R}}

\newcommand{\Vol}{\text{Vol}}

\newcommand{\Ti}{\mathscr T}

\usepackage[OT2,T1]{fontenc}
\DeclareSymbolFont{cyrletters}{OT2}{wncyr}{m}{n}
\DeclareMathSymbol{\Sha}{\mathalpha}{cyrletters}{"58}
\DeclareMathOperator{\Width}{Width}

\title{Toric Elliptic Pairs with Picard Number Three}
\author{Aditya Khurmi}

\usepackage{shortcuts}
\begin{document}

\begin{abstract}
An \defi{elliptic pair} $(X, C)$ is a generalization of a rational elliptic fibration $X \to \PP^1$ with fiber $C,$ introduced in \cite{jenia_blowup}. Here, $X$ is a projective rational surface with log terminal singularities, and $C$ is an irreducible curve contained in the smooth locus of $X,$ with $p_a(C)=1$ and $C^2=0.$ 
These naturally arise as blowups $X:=\Bl_e(\PP_\Delta)$ of projective toric surfaces, whose Newton polygon is \defi{elliptic}.
The order of $\mc O(C)|_C$ in $\Pic^0(C)$ gives a quantitative way to check if $X$ is an elliptic fibration, which is equivalent to finiteness of the order.
We call $\Delta$ a \defi{Lang-Trotter polygon} when this order is infinite, in which case $\overline{\text{Eff}(\Bl_e(\PP_\Delta))}$ is non-polyhedral.
The paper \cite{lizzie} shows there are exactly $3$ elliptic triangles up to $\SL_2(\ZZ),$ none of which is Lang-Trotter. 
The paper \cite{jenia_blowup} gives an infinite family of Lang-Trotter pentagons and heptagons, and various examples of other polygons when $\rho(\PP_\Delta)>2.$ Remark 4.7 in the paper asks if any Lang-Trotter quadrilaterals exist, and we answer this in the negative by studying the curves in the Zariski Decomposition of $K_X+C.$
\end{abstract}

\maketitle

\tableofcontents

\section{Introduction}

For a projective variety $X,$ one defines $\Eff(X),$ the cone of all effective divisors on $X.$ The \defi{pseudo-effective cone} is then defined as the closure $\ol{\Eff(X)}$ in $\text{NS}(X)_\RR$ which captures a lot of the birational geometry of $X.$ 
In recent work \cite{jenia_blowup} by Castravet, Laface, Tevelev and Ugaglia, they constructed a special class of surfaces $X$ which include various projective surfaces with non-polyhedral pseudo-effective cones.
More formally, an \defi{elliptic pair} $(C,X)$ consists of a projective rational surface $X$ with log terminal singularities and an irreducible curve $C \subseteq X$ of arithmetic genus one, disjoint from the singular locus of $X$ and such that $C^2=0.$ 
An interesting class of such surfaces is given by rational elliptic fibrations $X \to \PP^1$ with a fiber $\simeq C.$

\medskip

For elliptic surfaces, the order of $\mc O(C)|_C$ in $\Pic^0 C$ is an interesting object of study. For instance, $\ol{\Eff(X)}$ is non-polyhedral when the order is infinite \cite[Lem. 3.3]{jenia_blowup}. A method of constructing these surfaces is by blowing up toric surfaces coming from certain \defi{elliptic polygons}. 
Concretely, let $\PP_\Delta$ be the (projective) toric surface associated to a lattice polygon $\Delta \subseteq \ZZ^2.$
For a hyperplane section $H_\Delta$ of $\PP_\Delta,$ 
let $\mc L_\Delta$ denote the linear system $|H_\Delta|,$ and let $\mc L_\Delta(m)$ denote the linear subsystem consisting of curves $\Gamma$ which have multiplicity at least $m$ through $e,$ a fixed point in the torus.

\begin{definition}
\label{def: definition of elliptic polygon}
    Let $\Delta \subseteq \ZZ^2$ be a lattice polygon with lattice perimeter $m=|\partial \Delta \cap \ZZ^2|.$ We call $\Delta$ an \defi{elliptic polygon} if 
    \begin{enumerate}
        \item $\Vol(\Delta)=m^2,$ where $\Vol(\Delta)$ is the \defi{normalized volume}, numerically equal to twice the area.
        \item $\dim \mc L_\Delta(m)=1,$ and the unique (basis) curve $\Gamma$ of $\mc L_\Delta(m)$ is irreducible.
        \item the Newton Polygon of $\Gamma$ coincides with $\Delta.$
    \end{enumerate}
\end{definition}

Recall here that for a field $k$ and $f = \sum_{u \in \ZZ^2} a_uz^u \in [x^{\pm 1},y^{\pm 1}],$ the \defi{Newton Polygon of $f$} is the convex hull of the points $u \in \ZZ^2$ with $a_u \ne 0.$ In this paper, we assume $k$ is algebraically closed of arbitrary characteristic. 

\begin{example}
\label{example:random example}
    Consider 
    $$\Delta:\lt[\begin{matrix}
        0 & 1 & 20 & 7\\
        0 & 0 & 14 & 5
    \end{matrix}\rt].$$
    We can verify this is an elliptic quadrilateral with $m=4;$ see Computation \ref{comp:random example}. 
\end{example}

\begin{remark}
    Elliptic polygons are called \defi{good} in \cite{jenia_blowup}. The reason we call them elliptic is their connection to elliptic pairs; see the definition that follows.
\end{remark}

\begin{definition}
    Let $X:=\Bl_e(\PP_\Delta)$ be the blowup of $\PP_\Delta$ at $e.$ Let $E$ be the exceptional divisor and $C$ be the strict transform of $\Gamma.$ Theorem 4.4 in 
    \cite{jenia_blowup} shows that the conditions for $\Delta$ to be elliptic imply that $(X, C)$ is an elliptic pair.  We call such an elliptic pair $(X,C)$ a \defi{toric elliptic pair}.    
\end{definition}

\begin{remark}
\label{remark: toric boundary for X}
    In what follows, we will abuse notation and use terms like ``toric boundary'', ``toric fan'' etc. for $X$ as well, unless stated otherwise. 
    This is possible since $\PP_\Delta$ and $X$ are away from $e, E$ resp., which are disjoint from the toric boundary.
\end{remark}

\begin{definition}
    Let $\Delta$ be an elliptic polygon, and $X=\Bl_e(\PP_\Delta)$ be the corresponding blown up toric surface and $C$ the strict transform of $\Gamma.$ If $\res\;C$ has infinite order in $\Pic^0C$, we call $\Delta$ a \defi{Lang-Trotter} polygon. Equivalently (cf. Lemma–Definition 3.2. in \cite{jenia_blowup}), $\Delta$ is Lang-Trotter if $h^0(X,\mc O(nC))=1$ for all $n>0.$
\end{definition}

While an interesting object in its own right, \cite{jenia_blowup} uses this construction to prove other results such as the non-polyhedrality of the pseudo-effective cone of the Grothendieck–Knudsen moduli space $\ol{M_{0,n}}$ for $n \ge 10.$
However if one studies these Lang-Trotter polygons, the first natural question is the following:
\begin{question}
\label{question: does there exist lang trotter}
    Does there exist a Lang-Trotter $n$-gon for every $n \ge 3?$
\end{question}
The authors in \cite{jenia_blowup} provide infinite families of pentagons and heptagons (see Remarks 5.14, 5.15), and provide a database of various other examples.
They also remark in the introduction that the answer to Question \ref{question: does there exist lang trotter} is not known.
Pratt in \cite{lizzie} shows the answer is negative in the case for triangles.
In this paper, we consider the quadrilateral case, and prove the following:

\begin{theorem}
\label{theorem:key theorem, no Lang Trotter}
There are no Lang-Trotter quadrilaterals.
\end{theorem}

This answers the question raised in Remark 4.7 in \cite{jenia_blowup}. 
Our method of approach is to study the curves in the Zariski Decomposition of $K_X+C$ (for $X=\Bl_e(\PP_\Delta),$ with $\Delta$ Lang-Trotter) which, in this case, is a non-negative integral decomposition into irreducible negative curves\footnote{a \defi{negative curve} is a curve with negative self-intersection} (cf. \cite[Lem. 3.4. and Cor. 3.12]{jenia_blowup}).
These are also closely related to the generators of $\ol{\Eff}(X).$

\medskip

It is easy to prove that $K_X+C \sim 0$ is not possible.
We prove more in this paper:

\begin{theorem}
\label{theorem: main theorem}
    Let $(X,C)$ be a toric elliptic pair with $\rho(X)=3.$ 
    The Zariski Decomposition of $K_X+C$ has exactly two curves, all of which are generators of extremal rays of the pseudoeffective cone $\ol{\Eff(X)}.$
\end{theorem}

Lemma \ref{lemma:dim of linear system} in this paper would then prove Theorem \ref{theorem:key theorem, no Lang Trotter}. Furthermore, Lemma 3.3 in \cite{jenia_blowup} then implies the following:

\begin{corollary}
    Every toric elliptic pair $(X,C)$ with $\rho(X)=3$ is an elliptic fibration.
\end{corollary}

The case we need to rule out to prove Theorem \ref{theorem: main theorem} is when $K_X+C \sim \alpha R,$ with $(X,C)$ a toric elliptic pair and $\rho(X)=3.$ 
The approach in this paper is to develop a combinatorial, a geometric and an algebraic constraint.
The first one is purely about the self-intersection numbers of certain curves in the minimal resolution $Y$ of $X,$ and the images of these under a series of contractions. 
Our surface $Y$ is a special example of what we call a \defi{semi-elliptic} surface. 
We define it as follows: 
\begin{definition}
\label{defi: semi-elliptic} 
Let $X$ be a projective toric surface, possibly blown up away from the toric boundary locus (cf. Rem. \ref{remark: toric boundary for X}).
Let $\sigma:Y \to X$ be the minimal resolution. 
We say $Y$ is \defi{semi-elliptic}, if there exists a non-boundary curve $\wt R \subset Y$ so that the following hold:
\begin{enumerate}
    \item $\wt R$ is smooth, rational and a $(-1)$ curve;
    \item Every toric boundary curve of $Y$ has negative self-intersection;
    \item $\wt R$ intersects at most $2$ exceptional curves of $Y \to X,$ and it does so transversally. Further, these two curves don't both lie over the same point in $X;$
    \item Any torus invariant points in $X$ that $R=\sigma_{\ast}\wt R$ doesn't pass through are smooth or Du Val singularities;
    \item Let $S$ be the set of torus invariant points in $X$ that $R$ passes through, so that $|S| \le 2$ by $(3).$
    Define the divisor $\mc D = \sum_{p \in S} \sigma^{-1}p,$ and call $\mc D \cup \wt R$ the \defi{relevant locus} (also see Def. \ref{defi: relevant locus}). 
    Let $\pi:Y \to W$ be the composite map contracting $\wt R$ and any subsequent $(-1)$ curves in the image of $\mc D \cup \wt R$ under these contractions (cf. Def. \ref{lemma: key commutative diagram X,Y,Z,W, and 8 curves}). 
    Then $\pi$ maps the exceptional locus of $Y \to X$ to exactly $8$ curves.
\end{enumerate}
Observe here that apart from the existence of $\wt R,$ all the axioms are purely combinatorial constraints on the self-intersection numbers of the toric boundary of $Y;$ see Example \ref{example: semi-elliptic} below.
\end{definition}

\begin{example}
\label{example: semi-elliptic}
A semi-elliptic quadrilateral $\Delta$ is given by
$$\Delta: \lt[\begin{matrix}
        0 & 39 & 45 & 0\\
        0 & 18 & 21 & 16
    \end{matrix}\rt].$$
    The resolution $Y$ and a hypothetical $(-1)$ curve $\wt R$ are sketched in Figure \ref{fig:semi_example}. 
    The right figure in \ref{fig:semi_example} shows the relevant locus.
    As we can see, it contracts to a smooth point in $Y \to W,$ and so the exceptional locus of $Y \to X$ maps to $8$ curves in $W.$
    Hence, $\Delta$ is semi-elliptic. In fact, $\Delta$ also satisfies $\Vol(\Delta)=m^2$ with $m=27,$ i.e. it satisfies the first condition for being an elliptic polygon (cf. Def. \ref{def: definition of elliptic polygon}).

    \begin{figure}[h]
        \centering
        \includegraphics[scale=0.2]{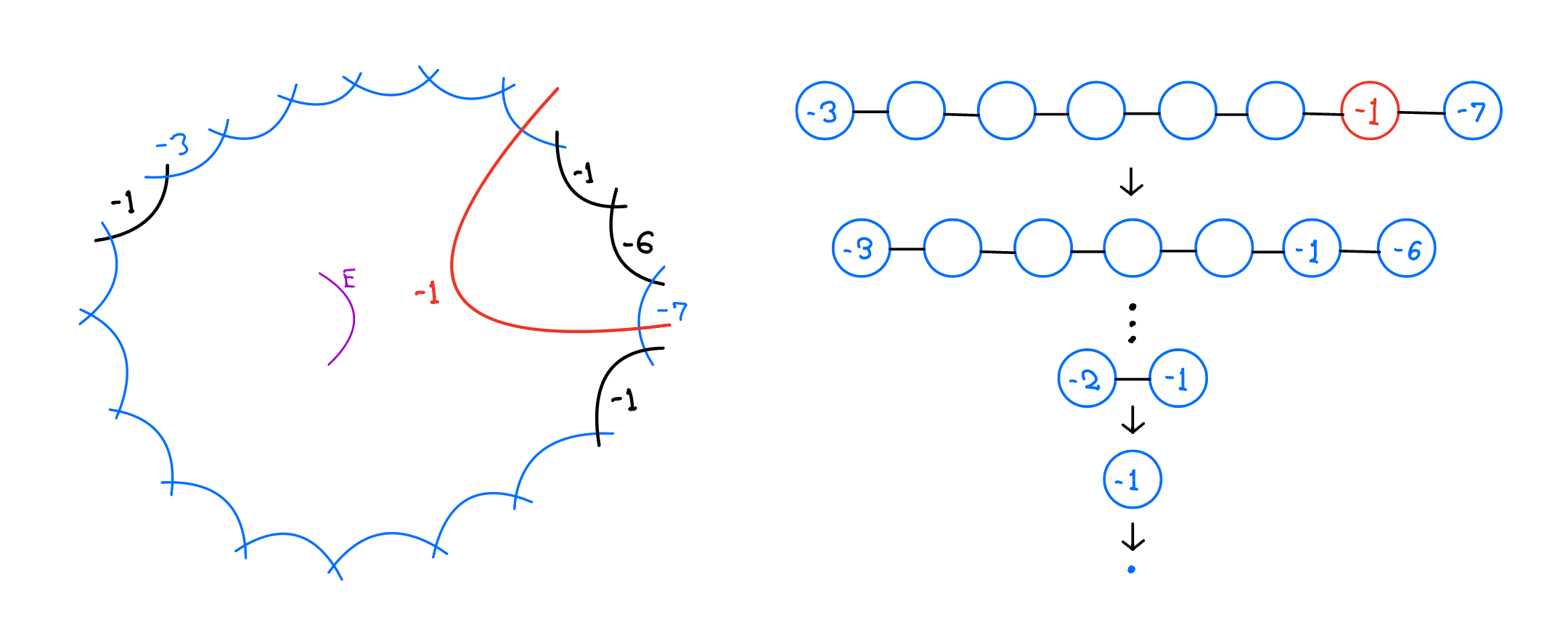}
        \caption{Left: The smooth blown up toric surface $Y$ for $\Delta,$ with a potential $(-1)$ curve intersecting two toric boundary curves. 
        The blue curves form the exceptional divisor of $Y \to X.$
        Any blue curve whose self-intersection has not been written is a $(-2)$ curve. Right: The dual graph of the relevant locus under the contractions $Y \to \dots \to W.$ }
        \label{fig:semi_example}
    \end{figure}
\end{example}

Section \ref{subsection: semi-elliptic} is devoted to proving that $Y$ is semi-elliptic when the Zarisiki Decomposition of $K_X+C$ has exactly one irreducible curve.
Given the singularity types of $X,$ semi-elliptic surfaces are easy to classify using the matrix equation in Lemma \ref{lem: matrix equation} in the Appendix, and hence form the base of our analysis.
We discuss the corresponding case subdivision in subsection \ref{subsec: case division}. 
For the cases when the Du Val singularities in $X$ are of type $DE,$ the combinatorial constraint is enough to restrict to a finite number of cases. This is done in Section \ref{sec: the DE case}.

\medskip

The next is a geometric constraint. We recall the following definition and result:

\begin{definition}
    Let $\Delta \subset \ZZ^2$ be a lattice polygon. For $v \in \ZZ^2,$ the \defi{width of $\Delta$ along $v$}, denoted $\Width_v(\Delta),$ is defined as
    $$\Width_v(\Delta) := \max_{u,w \in \Delta} \la u- w,v\ra, \footnote{In this paper, $\la \bullet, \bullet \ra$ is the standard Euclidean dot product.}.$$

    Subsequently, we define the \defi{width of $\Delta$}, denoted $\Width(\Delta),$ by $\min_{v\in \ZZ^2}\{\Width_v(\Delta)\}.$
\end{definition}

\begin{proposition}[{\cite[Lem.~2.3]{lizzie}}]
\label{prop: width}
    An elliptic polygon $\Delta$ satisfies $\Width(\Delta) \ge m.$
\end{proposition}

Proposition \ref{prop: width}, our geometric constraint, is extremely useful.
Not only is this used to rule out individual cases, it forms the key in understanding the toric fan (see Prop. \ref{prop: negative surface}, Lem. \ref{lemma: special fan}) and analyzing the \defi{Adjacent Du Val case} in Section \ref{sec: adjacent du val case}.

\medskip


Lastly, there is an algebraic constraint, namely $\Vol(\Delta)=m^2.$ We call this ``algebraic'' since this controls the genus of $C,$ see Prop. 4.2 in \cite{jenia_blowup}. Once we find the self-intersection numbers of the toric boundary of $Y$, there is a $\rho(\PP_\Delta)=2$ parameter family of polygons giving such toric surfaces.
Then $\Vol(\Delta)=m^2$ gives another equation; and in case of quadrilaterals gives a homogeneous quadratic form in two of the side lengths, say $\gamma,\delta$.
Up to scaling, roots $\delta/\gamma$ uniquely determine a potential quadrilateral. This approach is used in Sections \ref{sec: the smooth case} and \ref{sec: the A_n case}.
This also forms the bulk of the proof, making the two cases in these sections the hardest part of the paper.

\medskip

The cases in sections \ref{sec: the smooth case} and \ref{sec: the A_n case} are very similar, but have a subtle difference in how we define one of our variables. 
Both are heavily dependent on using \textsc{Macaulay2} to solve a system of matrix equations.
Both have two cases based on the toric boundary curves $\wt R$ intersects.
In the first case in both, the combinatorial and algebraic constraint give a pair of quadrics in $\mathbb{A}^3,$ whose intersection is singular; i.e. has geometric genus $0.$
The way to rule out these cases is an exhaustive case analysis.
The case II in both the sections leads to a biquadratic equation which has finitely many integral solutions.






\subsection*{Acknowledgements}
\tocless

I am extremely grateful to my advisor Professor Jenia Tevelev, for teaching and mentoring me. This project would not have been possible without his constant ideas, motivation and feedback throughout the project. 
Professors Antonio Laface and Luca Ugaglia reviewed the paper and gave helpful feedback.
Discussions with Professor Tom Weston helped in the final part of the paper.
Thanks to Gabriel Ong who worked with me in the initial stages, and Elias Sink for his constructive feedback.
Lastly, I would like to thank my mother for always being there for me.
This project has been partially supported by the NSF grant DMS-2101726 (PI Jenia Tevelev).

\section{Elliptic Pairs and Polygons}
\subsection{Semi-Elliptic Quadrilaterals}
\label{subsection: semi-elliptic}

In this section, we prove results about the combinatorics of the configuration of $Y,$ the minimal resolution of $X.$ 
Then we prove all the essential results to show that $Y$ is semi-elliptic when $K_X+C \sim \alpha R.$

\medskip

First, an essential result from \cite{jenia_blowup} that relates the Zariski Decomposition of $K_X+C$ to the geometry of the cone.

\begin{proposition}
[{\cite[Lem. 3.4. and Cor. 3.12]{jenia_blowup}}]
\label{prop: R are extremal rays}
    Let $(X,C)$ be a toric elliptic pair coming from a Lang-Trotter polygon.
    \begin{enumerate}
        \item The divisor $K_X+C$ is effective, and $h^0(K_X+C)=1.$ 
        \item The Zariski decomposition $K_X+C \sim \alpha_1R_1+\dots+\alpha_nR_n$ is integral, i.e. $\alpha_i \in \ZZ_{>0}.$ Further, $R_i$ have a negative define intersection matrix, and in particular have negative self-intersection.
        \item The curves $R_i$ are generators of extremal rays of $\ol{\Eff(X)}$, and they lie in $C^\perp.$ 
        Further, they can be contracted to give a minimal elliptic model $(Z,C)$ (cf. \cite[Def. 3.5]{jenia_blowup})
    \end{enumerate}
\end{proposition}


\begin{proposition}
\label{prop: negative surface}
    Let $\Delta$ be an elliptic polygon, and $\sigma:Y \to X=\Bl_e(\PP_\Delta)$ be the minimal resolution.
    Then the self-intersection of every toric boundary curve of $Y$ is negative.
\end{proposition}
\begin{proof}
It suffices to show that every toric boundary curve on $\PP_\Delta$ has negative self-intersection.
Assume on the contrary, and let $D \subset \PP_\Delta$ be a toric boundary curve with $D^2>0.$ 
    Assume not, and suppose $D_i^2 \ge 0$ for some toric boundary curve $D_i.$ 
    Then $D_i$ cannot be an exceptional curve of $Y \to X.$ 
    Let $v_{i-1}, v_i, v_{i+1}$ be consecutive primitive vectors in the toric fan $\Ti$ of $Y$ so that $D_i$ corresponds to the vector $v_i.$ Then $D_i^2 \ge 0$ corresponds to $v_{i-1}+v_{i+1}=kv_i$ for $k \le 0.$ Assume without loss of generality that $v_{i-1}$ is the vector $(0,1)$ and $v_i$ is $(1,0),$ so that $k \le 0$ is equivalent to $v_{i+1}=(x,y)$ with $x<0$ (see Figure \ref{fig: negative surface proof}). 
    Let $u, w$ be the closest vectors to $v_i$ in the fan $\Ti$ which correspond to non-exceptional (for $Y \to X)$ divisors, so that the $u,v_{i-1},v_i,v_{i+1},w$ are in this order in $\Ti.$ Note that we might possibly have $u=v_{i-1}$ or $v_{i+1}=w.$

    \begin{figure}[h]
        \centering
        \includegraphics[scale=0.17]{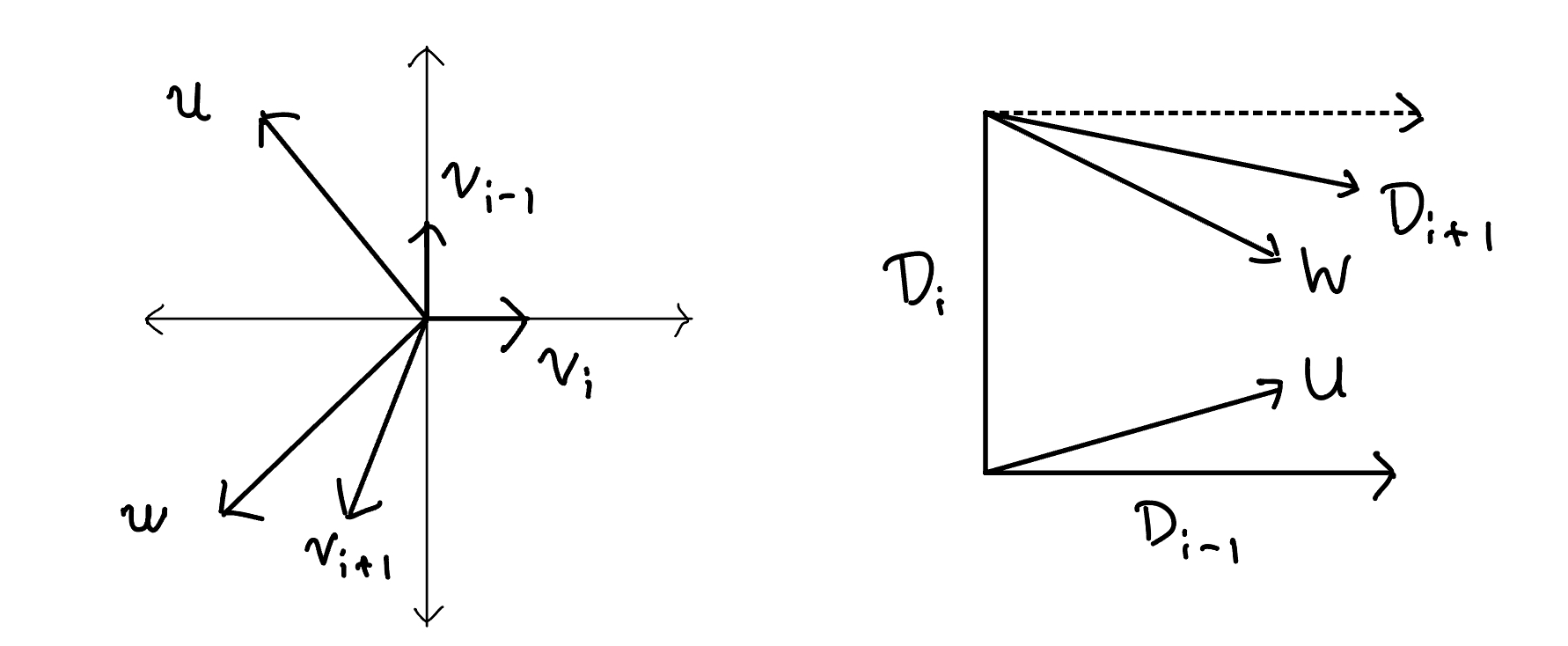}
        \caption{Left: The toric fan (notation from \cite{fulton}). Right: (Part of) the corresponding Newton polygon.}
        \label{fig: negative surface proof}
    \end{figure}

    Abusing notation, let $D_i$ denote the side of $\Delta$ corresponding to the divisor $D_i.$ Further, let $U, W$ be the sides corresponding the vectors $u,w.$ Since the corresponding sides of $X$ are perpendicular to the vectors in $\Ti$ (since the $M$ and $N$ fan are dual), we find that $D_i$ lies along the $y$-axis. Then the unit vector along $D_{i-1}$ is $(1,0),$ so translate to assume $D_i, D_{i-1}$ meet at the origin. 
    Let $\ell_1,\ell_2$ be the two lines parallel to the $x$ axis passing through the endpoints of $D_i,$ where $\ell_1$ lies above $\ell_2.$

    \medskip
    
    The order $u,v_{i-1},v_i$ shows $\measuredangle (U,D_i) < \pi/2,$\footnote{In this paper, we use directed angles $\measuredangle,$ measured counter-clockwise.} meaning $U$ lies above $\ell_2.$ Similarly, $\measuredangle (v_i,v_{i+1}) \ge \pi/2$ shows $\measuredangle (D_i,W) \le \measuredangle (D_{i},D_{i+1}) \le \pi/2,$ and so $W$ lies below $\ell_1.$ Hence, the other endpoints of $U, W$ lie between $\ell_1,\ell_2.$

    \medskip

    By convexity of $\Delta,$ all other vertices lies between $\ell_1, \ell_2.$ Hence,
    $$\Width(\Delta) = \min_{v \in \ZZ^2} \Width_{v}(\Delta) \le \Width_{(0,1)}(\Delta) = |D_i \cap \ZZ^2| < |\partial \Delta \cap \ZZ^2|=m$$
    contradicting the bound $\Width(\Delta) \ge m$ from Proposition \ref{prop: width}.
\end{proof}

\begin{lemma}
\label{lem: at least three (-1) curves}
    There are at least three $(-1)$ toric boundary curves in $Y.$ Further, all of these must be the strict transforms of toric boundary curves in $X.$
\end{lemma}
\begin{proof}
    The latter follows from Proposition \ref{prop: negative surface} and the minimality of the resolution $Y \to X.$ To prove the former, we prove in general that for any toric surface with negative boundary curves, there are at least three $(-1)$ curves. 

    \medskip

     Firstly, $X$ must be different from $\PP^2, \PP^1 \times \PP^1$ and the Hirzebruch surface $\mathbb{F}_a = \mathbf{P}(\mc O \oplus \mc O(-a))$ (for $a \ge 2)$ as all of these have a nonnegative toric boundary curve. Hence, $X$ has at least one $(-1)$ toric boundary curve (cf. \cite[2.5]{fulton}).
     Fix a vector $v_0$ in the toric fan corresponding to this curve.  
     By the contrapositive of Lemma \ref{lemma: convexity of (-2) curves}, there must be another $(-1)$ curve, say corresponding to some $v_i$ with $\measuredangle (v_0,v_i) < \pi$ and so $\measuredangle (v_i,v_0) > \pi$ (recall here that $\measuredangle (\alpha,\beta)$ is the directed angle between $\alpha,\beta,$ measured counter-clockwise). The contrapositive of Lemma \ref{lemma: convexity of (-2) curves} again shows there must be another $(-1)$ curve, corresponding to a vector between $v_i, v_0$ (counter-clockwise), as desired.
\end{proof}

Proposition \ref{prop: R are extremal rays} $(3)$ 
shows that the $R_i$ lie in $C^\perp.$
When $\Delta$ is a Lang-Trotter quadrilateral, $\rho(X)=3$ and hence we can bound the number of $R_i.$

\begin{lemma}
\label{lemma:dim of linear system}
    For a toric elliptic pair $(X,C)$ with $\rho(X)=3,$ the linear system $|K_X+C|$ has exactly one (effective) divisor which is nonzero and has at most $2$ irreducible components.
    Furthermore, for a (hypothetical) Lang-Trotter quadrilateral, the linear system must have exactly one irreducible component.
\end{lemma}
\begin{proof}
By Proposition \ref{prop: R are extremal rays} $(1)$, there is exactly one (effective) divisor in $|K_X+C|.$ Since $\rho(X)=3,$ $C^\perp$ is a subspace of dimension $2,$ and is tangent to the light cone $\{D:D^2 \ge 0\}.$ Hence there are at most two irreducible negative curves in the decomposition, say $R_1, R_2.$ 
We rule out the possibility that there are exactly $2$ curves. Otherwise, these curves are different from $C$ since $C^2=0,$ while the $R_i$ are negative curves (see Prop. \ref{prop: R are extremal rays}).
Furthermore, $C \in C^\perp$ and $C^\perp$ is a subspace intersecting $\ol{\text{Eff}(X)}$ in a facet generated by the two extremal rays $R_1,R_2,$ we must have $C+a_1R_1+a_2R_2 \sim 0$ for some $a_1,a_2 \in \QQ.$ Clearing denominators and restricting to $C$ then, however, shows that $a_0\res (C)=0$ for some $a_0 \in \ZZ.$ 
This contradicts the definition of a Lang-Trotter polygon. 

\medskip

Finally, we consider the case $K_X+C \sim 0.$ As $E \cap X_{\text{sing}} = \phi,$ the adjunction formula for $E$ shows $K_X \cdot E=-1.$ Hence
$$0=(K_X+C) \cdot E=K_X \cdot E+m=2g(E)-2-E^2+m=m-1$$
which shows $m=1,$ a contradiction as $m=|\partial \Delta \cap \ZZ^2| \ge 4.$ Hence the only possibility is $K_X+C \sim \alpha R$ for some $\alpha \in \ZZ_{>0}.$ 
\end{proof}

\begin{remark}
    The above proof is adapted from the proof of Proposition 4.6 in \cite{jenia_blowup}. 
    It must be remarked that the proof to show the strict inequality $\Width(\Delta)>m$ given there is incorrect; the error is near the end which assumes that every interior point of $\Delta$ lies on the segment of lattice length $m-1.$
    While one can fix the proof by a more detailed analysis, the bound $\Width(\Delta) \ge m$ is sufficient in this paper.
\end{remark}

Hence, Theorem \ref{theorem: main theorem} with Lemma \ref{lemma:dim of linear system} imply that there are no Lang-Trotter quadrilaterals, proving Theorem \ref{theorem:key theorem, no Lang Trotter}. So from here on, we proceed with the proof of Theorem \ref{theorem: main theorem}: assume on the contrary that there is a toric elliptic pair $(X,C)$ with $\rho(X)=3$
and $K_X+C \sim \alpha R$ with $\alpha \in \ZZ_{>0}.$ 
Let $\sigma : Y \to X$ be the minimal resolution of $X,$ and $\wt R$ be the strict transform of $R.$
We can now make some geometric remarks about $R$ and $\wt R.$

\begin{lemma}
\label{lemma:tilde r is -1}
\hangindent\leftmargini
$(1)$\hskip\labelsep The curve $R$ is a smooth rational $K$-negative curve, passing through at most $2$ torus invariant points in $X.$
\begin{enumerate}
     \item[(2)] The curve $\tr$ is a smooth rational $K$-negative $(-1)$ curve. In particular, $\wt R$ can be contracted to a smooth point (by Castelnuovo's Theorem). 
    \item[(3)] The curve $\wt R$ intersects any exceptional curve of $Y \to X$ transversally in at most one point. Further, it intersects at most one curve in the exceptional divisor over a fixed torus invariant point.
\end{enumerate}
\end{lemma}
\begin{proof}
\hangindent\leftmargini
$(1)$\hskip\labelsep Firstly, $K_X+C \sim \alpha R$ shows $K_X \cdot R=(K_X+C) \cdot R = \alpha R^2<0,$ so $R$ is $K$-negative.
Hence Proposition 3.8 in \cite{Fujino_2012a} shows $R$ is smooth and rational. As $X$ is a klt surface, 
Lemma 3.2 in \cite{liu2021numbersingularpointsprojective} shows that $R$ passes through at most $2$ singular points, hence at most $2$ torus invariant points (since $X$ is toric outside $E,$ and the only singular points of a toric surface can be the ones fixed by the torus action).
\begin{enumerate}
    \item[(2)] Rationality of $R$ shows $\tr \cong \PP^1$ too, and so the adjunction formula shows
    \begin{equation}
        K_Y \cdot \tr + (\tr)^2 = \deg K_{\tr}= -2.
    \end{equation}
    It suffices to show that $(\tr)^2 < 0$ and $K_Y \cdot \tr < 0.$ 
    As $R$ is a negative curve, $(\tr)^2<0$ as well. 
    Let $\alpha_1,\dots,\alpha_n$ be the discrepancies of $Y \to X.$
    Minimality of the resolution $Y \to X$ shows $\alpha_i \le 0$ for all $i.$ As $E_i, \tr$ are distinct irreducible curves, hence $\alpha_i(E_i \cdot \wt R) \le 0$ for all $i.$
    Further, $K_X \cdot R<0$ as $R$ is $K$-negative. Thus, we find
    \begin{align*}
        K_Y \cdot \tr &= \Big( \sigma^\ast K_X + \sum_{i} \alpha_i E_i \Big) \cdot \tr  \le \sigma^\ast K_X \cdot \wt R = K_X \cdot R < 0.
    \end{align*}
    where we used the projection formula at the end. This shows $K_Y \cdot \wt R=(\wt R)^2 = -1,$ as desired.
    \item[(3)] Assume on the contrary, and let $p \in X$ be a torus invariant point such that $\tr \cdot \sigma^{-1} p \ge 2.$ 
    Under the contraction $Y \to X,$ the point $p$ becomes a singular point of $R,$ contradicting rationality of $R$ (also see Figure \ref{fig:wt R passes through torus inv points}).  \qedhere
\end{enumerate} 
\end{proof}

A clear implication of $(1), (3)$ in Lemma \ref{lemma:tilde r is -1} is the following corollary, which is important enough to be noted separately.

\begin{corollary}
\label{cor: wt R passes through torus inv points}
    The curve $\wt R$ intersects at most two exceptional curves of $Y \to X.$ It does not pass through a torus invariant point which is the intersection of two exceptional (under $Y 
    \to X)$ toric boundary curves (see Figure \ref{fig:wt R passes through torus inv points}). 
\end{corollary}

\begin{figure}[h]
    \centering
    \includegraphics[scale=0.1]{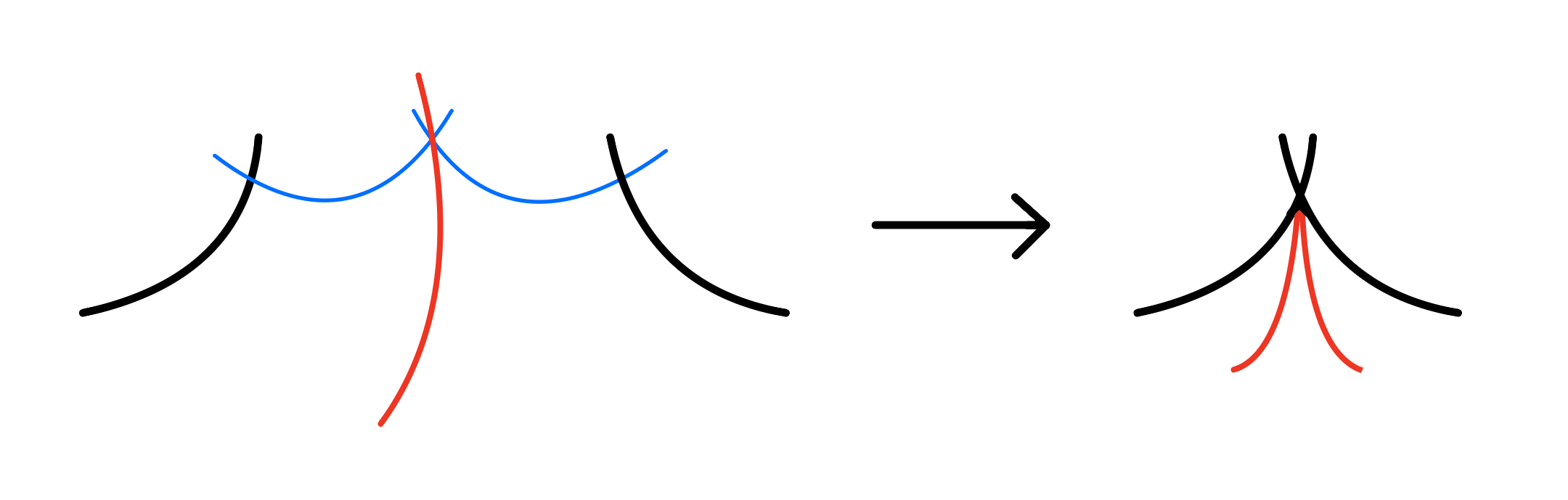}
    \caption{Contraction introduces a singular point on $R$ if $\wt R$ passes through torus invariant points in the exceptional locus}
    \label{fig:wt R passes through torus inv points}
\end{figure}

Proposition \ref{prop: R are extremal rays} shows that $R$ can be contracted to give the minimal elliptic model $(Z,C).$
Further, $\wt R \subset Y$ is a $(-1)$ curve, and hence can be contracted to a smooth point by Castelnuovo's Theorem giving a surface $Y_1.$
This gives a resolution of $Z,$ not necessarily minimal. 
As $Y_1 \to Z$ factors through the minimal resolution $W$ of $Z,$ we get a commutative diagram. 

\[\begin{tikzcd}
	Y & {Y_1} & \dots & Y_\ell & W \\
	X &&&& Z
	\arrow["{\pi_0}", from=1-1, to=1-2]
	\arrow["\pi"{description}, curve={height=-40pt}, from=1-1, to=1-5]
	\arrow["\sigma"', from=1-1, to=2-1]
	\arrow["{\pi_1}", from=1-2, to=1-3]
        \arrow["{\pi_{\ell-1}}", from=1-3, to=1-4]
	\arrow["{\pi_\ell}", from=1-4, to=1-5]
	\arrow["{\wt \sigma}", from=1-5, to=2-5]
	\arrow["{\wt \pi}"', from=2-1, to=2-5]
\end{tikzcd}\]
    
However we can say much more about the map $Y_1 \to W$ and the surface $W.$

\begin{lemma}
\label{lemma: key commutative diagram X,Y,Z,W, and 8 curves}
    There is a series of contractions $Y=Y_0 \to Y_1 \to \dots \to Y_{\ell} \to W$ of $(-1)$ curves, where $W$ is the minimal resolution of $Z,$ and $Y=Y_0 \to Y_1$ is the contraction of $\wt R.$ Here,
    \begin{enumerate}
        \item $\pi$ maps the exceptional locus of $Y \to X$ to the exceptional locus of $W \to Z.$
        \item $Z$ has at most Du Val singularities, and any torus invariant points in $X$ that $R$ doesn't pass through are smooth or Du Val singularities
        \item $\wt \sigma:W \to Z$ has exactly $8$ exceptional curves.
    \end{enumerate}
\end{lemma}
\begin{proof}
    Let $\wt{\mc D}$ be the toric boundary divisor of $Y,$ and $\mc D := \sigma_\ast \wt{\mc D}.$ 
    As $W$ is the minimal resolution of $Z,$ it is the minimal surface over $Z$ with no $(-1)$ boundary curves. As $Y_1$ and $Z$ are isomorphic outside the images of the boundary locii $(\pi_0)_\ast \wt{\mc D}$ and $\wt \pi_\ast \mc D,$ the map $Y_1 \to W$ is a series of contractions of $(-1)$ curves $Y_1 \to \dots \to Y_\ell \to W$ where $\ell$ is the number of toric boundary curves contracted. Further, $(1)$ is clear as $\pi:Y \to W$ only contracts $\wt R$ or the curves in the exceptional locus of $Y \to X.$
    
    \medskip

    The minimal model $Z$ has at most Du Val singularities by Corollary 3.12 in \cite{jenia_blowup}.
    Since $X$ and $Z$ are isomorphic outside $R,$ any singular torus invariant points in $X$ must also be Du Val.

    \medskip

    Lastly, as $K^2$ decreases by one after each blowup $\pi_i,$ $K_Y^2=K_W^2-(\ell+1)=C^2-(\ell+1)=-(\ell+1).$
    Since $Y$ is the blowup of a smooth toric surface in a point, we find the topological Euler characteristic $e(Y)=n+1,$ where $n$ is the number of toric boundary divisors in $Y.$ So Noether's formula shows
    $$1=\chi(Y)=\frac{K_Y^2+e(Y)}{12} \implies K_Y^2=11-n.$$
    Hence we find $\ell=n-12.$ 
    As argued above, $\ell$ is also the number of toric boundary divisors contracted. 
    Hence, out of the $n$ toric boundary divisors of $Y,$ there are $12$ that remain in $W,$ out of which $12-4=8$ are exceptional in $W \to Z.$
\end{proof}

\begin{remark}
    In what follows, we will abuse notation for curves disjoint from the contracted locus. That is, if $\mc C$ is a curve in $\mc Y$ and $\pi:\mc Y \to \mc X$ is the contraction of a curve $\mc E \subset \mc Y$ with $\mc E \cap \mc C = \phi,$ we use the same label $\mc C$ for $\pi_\ast \mc C.$
\end{remark}

\begin{definition}
\label{defi: relevant locus}
    Define the \defi{relevant locus} $\mc D$ in $Y$ to be the curve $\wt R$ along with the exceptional locus of $Y \to X$ over the torus invariant points in $X$ that $R$ passes through. 
    In the series of contraction $Y \overset{\pi_0}{\to} Y_1\overset{\pi_1}{\to} \dots \to W,$
    we will abuse notation and use ``relevant locus in $Y_i$'' to mean $(\pi_{i-1} \circ \dots \circ \pi_0)_\ast \mc D,$ i.e. the image of the relevant locus in $Y_i.$ 
\end{definition}

Combining all our results, we finally obtain the following:

\begin{corollary}
$Y$ is a semi-elliptic surface (see Definition \ref{defi: semi-elliptic}).
\end{corollary}
\begin{proof}
Follows from Proposition \ref{prop: negative surface}, Lemmas \ref{lemma:tilde r is -1} and Lemma \ref{lemma: key commutative diagram X,Y,Z,W, and 8 curves}.
\end{proof}
    Reiterating, this means $\wt R$ intersects exactly two exceptional curves of $Y \to X.$
    Further, $Z$ has only Du Val singularities, so that the series of contractions $Y \to \dots \to W$ of $(-1)$ curves leads to a surface with exactly eight $(-2)$ curves in the exceptional locus of $W \to Z.$ See Example \ref{example: semi-elliptic} for an instance of this.
    Observe that this is a purely combinatorial condition on the self-intersection numbers of the toric boundary curve of $Y,$ and the curves $\wt R$ intersects.

\begin{remark}
\label{remark: semi-elliptic surface}
    We will often abuse notation and call a lattice polygon $\Delta$ a \defi{semi-elliptic polygon} if the minimal resolution $Y$ of the blown up toric surface $\Bl_e(\PP_\Delta)$ is a semi-elliptic surface.
\end{remark}

\subsection{Case Division} 
\label{subsec: case division}
In the rest of our paper, we look at possibilities for the Du Val singularities of $Z.$ Our goal is to show that no surface $X$ exists, which we do by analyzing each case separately.

\begin{definition}
\label{defi: opposite and adjacent points}
Call a pair of torus invariant points in $X$ \defi{opposite points} if they don't lie on any common toric boundary curves. 
On a similar note, call a pair of torus invariant points (of $X)$ \defi{adjacent points} if they lie on the same toric boundary curve.
\end{definition}

\begin{convention*}
    For simplicity, we will refer to a smooth point also as a Du Val singularity; the $A_0$ singularity.
\end{convention*}

The first case is when there are two opposite Du Val singularities in $X,$ which we naturally call the \defi{Opposite Du Val case}. If this doesn't hold, we claim that $R$ passes through exactly two adjacent points in $X$ (a priori, it passes through at most two torus invariant points by Lemma \ref{lemma:tilde r is -1}).
Indeed, the torus invariant points in $X$ that $R$ does not pass through are Du Val singularities by Lemma \ref{lemma: key commutative diagram X,Y,Z,W, and 8 curves}, and hence if $R$ passes through less than $2$ torus invariant points or through $2$ opposite point, we find two opposite Du Val singularities (because there are $4$ torus invariant points).

\medskip

Now assume there aren't two opposite Du Val singularities. By the analysis above, we conclude that $R$ passes through exactly $2$ adjacent torus invariant points. 
So we (naturally) call this the \defi{Adjacent Du Val Case}.
We further subdivide cases based on the singularity type of the point $\wt \pi_\ast R \in Z:$ smooth, Du Val of type $DE,$ or Du Val of type $A_n.$

\begin{figure}[h]
    \centering
    \includegraphics[scale=0.2]{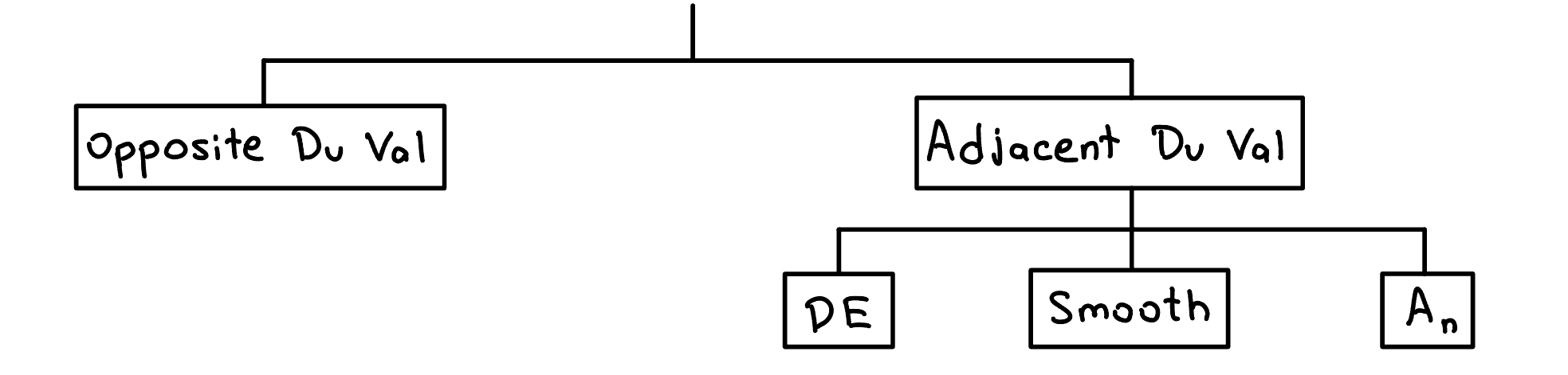}
    \caption{The tree of cases}
    \label{fig:The tree of cases}
\end{figure}

\section{Opposite Du Val case}
The approach in this section is purely combinatorial, establishing a lot of results which will be essential for our future analysis as well. 
The toric fan of a projective toric surface is determined up to $\GL_2(\ZZ).$ 
However, if we choose the right fan in this equivalence class, the condition $\Width(\Delta) \ge m$ is easier to interpret (like we did in Proposition \ref{prop: negative surface}).
Establishing structure on the fan using this condition is the underlying theme in this section, leading to the key Lemma \ref{lemma: special fan}.

\begin{definition}
\label{def: sign}
    Consider the toric fan $\Ti$ of a projective toric surface $\PP_\Delta.$ Consider three consecutive vectors $u,v,w \in \Ti.$ We define the \defi{sign} of the vector $v$ as the sign of the (possibly rational) self-intersection number of the correspond toric boundary divisor in $X.$
Equivalently,
$$\text{sign of }v \text{ is }\begin{cases}
    \text{positive} & \text{if }\measuredangle (u,w) > \pi \\ 
    \text{negative} & \text{if }\measuredangle (u,w) < \pi \\
    0 & \text{if }\measuredangle (u,w) = \pi
    \end{cases}.$$
\end{definition}

\begin{remark}
    Observe that sign is $\GL_2(\ZZ)$ invariant (even though $\measuredangle (u,w)$ is not).
\end{remark}

\begin{lemma}
\label{lem:signature}
    Consider a complete fan (where, hence, the angle between consecutive vectors is at most $\pi)$ consisting of $4$ vectors, none of whose sign is $0.$ If two vectors $u,v$ are opposite (i.e. not consecutive), they have the opposite sign. Consequently, the signs of the vectors form the ordered list $(+,+,-,-).$ 
\end{lemma}
\begin{proof}
    Recalling that sign is $\GL_2(\ZZ)$ invariant, assume the vectors are $e_1,\dots,e_4$ with $e_1=(1,0).$ The opposite vector to $e_1$ then is $e_3,$ say $(x,y)$ with $y>0$ $(y \ne 0$ as the sign of $e_2$ is not $0$ by hypothesis, see Figure \ref{fig:signature}). Then $\measuredangle (e_4,e_1) < \pi$ shows $u$ lies under the $y$-axis. Then $e_2, e_4$ indeed have opposite signs, namely $-, +$ respectively. 
\end{proof}

\begin{figure}[h]
    \centering
    \includegraphics[scale=0.3]{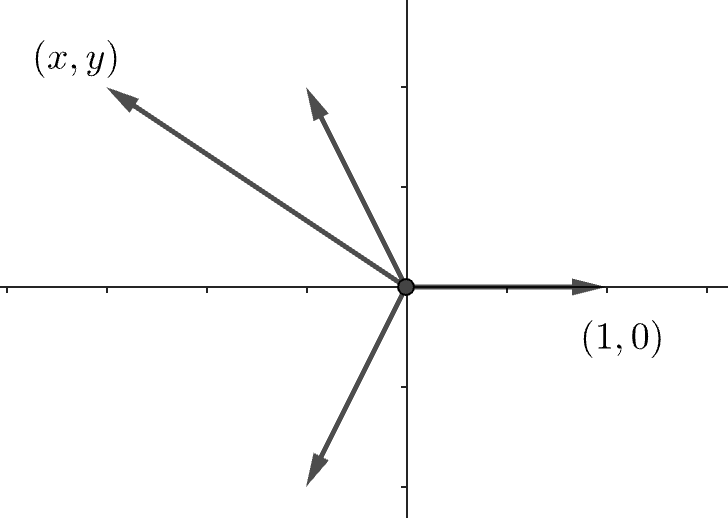}
    \caption{The (unordered) signature is $(+,+,-,-)$}
    \label{fig:signature}
\end{figure}

\begin{definition}
\label{def: Du Val point}
In what follows, if we say a point $p$ in $\Delta$ is \defi{Du Val}, we mean the point corresponding to $p$ in the toric surface $\PP_\Delta$ is either a smooth point or a Du Val singularity. 
A smooth point can be thought of as an $A_0$ singularity.
\end{definition}

\begin{lemma}
\label{lemma: not between +}
    Let $\Delta$ be an elliptic quadrilateral with toric surface $\PP_\Delta.$ If $\ell_1,\ell_2$ are two boundary divisors of $\PP_\Delta$ such that $\ell_1 \cap \ell_2$ is \textit{Du Val}, then at least one of $\ell_1^2, \ell_2^2$ is negative.
\end{lemma}
\begin{proof}
    Assume on the contrary that both $\ell_1^2,\ell_2^2$ are non-negative. 
    Consider the fan $\Ti$ of $\PP_\Delta$ with vectors $v_0,v_1,v_2,v_3$ (counterclockwise) where $v_i$ corresponds to $\ell_i$ for $i=1,2.$ Let $V_i$ be the side in $\Delta$ corresponding to $v_i.$ Let $p=V_1 \cap V_2 \in \Delta$ correspond to the Du Val singularity.

    \begin{figure}[h]
        \centering
        \includegraphics[scale=0.3]{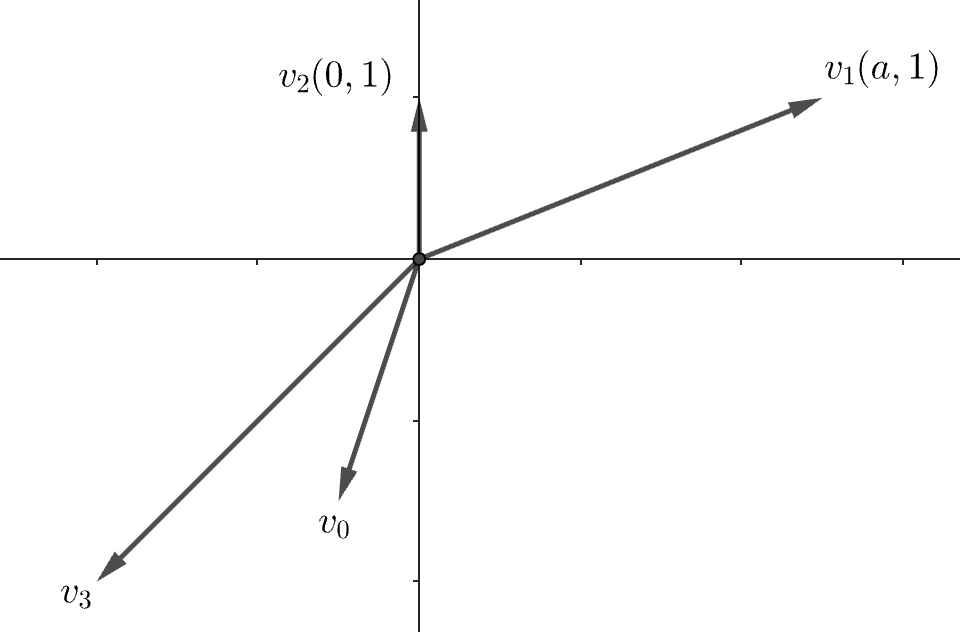}
        \includegraphics[scale=0.3]{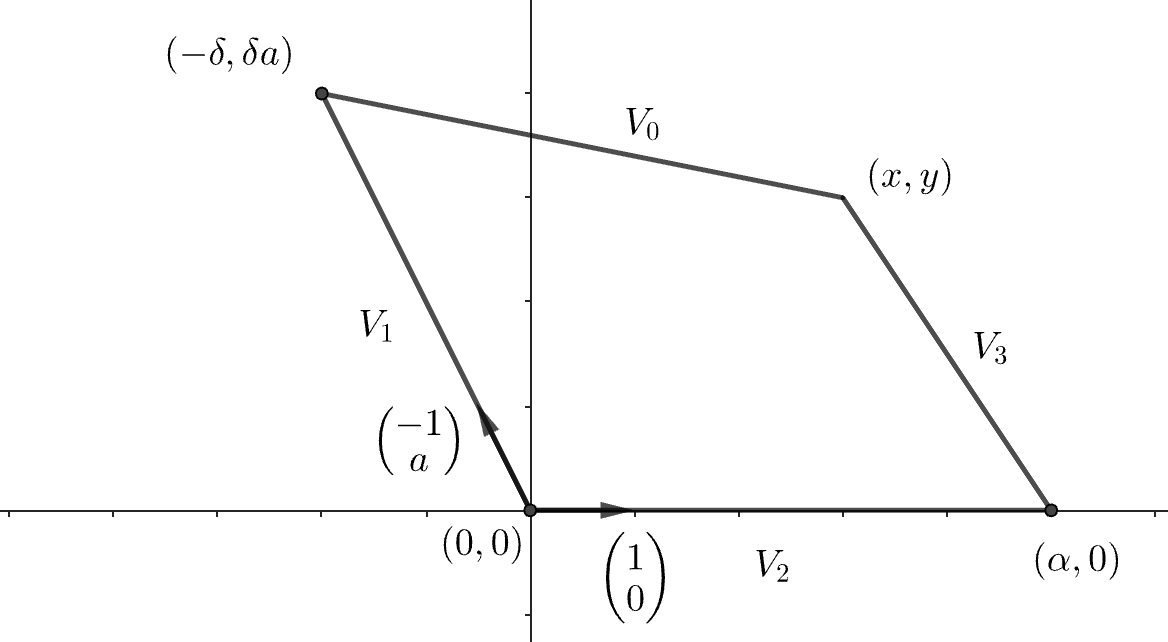}
        \caption{Toric $N$ and $M$ fan.}
        \label{fig:two positive}
    \end{figure}

    Up to $\GL_2(\ZZ),$ we can assume $v_1$ is $(a,1)$ and $v_2$ is $(0,1)$ (c.f. \ref{example: du val matrix}).
    As $v_1, v_2$ are positive or $0$ (cf. Def. \ref{def: sign}), hence $v_0, v_3$ both lie in the third quadrant. 
    Further, using inner normals, $V_1$ is along $(-1,a)$ and $V_2$ along $(1,0).$
    Let $\delta,\alpha$ be the lattice lengths of $V_1,V_2.$ As $v_0,v_3$ are in the third quadrant, the width along $(1,0)$ is (see Figure \ref{fig:two positive})
     $$m \le \Width(\Delta) \le \Width_{(1,0)}(\Delta) = \alpha+\delta < m,$$
     a contradiction.
\end{proof}

The proof of Lemma \ref{lemma: not between +} shows something much stronger about the fan, something we discuss in the next lemma:

\begin{lemma}
\label{lemma: special fan}
    Suppose that $v_0, v_1, v_2, v_3$ form the vectors of the toric fan of $\PP_\Delta,$ where $\Delta$ is an elliptic quadrilateral. Suppose $v_0, v_1$ are positive vectors, and (the point corresponding to) $v_1 \cap v_2$ is Du Val of type $A_{n-1}$. 
    By a $\GL_2(\ZZ)$ transform, say $v_0=(w,-z)$ and $v_1=(0,1),$ where $w>z>0$ (see Figure \ref{fig:special choice}). Then $v_2=(-n,1)$ and $v_3$ lies over the $x$-axis.

    \begin{figure}[h]
        \centering
        \includegraphics[scale=0.4]{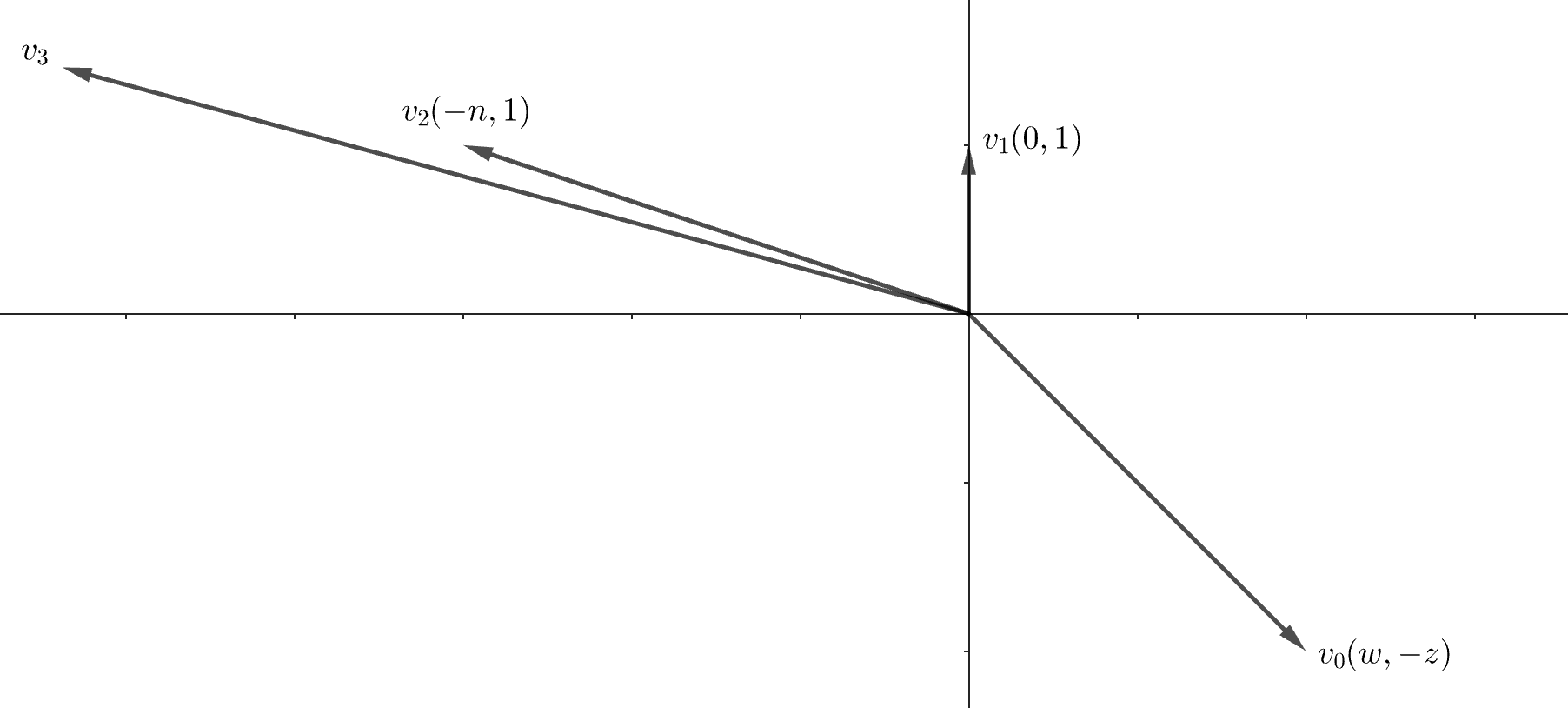}
        \caption{Good choice of fan up to $\GL_2(\ZZ).$}
        \label{fig:special choice}
    \end{figure}
\end{lemma}
\begin{proof}
    Based on the singularity type (or, say, an appropriate $\GL_2(\ZZ)$ transform), we can pick the fan so that $v_0=(w,-z)$ and $v_1=(0,1)$ so that $w>z>0.$ Since $v_0$ is positive, hence $\la v_3, (1,0) \ra <0.$\footnote{In this paper, $\la \bullet, \bullet \ra$ is the standard Euclidean dot product.} 

    \medskip
    
    As (the point corresponding to) $v_1 \cap v_2$ is a Du Val singularity of type $A_{n-1}$, hence $v_2=(-n, 1+kn)$ for some $k \in \ZZ$ (cf. end of \ref{example: du val matrix}). Note here that $n=1$ corresponds to a smooth point, consistent with our convention in Definition \ref{def: Du Val point}.

    \medskip

    Firstly, we show $\la v_2, v_1 \ra > 0.$ Otherwise, the fan and quadrilateral look like what we have in Figure \ref{fig:v2y}. 
    In the corresponding quadrilateral, we thus find $$\Width \le \Width_{(1,0)} \le |V_1 \cap \ZZ^2|<m,$$ a contradiction. Hence $\la v_2,v_1 \ra = 1+kn>0.$ 

     \begin{figure}[h]
        \centering
        \includegraphics[scale=0.3]{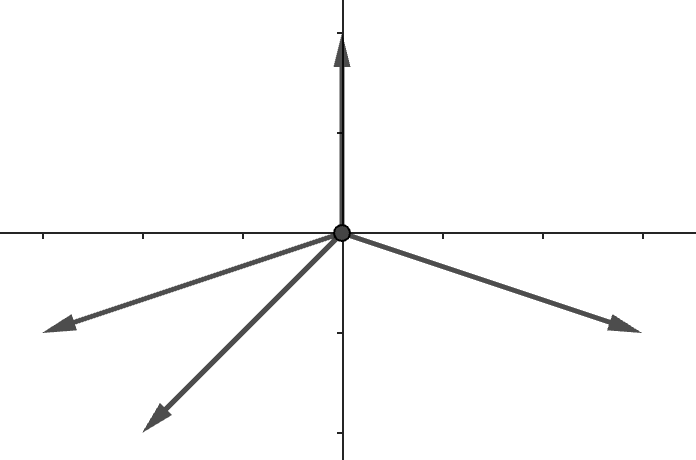}
        \includegraphics[scale=0.3]{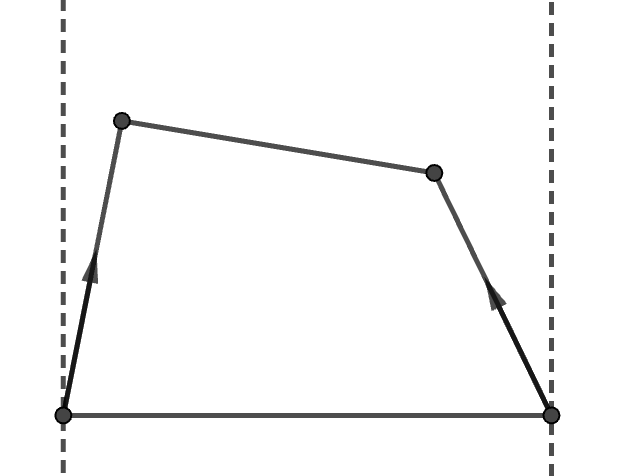}
        \caption{The case when $\la v_2,v_1 \ra \le 0.$}
        \label{fig:v2y}
    \end{figure}
    
    Next, since $v_1$ is positive, the vector $-v_0$ must lie between $v_1,v_2.$ Hence $-v_0$ has a higher slope than $v_1,$ showing
    $$0<\frac{1+kn}{n}<\frac{z}{w}<1.$$
    Since $k \in \ZZ$ and $n \ge 1,$ we must have $k=0,$ as desired.

    \medskip

    Lastly, it suffices to show $\la v_3 , v_1 \ra>0.$ However, this follows from a similar argument as the one used Lemma \ref{lemma: not between +}; say $\la v_3,v_1 \ra \le 0.$ Consider the corresponding quadrilateral with side $V_i$ corresponding to the vector $v_i$. As $v_2$ is $(-n,1),$ hence $\la V_2, (1,0) \ra = |V_2 \cap \ZZ^2|.$ Thus 
    $$\Width \le \Width_{(1,0)} \le |V_1 \cap \ZZ^2|+\la V_2, (1,0) \ra = |V_1 \cap \ZZ^2|+|V_2 \cap \ZZ^2| < m,$$
    a contradiction (see Figure \ref{fig:key-final}).
\end{proof}

\begin{figure}[h]
        \centering
        \includegraphics[scale=0.3]{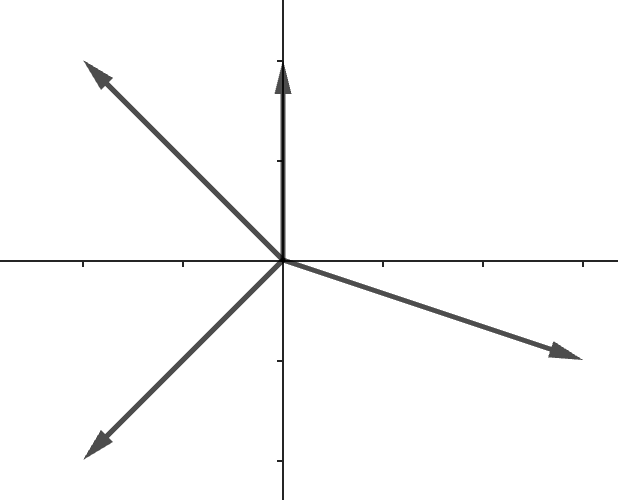}
        \includegraphics[scale=0.3]{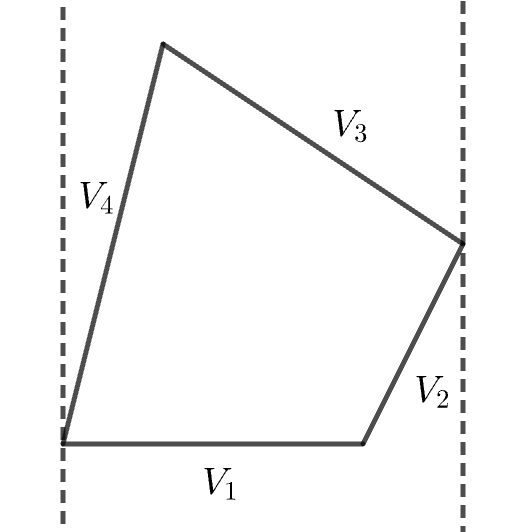}
        \caption{The width is less than $m$ along $(1,0).$}
        \label{fig:key-final}
    \end{figure}

How do the signs of vectors here relate to the signs of their strict transforms in the resolution? We can analyze that by subdividing our cones:

\begin{lemma}
\label{lemma: self-intersections in resolution}
    Let $\Ti$ be the toric fan of $\PP_\Delta.$ If $u,v,w \in \Ti$ are consecutive (counterclockwise) such that $u,v$ are positive and $\measuredangle (v,w)$ is Du Val\footnote{meaning the point $V \cap W$ is Du Val, where $V, W$ are sides in $\Delta$ corresponding to $v,w$}, then the strict transforms of the divisor corresponding to $v$ has self-intersection $(-1),$ while that of $w$ is $\le (-2).$
\end{lemma}
\begin{figure}[h]
        \centering
        \includegraphics[scale=0.4]{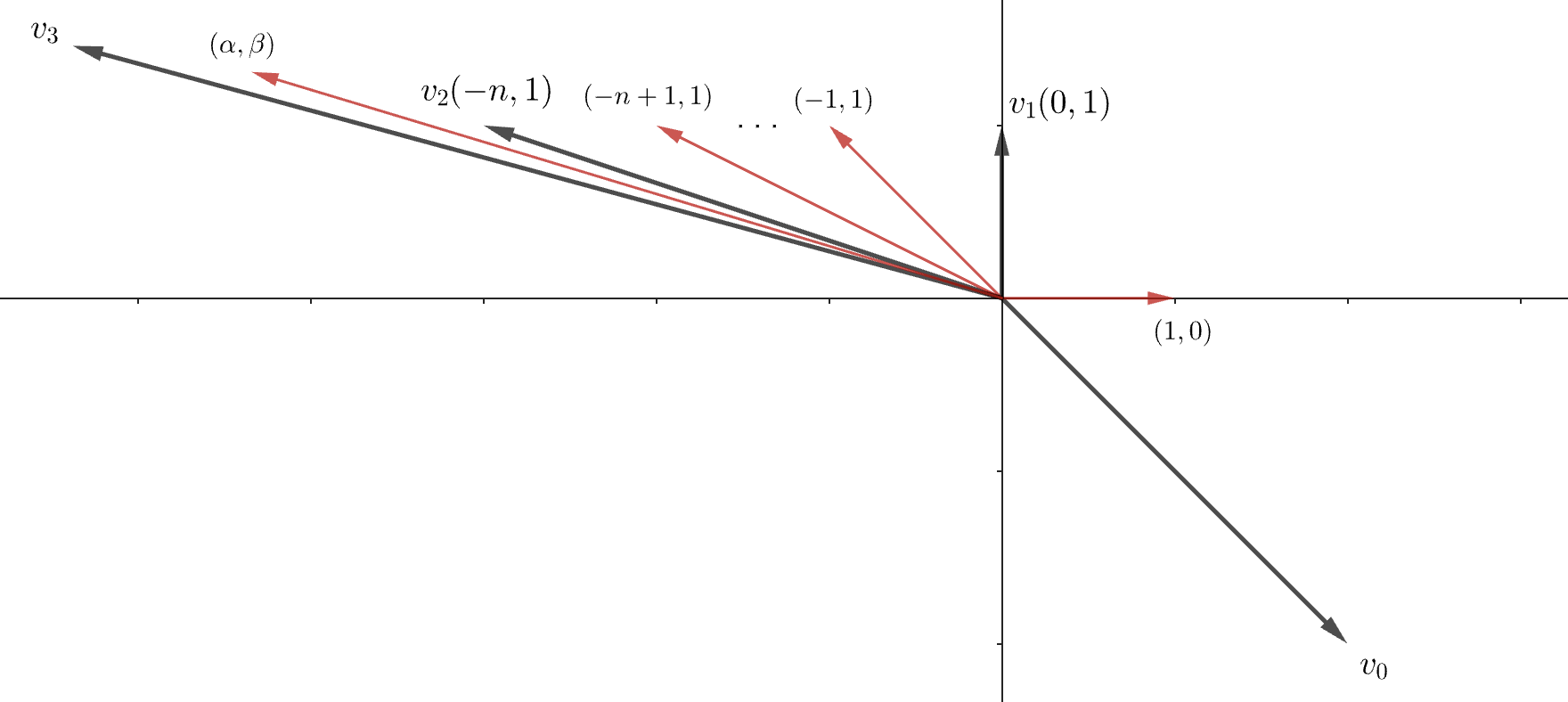}
        \caption{The resolution of our fan $\Ti$}
        \label{fig:resolution}
    \end{figure}
\begin{proof}
    Choose a fan as in \ref{lemma: special fan}. It suffices to prove that the strict transforms of the vectors $v_1$ will have self-intersection $(-1),$ while that of $v_2$ will have self-intersection at most $(-2).$ 
    We know that the resolution of a convex cone is unique. Hence running the cone resolution algorithm (cf. \cite{fulton}), the last vector (counter-clockwise) in the cone $(v_0,v_1)$ is $(1,0).$ The resolution of $(v_1,v_2)$ will be the $(-2)$ vectors $\{(-k,1):k=0, \dots, n\}.$ Hence, $(0,1)=(-1,1)+(1,0)$ shows the the self-intersection of the strict transform of $v_1$ will be $(-1)$ (see Figure \ref{fig:resolution}).

    \medskip

    Next, we consider the first vector (counterclockwise) in the resolution of the cone $(v_2,v_3).$ Since both $v_2,v_3$ lie above the $x$ axis, this vector will be $(-\alpha, \beta)$ for some $\alpha,\beta \in \ZZ,$ with $\beta>0.$ Hence the self-intersection $\tau$ of $v_2=(-n,1)$ satisfies $$-\tau(-n,1)=(-n+1,1)+(\alpha,\beta)=(\bullet, 1+\beta),$$ showing $-\tau = 1+\beta > 1$ since $\beta>0.$ This shows the second part
\end{proof}

\begin{remark}
\label{remark: never used opposite du val}
    While the section is titled the ``Opposite Du Val case'', notice that we never used this hypothesis in either of Lemma \ref{lemma: special fan} or Lemma \ref{lemma: self-intersections in resolution}.
\end{remark}

Using the results we have established, we can rule out the opposite Du Val case completely. 
For this, we have to use the symmetry in Lemma \ref{lemma: self-intersections in resolution}:

\begin{proposition}
\label{prop: opposite duval}
    Let $\Delta$ be an elliptic quadrilateral. Then $X=\Bl_e(\PP_\Delta)$ can't have opposite Du Val singularities. 
\end{proposition}
\begin{proof}
    Assume on the contrary, and consider the fan $\Ti$ of $\PP_\Delta.$ Lemma \ref{lem:signature} shows there is one consecutive pair of positive vectors, and \ref{lemma: not between +} shows that these two vectors cannot correspond to a Du Val singularity. Hence the Du Val singularities are between the opposite $(+,-)$ pairs. 
    Let $v_0,\dots,v_3$ be the vectors so that $v_0,v_1$ are positive and $\measuredangle (v_1,v_2)$ and $\measuredangle (v_3,v_0)$ are both Du Val. 
    By Lemma \ref{lemma: self-intersections in resolution}, the strict transforms of both $v_2, v_3$ will have self-intersection at most $(-2).$ 
    However, Lemma \ref{lem: at least three (-1) curves} shows that we must have at least three $(-1)$ curves in the non-exceptional (toric boundary) locus of $Y \to X$, giving the desired contradiction.
\end{proof}

\begin{remark}
\label{remark: smooth = Du Val}
    It is important to note that the proof also works for an ``$A_0$ singularity'', i.e. a smooth point. 
    Also, we never assumed $K_X+C \sim \alpha R$ in this proof, so this is stronger than what we need.
\end{remark}    

\begin{corollary}
\label{cor: R through exactly 2 points}
Let $(X,C)$ be a (hypothetical) toric elliptic pair with $\rho(X)=3$ and $K+C \sim \alpha R.$ There are exactly two torus invariant points in $X$ that $R$ doesn't pass through. 
\end{corollary}
\begin{proof}
   By Lemma \ref{lemma:tilde r is -1}, $R$ passes through at most $2$ torus invariant points in $X.$ If it passes through at most one, then there are at least two opposite points (cf. Def. \ref{defi: opposite and adjacent points}) in $X$ that $R$ doesn't pass through.
   By Lemma \ref{lemma: key commutative diagram X,Y,Z,W, and 8 curves}, these points are Du Val singularities (possibly smooth), and hence $X$ has opposite Du Val singularities.
   However, Proposition \ref{prop: opposite duval} shows this is impossible.
\end{proof}

\section{Adjacent Du Val case}
\label{sec: adjacent du val case}
In this section, we first discuss what combinatorial configurations are possible for the relevant locus of $Y$ in the Adjacent Du Val case. This will be essential for using the Matrix Equation \ref{lem: matrix equation} in future sections. Next, we look at what more can be said about these configurations in the $A_n$ case (cf. Fig. \ref{fig:The tree of cases}), since that turns out to be hardest case.

\subsection{A General Analysis on the Combinatorial Configuration}

We now prove a few combinatorial results on the configuration of $Y$ and the curves $\wt R$ intersects. 

\begin{definition}
\label{def: endcurve, endchain}
    An \defi{endcurve} of a chain of curves $\mc C_1,\dots, \mc C_n$ is one of the two curves $\mc C_1, \mc C_n.$
    Consider the two toric boundary chains in the relevant locus (cf. Cor. \ref{cor: R through exactly 2 points}). We call any of these chains an \defi{endchain} if $\wt R$ intersects an endcurve of the chain. 
\end{definition}

By Corollary \ref{cor: R through exactly 2 points}, the relevant locus consists of $\wt R$ and two disjoint chains on the toric boundary of $Y.$ 

\begin{lemma}
\label{lemma: key lemma about endchains}
Consider a semi-elliptic toric surface $Y$ (cf. Def. \ref{remark: semi-elliptic surface}).
\begin{enumerate}
    \item At least one of the toric boundary chains (in the relevant locus) that $\wt R$ intersects is an endchain.
    \item If we assume the singularities of $Z$ are of type $A_n$ (cf. Fig. \ref{fig:The tree of cases}), then both the chains (in the relevant locus) that $\wt R$ passes through must be endchains.
\end{enumerate}
\end{lemma}
\begin{proof}
To prove this, we first prove an auxiliary lemma.

\begin{lemma}
\label{lemma: auxilliary lemma about high multiplicity intersections}
    At any step during the series of contractions $Y \to Y_1 \to \dots \to W,$ the relevant locus (cf. Def. \ref{defi: relevant locus})
    \begin{enumerate}
        \item is a simple normal crossing divisor. In particular, no three curves pass through the same point or intersect with multiplicity $>1;$
        \item has no more than one $(-1)$ curve.
    \end{enumerate}
\end{lemma}
\begin{proof}
We know that the contractions $Y \to \dots \to W$ always contract a $(-1)$ curve, and the dual graph of the exceptional locus of $W \to Z$ is a forest (since all singularities of $Z$ are Du Val, cf. Lem. \ref{lemma: key commutative diagram X,Y,Z,W, and 8 curves}).
Hence, the relevant locus must always be a simple normal crossing divisor. 
For $(2),$ observe that we have exactly one $(-1)$ curve in the relevant locus in $Y$ initially (namely $\wt R),$ so consider the smallest $i$ such that the relevant locus of $Y_{i+1}$ has two $(-1)$ curves. Let $\mc C \subset Y_i$ be contracted in $\pi_i:Y_i \to Y_{i+1},$ so that it intersects at most two relevant curves in $Y_i$ by $(1).$
The number of $(-1)$ curves increases if and only if the two relevant curves $\mc C$ intersected both become $(-1)$ curves in $Y_{i+1}.$ 
However, since they are the only $(-1)$ curves in $Y_{i+1},$ one of them must be contracted to get $Y_{i+2}.$ But the other curve then becomes a curve with self-intersection $0$ in $Y_{i+2},$ 
which is a contradiction since the exceptional locus of $Y_{i+2} \to Z$ can have only negative curves.
\end{proof}

We can now prove the first part of Lemma \ref{lemma: key lemma about endchains}.

\begin{enumerate}
    \item To see that at least one of the chains is an endchain, assume not, and consider curves $\mc C, \mc C'$ that $\wt R$ intersects. By assumption, neither of them is an endcurve.
    Consider the contraction of $\wt R,$ so that $(\pi_0)_\ast \mc C, (\pi_0)_\ast \mc C'$ both intersect three other curves. 
    We claim $\ell>0,$ i.e. $Y_1 \ne Y_{\ell+1} = W.$
    Indeed, the dual graph of the relevant locus in $Y_1$ has two nodes with degree $3,$ which means the locus is not of ADE type. However, $W \to Z$ is the resolution of Du Val singularities.  
    \medskip

    Since $\mc C, \mc C'$ are the only curves whose self-intersection changed in the contraction, one of them maps to a $(-1)$ curve after this contraction.
    Contracting either of these curves introduces three curves intersecting in the same point. However, then the relevant locus in $Y_1$ is not simple normal crossing, contradicting Lemma \ref{lemma: auxilliary lemma about high multiplicity intersections}.
\end{enumerate}

To prove the second part of Lemma \ref{lemma: key lemma about endchains}, we need an easy lemma:

\begin{lemma}
\label{lem: endchain is -2}
    Suppose $Z$ has only $A_n$ type singularities. If there is only one endchain in $Y,$ it must be a chain of $(-2)$ curves.
\end{lemma}
\begin{proof}
    Assume on the contrary, and suppose $\wt R$ intersects the curve $\mc C$ in the chain which is not the endchain, and that $\mc D, \mc D'$ intersect $\mc C$ in that chain.
    By Lemma \ref{lemma: auxilliary lemma about high multiplicity intersections}, there is always exactly one $(-1)$ curve in the series of contractions $Y \to W.$ Furthermore, since our contractions end at a chain of curves (since $Z$ has $A_n$ type singularities), as long as $\mc C$ intersects at least three curves, this cannot be in the exceptional locus of $W.$

    \medskip

    The self-intersections of $\mc D, \mc D'$ don't change until $\mc C$ is contracted, so consider the first time that happens. 
    If the endchain hasn't contracted to a smooth point, then $\mc C$ intersects at least three curves, and so contracting it will give three curves intersecting in a point. 
    However, then the divisor is not simple normal crossing, contradicting Lemma \ref{lemma: auxilliary lemma about high multiplicity intersections}. Hence the endchain contracts completely before $\mc C$ is contracted.

    \medskip
    
    Let the endchain in $Y$ be $\mc C_1,\dots, \mc C_n,$ so that $\wt R$ intersects $\mc C_1.$ We inductively show that $\mc C_i^2=-2$ for every $i.$ Indeed, after contracting $\wt R,$ the curve $\mc C_1$ maps to a $(-1)$ curve, hence $\mc C_1^2=-2.$ Assuming the result for some $\mc C_k,$ observe that once we eventually contract $\mc C_k,$ the only curves whose self-intersection numbers change are $\mc C, \mc C_{k+1},$ and so $\mc C_{k+1}$ maps to a $(-1)$ curve. 
    Further, this is the first time in the series of contractions that its self-intersection number of $\mc C_{k+1}$ changes. Hence, $\mc C_{k+1}^2=-2,$ and the induction is complete.
\end{proof}

Now we can finish the proof of $(2)$ in Lemma \ref{lemma: key lemma about endchains}. 
Assume on the contrary that both chains aren't endchains. By Lemma \ref{lem: endchain is -2}, there is one endchain, and it is a chain of $(-2)$ curves, i.e. the exceptional locus of a Du Val singularity. However since we already have two other Du Val singularities, this shows that $X$ has at least three Du Val singularities. Two of them must be opposite, contradicting Proposition \ref{prop: opposite duval}. 
\end{proof}

We conclude this subsection by a definition.

\begin{definition}
\label{def: heavy curve}
Consider the four toric boundary curves $\wt \gamma_0,\wt \gamma_1,\wt \gamma_2,\wt \gamma_3$ in $Y$ that are strict transforms of the four toric boundary curves in $X.$ By Lemma \ref{lem: at least three (-1) curves}, (at least) three of these have self-intersection $(-1).$ 
Call the fourth curve the \defi{heavy curve} in $Y.$
\end{definition}

Let $p,p' \in X$ be the two torus invariant points that $R$ doesn't pass through (cf. Cor. \ref{cor: R through exactly 2 points}), so that $p,p'$ are Du Val singularities (cf. Lem. \ref{lemma: key commutative diagram X,Y,Z,W, and 8 curves}). 
Say $\gamma_0=\sigma_\ast \wt \gamma_0$ is the toric boundary curve in $X$ passing through $p,p'.$
By Lemma \ref{lemma: self-intersections in resolution} and the discussion above, $(\wt \gamma_i)^2 =(-1)$ for $i \ne 0,$ and $(\wt \gamma_0)^2 < (-1).$ Hence the heavy curve is $\wt \gamma_0,$ i.e. it is the toric boundary curve in $X$ through those torus invariant points of $X$ that don't lie on $R.$

\subsection{The $A_n$ case subdivision}

Consider the relevant locus, and let $\mc S, \mc S' \subset Y$ be the two disjoint chains in it, lying over the adjacent torus invariant points $p, p' \in X$ respectively (cf. Def. \ref{defi: opposite and adjacent points}). Let $\mc D \subset Y$ be the strict transform under $Y \to X$ of the toric boundary divisor in $X$ through $p,p'.$

\medskip

As $\mc D$ is not the heavy curve (cf. Def. \ref{def: heavy curve}), it is a $(-1)$ curve.
In the $A_n$ case (cf. Fig. \ref{fig:The tree of cases}), Lemma \ref{lemma: key lemma about endchains} shows that
$\mc S, \mc S'$ are both endchains, so that
we only have the $3$ possibilities shown in Figure \ref{fig:3 cases};
namely, let $\mc C \subset \mc S, \mc C' \subset \mc S'$ be the curves that $\wt R$ intersects. 
Then these cases correspond to whether both, one, or none of the curves $\mc C, \mc C'$ intersect $\mc D.$ 

\begin{figure}[h]
    \centering
    \includegraphics[scale=0.2]{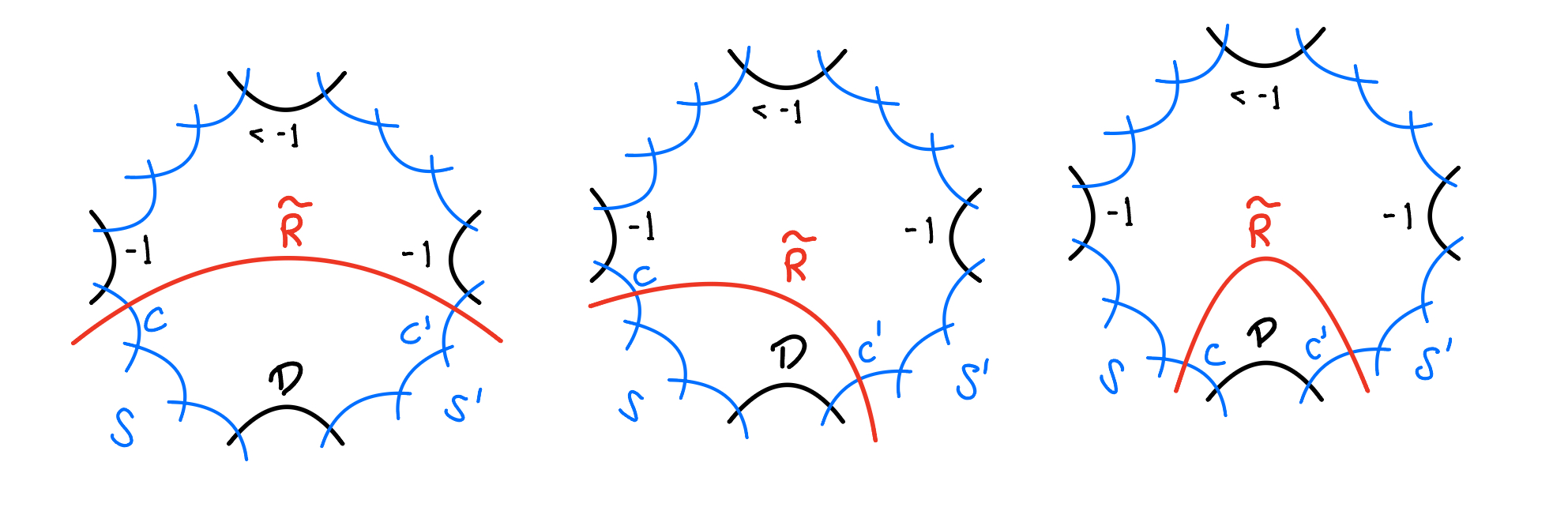}
    \caption{Cases I,II,III respectively (from left to right)}
    \label{fig:3 cases}
\end{figure}

The first two cases will need more work, and are the bulk of the sections that follow. 
The last case, however, is easy to rule out:

\begin{proposition}
    \label{prop: first case not possible}
    Case III in Figure \ref{fig:3 cases} is impossible.
\end{proposition}
\begin{proof}
    We know that contracting $\wt R$ and subsequent contractions contracts the locus $\mc S \cup \wt R \cup \mc S'$ to the resolution of an $A_n$ type singularity. 
    Combinatorially, this process of contractions works the same on the chain $\mc S \cup \mc D \cup \mc S'$ since $\mc D^2=\wt R^2=-1.$ This means the chain $\mc S \cup \mc D \cup \mc S'$ in the toric boundary contracts down to a chain of $(-2)$ curves. 

    \medskip

    Let $\Ti$ be the toric fan of $X$ with vectors $v_0,v_1,v_2,v_3$ such that $\measuredangle (v_3,v_0), \measuredangle (v_0,v_1)$ correspond to the chains $\mc S, \mc S'$ (notation from Figure \ref{fig:special choice}). The above implies that $\measuredangle (v_3,v_1)$ is the (resolution of the) cone of a Du Val singularity, in particular convex. However, $v_2$ corresponds to the heavy curve, i.e. a curve with self-intersection $<(-1)$ in $Y$ (see Definition \ref{def: heavy curve}).
    Hence $\measuredangle (v_1,v_3)$ is convex too, implying $$2\pi = \measuredangle (v_3,v_1)+\measuredangle (v_1,v_3)<\pi + \pi=2\pi$$
    which is the desired contradiction.
\end{proof}

To specify orientations of certain toric boundary chains, we will use the following notation in the future:

\begin{definition}
\label{def: inner and outer curves}
    Consider the chains $\mc S, \mc S',$ and let $\mc E_{\text{in}},\mc E_{\text{out}} \subset \mc S$ and $\mc E'_{\text{in}}, \mc E'_{\text{out}} \subset \mc S'$ be the $4$ endcurves in these chains such that $\mc E_{\text{in}} \cap \mc D \ne \phi$ and $\mc E_{\text{in}}' \cap \mc D \ne \phi.$
    We call $\mc E_{\text{in}}, \mc E_{\text{in}}'$ the \defi{inner curves} and $\mc E_{\text{out}}, \mc E_{\text{out}}'$ the \defi{outer curves}. 
\end{definition}
For instance, in Figure \ref{fig:3 cases}, $\mc C, \mc C'$ are the two outer curves in Case I, and they are the two inner curves in Case III.

\section{The $DE$ cases}
\label{sec: the DE case}
The result in this case is very interesting; we prove that there are exactly $6$ two parameter families of semi-elliptic quadrilaterals where $Z$ has a $D_n$ or $E_n$ type singularity.
To do this, we first classify the chains that contract to a $D_n, E_n$ Dynkin diagram, up to a free parameter $\ell.$
We then use the Matrix equation \ref{lem: matrix equation} to show there are only these $6$ solutions.

\subsection{Classification}
We can classify all $DE$ cases simultaneously. The Dynkin diagram of a $D_n, E_n$ graph has a node of degree $3,$ call $v.$ Say the other three chains have $p,q,r$ nodes, so that $p+q+r+1=n.$
Note that up to permutation, in the $D_n$ case, $(p,q,r)=(1,1,n-3),$ and in the $E_n$ case, $(p,q,r)=(1,2,n-4).$ 

\medskip

Let $\mc S, \mc S'$ be the two chains in the relevant locus in $Y$, and $\mc C \subset \mc S, \mc C' \subset \mc S'$ be the two curves that $\wt R$ intersects. 
By $(1)$ of Lemma \ref{lemma: key lemma about endchains}, one of $\mc C, \mc C'$ is an endcurve (cf. Def. \ref{def: endcurve, endchain}), say $\mc C.$ 
Then up to a free parameter, we can classify the relevant locus completely:

\begin{lemma}
\label{lemma: DE classfication}
There exists some $\ell \in \ZZ_{\ge 0}$ so that $\mc C^2=-2$ and $(\mc C')^2 = -(\ell+3).$ Further, $\mc S$ has $(\ell+r)$ curves, whose self-intersections are all $(-2),$ except the $(\ell+1)^{\text{th}}$ curve (starting from $\mc C$, see Figure \ref{fig:DE classfication}) which has self-intersection $(-3)$.
Further, $\mc C'$ intersects two chains of $(-2)$ curves of lengths $p$ and $q$ which are in the non-exceptional locus of $\pi:Y \to W.$ 

\begin{figure}[h]
    \centering
    \includegraphics[scale=0.15]{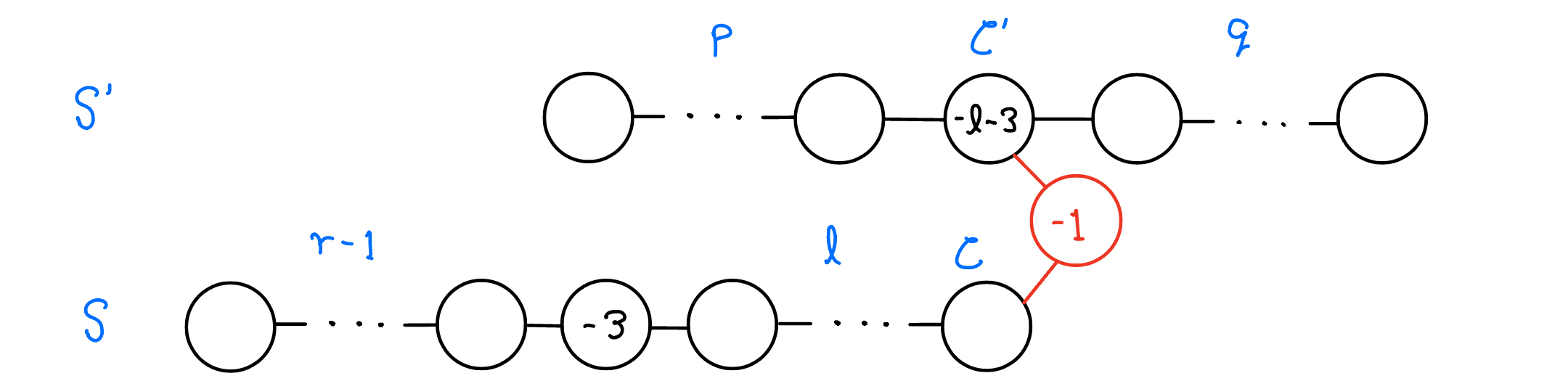}
    \caption{Relevant locus contracting to the resolution of $DE$ type singularity. Any unlabelled node has self-intersection $(-2)$}
    \label{fig:DE classfication}
\end{figure}
\end{lemma}
\begin{proof}
Firstly, observe that $\mc S$ is not fully contracted in $Y \to Y_1 \to \dots W,$ since the image of the relevant locus in $Y \to W$ is not a chain of curves.
Hence $\mc C'$ is never contracted in $Y \to Y_1 \to \dots W,$ else the two neighboring curves of $\mc C'$ will intersect a curve from $\mc S$ at the same point, contradicting the fact that the relevant locus is always normal crossing (cf. Lem. \ref{lemma: auxilliary lemma about high multiplicity intersections}).

\medskip

Let $v$ be the node in the Dynkin diagram of $D_n$ (or $E_n)$ which has degree $3.$
Observe that no blowup can decrease the degree of $v.$
Furthermore, there is only one curve in $\mc S \cup \mc S' \cup \wt R \subset Y$ which intersects more than two curves, namely $\mc C'$ (recall $\mc C$ is an endcurve). As argued above, this is never contracted, hence $\mc C'$ must be the strict transform under $Y \to W$ of the curve corresponding to $v.$ 

\medskip

As $\mc C'$ is never contracted, the two chains of $(-2)$ curves it intersects in $\mc S'$ stay invariant, hence have lengths $p,q$ (since $v$ is adjacent to two $(-2)$ chains of length $p,q)$ 
Since $\mc S \cup \wt R$ contracts to a chain of length of $r,$ it is easy to check the self-intersection numbers are as claimed, where $\ell+1$ is the number of curves contracted in $\mc S \cup \wt R \cup \mc S'$ in $Y \to W.$
\end{proof}

\begin{remark}
As mentioned at the end of the proof, the $\ell$ in Lemma \ref{lemma: DE classfication} is the same as the $\ell$ in Lemma \ref{lemma: key commutative diagram X,Y,Z,W, and 8 curves}, i.e. one less than the number of contractions $Y \to Y_1 \to \dots W.$
\end{remark}

\begin{lemma}
\label{lemma: DE heavy curve}
    The heavy curve (cf. Def. \ref{def: heavy curve}) has self-intersection $-3.$
\end{lemma}
\begin{proof}
Let $s$ be the self-intersection number of the heavy curve.
Consider the toric boundary curves of $Y.$
There are three $(-1)$ toric boundary curves (cf. discussion in \ref{def: heavy curve}), and two chain of $(-2)$ curves of lengths $a,b.$ 
Let $N$ be the number of toric boundary divisors in $Y.$ By Lemma \ref{lemma: sum of smooth numbers}, the sum of self-intersection numbers of toric boundary divisors of $Y$ is $12-3N.$
Hence using Lemma \ref{lemma: DE classfication} and $p+q+r+1=n,$
{\footnotesize
\begin{align*}
    s+3(-1)+(a+b)(-2)+(p+q+(r-&1)+\ell)(-2)-3-(3+\ell)=12-3(4+a+b+p+q+r+1+\ell) \\
    &\implies s = 5-(a+b+n).
\end{align*}
}
By Lemma \ref{lemma: key commutative diagram X,Y,Z,W, and 8 curves}, we find $a+b+n=8,$ and thus we conclude $s=-3.$
\end{proof}

As mentioned in Definition \ref{def: heavy curve}, the heavy curve is the strict transform 
under $Y \to X$ of the toric boundary divisor passing through the torus invariant points that $R$ doesn't pass through.
Combining this observation with our classification in Lemmas \ref{lemma: DE classfication}, \ref{lemma: DE heavy curve} we find two possible orientations for the boundary divisor of $Y$ as in figure \ref{fig:DE schematic}.

\begin{figure}[h]
    \centering
    \includegraphics[scale=0.2]{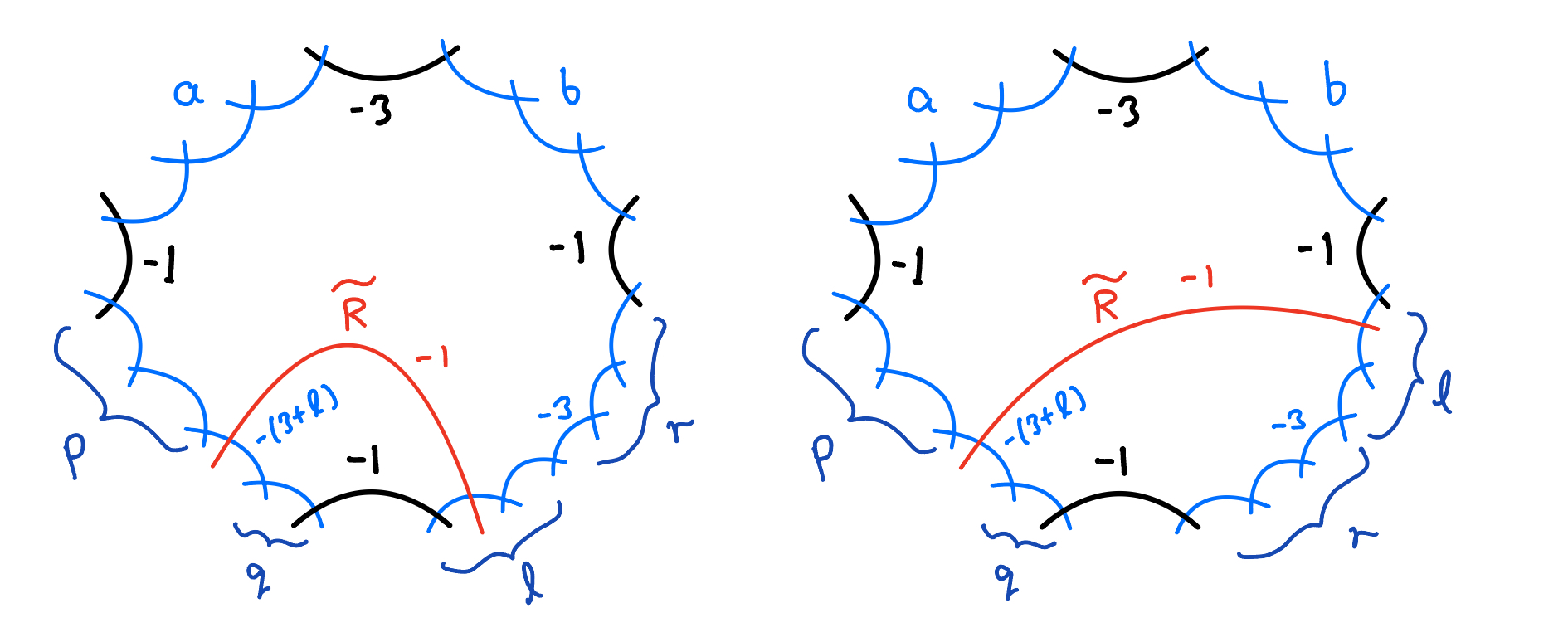}
    \caption{Schematic Picture of $Y$ in the $DE$ case. Any unlabelled curve has self-intersection $(-2).$}
    \label{fig:DE schematic}
\end{figure}

Now $p+q+r=n-1$ and $a+b+n=8$ with $a,b \ge 0$ and $n \ge 0.$ Further, up to permutation, $(p,q,r)=(1,1,n-3)$ in the $D_n$ case and $(p,q,r)=(1,2,n-4)$ in the $E_n$ case.
Hence up to our free parameter $\ell,$ there are a finite number of toric surfaces. 
Quite fascinatingly, \textsc{Macaulay2} can be used to show $\ell \in \{0,1\}$ which gives a finite number of valid toric surfaces.

\begin{proposition}
Let $(X,C)$ be a (hypothetical) toric elliptic pair with $\rho(X)=3$ and $K+C \sim \alpha R.$ Then the DE case in Figure \ref{fig:The tree of cases} is impossible, i.e. $Z$ cannot have $D_n,E_n$ type singularities.
\end{proposition}
\begin{proof}
It is well known that a list of integers form the self-intersection numbers of the toric boundary of some smooth toric surface if and only if they satisfy a matrix equation, see Lemma \ref{lem: matrix equation} in the Appendix for details. Hence, in the first case in Figure \ref{fig:DE schematic}, we find
    \begin{equation}
    \label{eq: DE eq 1}
        \begin{split}
        &\begin{pmatrix}
        0&-1\\
        1&2
    \end{pmatrix}^p\begin{pmatrix}
        0&-1\\
        1&3+\ell
    \end{pmatrix}\begin{pmatrix}
        0&-1\\
        1&2
    \end{pmatrix}^q\begin{pmatrix}
        0&-1\\
        1&1
    \end{pmatrix}\begin{pmatrix}
        0&-1\\
        1&2
    \end{pmatrix}^\ell\begin{pmatrix}
        0&-1\\
        1&3
    \end{pmatrix}\begin{pmatrix}
        0&-1\\
        1&2
    \end{pmatrix}^{r-1} \\
    &\times \begin{pmatrix}
        0&-1\\
        1&1
    \end{pmatrix}\begin{pmatrix}
        0&-1\\
        1&2
    \end{pmatrix}^b\begin{pmatrix}
        0&-1\\
        1&3
    \end{pmatrix}\begin{pmatrix}
        0&-1\\
        1&2
    \end{pmatrix}^a\begin{pmatrix}
        0&-1\\
        1&1
    \end{pmatrix}=\begin{pmatrix}
        1&0\\
        0&1
    \end{pmatrix}.
    \end{split}
    \end{equation}

    Consider first the $D_n$ case, so that up to permutation $(p,q,r) =(1,1,n-3).$ Further, $a+b+p+q+r=7$ (cf. Lem. \ref{lemma: key commutative diagram X,Y,Z,W, and 8 curves}).
    Using Computation \ref{comp: DE case}, we find the only integral possibilities with $n \ge 4, \ell \ge 0$ to be
    \begin{align*}
        (p,q,r,a,b,n,\ell) \in \{(1,1,1,0,4,4,1),(1,1,4,1,0,7,0)\}
    \end{align*}

    These give valid toric surfaces, and hence a $2$ parameter family of semi-elliptic quadrilaterals. 
    The algebraic constraint, $\Vol(\Delta)=m^2,$ however has no integral solutions for either tuple as can be checked in Computation \ref{comp: DE case}. Hence this case is impossible.

    \medskip

    Next, consider the $E_n$ case, so that up to permutation $(p,q,r) =(1,2,n-4).$ Further, $a+b+p+q+r=7.$ This time, Computation \ref{comp: DE case} shows the only integral possibilities with $n \ge 5, \ell \ge 0$ are
    \begin{align*}
        (p,q,r,a,b,n,\ell)=(1,2,2,0,2,6,0)
    \end{align*}

    This gives a valid toric surface, and hence a $2$ parameter family of semi-elliptic quadrilaterals. 
    While the algebraic constraint, $\Vol(\Delta)=m^2,$ does give rational solutions, one of the side lengths is negative as can be checked in Computation \ref{comp: DE case}. Hence this case is impossible.

    \medskip
    
    Next, we get a similar equation for the second case in Figure \ref{fig:DE schematic}:

    \begin{equation}
    \label{eq: DE eq 2}
        \begin{split}
        &\begin{pmatrix}
        0&-1\\
        1&2
    \end{pmatrix}^p\begin{pmatrix}
        0&-1\\
        1&3+\ell
    \end{pmatrix}\begin{pmatrix}
        0&-1\\
        1&2
    \end{pmatrix}^q\begin{pmatrix}
        0&-1\\
        1&1
    \end{pmatrix}\begin{pmatrix}
        0&-1\\
        1&2
    \end{pmatrix}^{r-1}\begin{pmatrix}
        0&-1\\
        1&3
    \end{pmatrix}\begin{pmatrix}
        0&-1\\
        1&2
    \end{pmatrix}^\ell \\
    &\times \begin{pmatrix}
        0&-1\\
        1&1
    \end{pmatrix}\begin{pmatrix}
        0&-1\\
        1&2
    \end{pmatrix}^b\begin{pmatrix}
        0&-1\\
        1&3
    \end{pmatrix}\begin{pmatrix}
        0&-1\\
        1&2
    \end{pmatrix}^a\begin{pmatrix}
        0&-1\\
        1&1
    \end{pmatrix}=\begin{pmatrix}
        1&0\\
        0&1
    \end{pmatrix}.
    \end{split}
    \end{equation}
    Consider first the $D_n$ case, so that up to permutation $(p,q,r) =(1,1,n-3).$ Further, $a+b+p+q+r=7.$
    Using Computation \ref{comp: DE case}, we find the only integral possibilities with $n \ge 4, \ell \ge 0$ to be
    \begin{align*}
        (p,q,r,a,b,n,\ell) =(1,1,2,0,3,5,0)
    \end{align*}

    As before, we get a valid toric surface and rational solutions to $\Vol(\Delta)=m^2,$ but one of the side lengths is negative as can be checked in Computation \ref{comp: DE case}. Hence this case is impossible.

    \medskip

    Next, consider the $E_n$ case, so that up to permutation $(p,q,r) =(1,2,n-4).$ Further, $a+b+p+q+r=7.$ This time, Computation \ref{comp: DE case} shows the only integral possibilities with $n \ge 5, \ell \ge 0$ is
    \begin{align*}
        (p,q,r,a,b,n,\ell) = (1,2,1,0,3,5,1).
    \end{align*}
    However, since $D_5 \simeq E_5,$ one can check this is the same case we treated above.

    \medskip

    This exhausts all the possibilities and hence we are done.
\end{proof}

\section{The smooth case}
\label{sec: the smooth case}
The method of approach is more algebraic now, using conditions such as $\Width(\Delta) \ge m, \Vol(\Delta)=m^2$ and the matrix equation (see \ref{lem: matrix equation}). However, we first classify the chains that can appear on the toric boundary of $Y.$ We do this by first finding the self-intersection of the heavy curve in terms of a parameter $s,$ and express the matrix of the chains in the relevant locus in terms of $s.$ We do this by the condition that the relevant locus contracts to a point.


\subsection{Classification}

We now look at the $A_0$ case, wherein the chain of contractions $Y = Y_0 \to \dots Y_n = W$ involves contractions of $(-1)$ curves, contracting the relevant locus (cf. Def. \ref{defi: relevant locus}) in a smooth point. 

\medskip

Look at $Y \to Y_1 \to \dots W$ as a series of blowups of a smooth point. By $(2)$ of Lemma \ref{lemma: auxilliary lemma about high multiplicity intersections}, there is always just one $(-1)$ curve in the relevant locus of $Y_i$ for $1 \le i < \ell.$ 
Hence we always have to blowup at a point on the $(-1)$ curve (otherwise we introduce an extra $(-1)$ curve).
An easy observation is the following:

\begin{lemma}
\label{lemma: (-1) intersects at most 2 other}
The $(-1)$ curve intersects at most $2$ other curves (in the relevant locus).
\end{lemma}
\begin{proof}
Assume not, and consider the contraction of the $(-1)$ curve. Then at least three curves intersect at the same point, contradicting the fact that the relevant locus is always simple normal crossing (cf. Lem. \ref{lemma: auxilliary lemma about high multiplicity intersections}).
\end{proof}


\begin{definition}
    \label{def: type I, type II blowups}
    Consider the intersection points of the $(-1)$ curve with other curves in the relevant locus (possibly none), and consider the blowup at a point on it.
    Call it a \defi{type I blowup} if the point is an intersection point. Otherwise, call it a \defi{type II blowup}.
\end{definition}

A priori, a type II blowup can occur even the $(-1)$ curve intersects two other curves. However we can show that doesn't happen when the relevant locus contracts to a point.

\begin{figure}[h]
    \centering
    \includegraphics[scale=0.2]{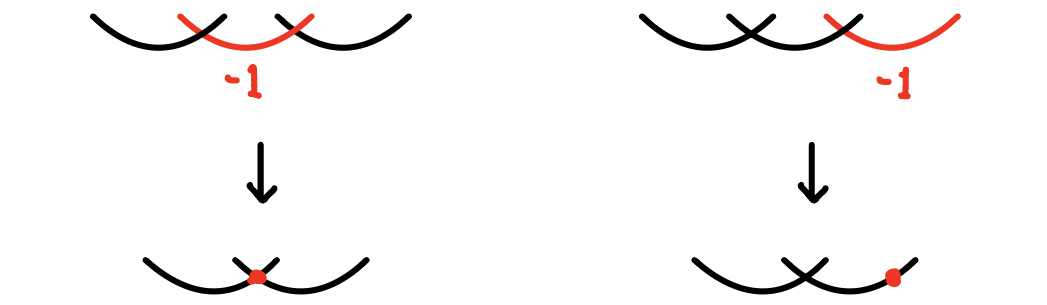}
    \caption{Possible blowups: Type I and II}
    \label{fig:possible contractions}
\end{figure} 

\begin{lemma}
\label{lemma: type II blowups then type I}
\hangindent\leftmargini
$(1)$\hskip\labelsep A type II blowup happens only if the $(-1)$ curve intersects exactly one curve in the relevant locus (also see Lemma \ref{lemma: (-1) intersects at most 2 other}).
\begin{enumerate}
    \item[(2)] Furthermore, there are some type II blowups initially $Y_{\ell-s} \to \dots Y_{\ell}$, and then any subsequent blowup is of type I. 
\end{enumerate}
\end{lemma}
\begin{proof}
Every curve in this proof is a curve in the relevant locus.
\begin{enumerate}
    \item 
    From Lemma \ref{lemma: key lemma about endchains}, the relevant locus in $Y$ is a chain of curves. Hence the relevant locus is always a chain of curves. 
    Consequently, no blowup can introduce a node of degree $3$ in the dual graph (of the relevant locus). This shows that a blowup of type II cannot occur if the $(-1)$ curve intersects two other curves. 
    \item Observe that after each initial type II blowup, the unique $(-1)$ curve intersects exactly one curve. 
    However, after a type I blowup, it will intersect two curves (see Figure \ref{fig:possible contractions}). The result hence follows from part $(1).$  \qedhere
\end{enumerate} 
\end{proof}



Hence starting at a point, there are some initial blowups of type II, and then any subsequent blowup must be of type I. Since a blowup of type I doesn't change the matrix of the chain (cf. Lem. \ref{lem: blowup at intersection}), we can classify relevant locus' matrix by the number of type I blowups.

\begin{figure}[h]
    \centering
    \includegraphics[scale=0.4]{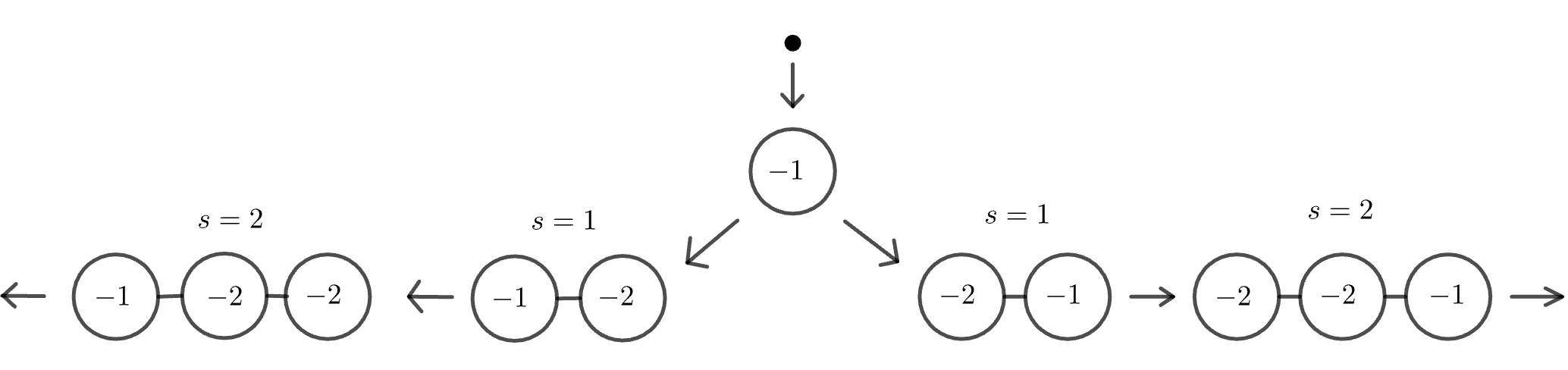}
    \caption{The tree of possibilities}
    \label{fig:key tree A0}
\end{figure}

Since the orientation of our chain matters (for Hirzebruch-Jung fractions, say), there are two branches in the tree of possibilities shown in Figure \ref{fig:key tree A0}.
Observe that the matrices we get in the two branches are reverses of each other (in the sense of Definition \ref{lemdef: reverse matrix}).

\medskip

Firstly, observe that the single $(-1)$ curve case in Figure \ref{fig:key tree A0} cannot occur, since it corresponds to $\wt R$ passing through two smooth points, which has been ruled out in Proposition \ref{prop: opposite duval}. Hence we can assume $s>0.$

\begin{lemma}
    \label{lem: classification of matrices, -(s+5)}
    Define $s>0$ as in Figure \ref{fig:key tree A0}, i.e. one less than the number of type II blowups. 
    Up to orientation,  the matrix of the chain is (in our case, these are the chains on the right in Figure \ref{fig:key tree A0})
    $$\begin{pmatrix}
        -s & -1 \\
        s+1 & 1
    \end{pmatrix}$$
    Furthermore, the heavy curve (cf. Def. \ref{def: heavy curve}) has self-intersection $-(s+5).$
\end{lemma}
\begin{proof}
    Recall that matrices stay invariant for blowups at intersection of curves, see Lemma \ref{lem: blowup at intersection}. So we only need to consider blowups of type II, i.e. at non-intersection points.
    It is easy to see the chain consists of $s$ $(-2)$ curves followed by a $(-1)$ curve. Hence Example \ref{example: du val matrix} 
    shows the matrix is
    $$\begin{pmatrix}
        0 & -1 \\
        1 & 2
    \end{pmatrix}^{s}\begin{pmatrix}
        0 & -1 \\ 1 & 1
    \end{pmatrix}=\begin{pmatrix}
        -(s-1) & -s \\
        s & s+1
    \end{pmatrix}\begin{pmatrix}
        0 & -1 \\ 1 & 1
    \end{pmatrix}=\begin{pmatrix}
        -s & -1 \\
        s+1 & 1
    \end{pmatrix}.$$
    Let $\mc L_i, \mc S_i$ be the length, sum of self-intersection numbers of the relevant locus in $Y_i.$
    We claim that $\mc S_i+3\mc L_i = s+2$ always.
    As $(-1)+3(1)=0+2,$ this holds for $Y_\ell$ when we have just one $(-1)$ curve.
    For a blowup at the intersection of two curves, observe that $\mc S_i$ decreases by $3$ and $\mc L_i$ increases by $1,$ showing $\mc S_i+3\mc L_i$ is invariant. Hence it only depends on the number of type II blowups.

    \medskip
    
    For a type II blowup, the key observation is that each blowup increases $\mc L_i$ by $1,$ but decreases $\mc S_i$ by only $2.$ Hence $\mc S_i+3\mc L_i$ increases by $1,$ just like $s+2.$ 

    \medskip

    Let $x$ be the self-intersection number of the heavy curve.
    Consider the toric boundary curves of $Y.$
    There are at three $(-1)$ toric boundary curves (cf. discussion in \ref{def: heavy curve}), and two chain of $(-2)$ curves of lengths $a,b$ satisfy $a+b=8$ by Lemma \ref{lemma: key commutative diagram X,Y,Z,W, and 8 curves}. 
    Let $N$ be the number of toric boundary divisors in $Y.$
    Furthermore, the length, sum of self-intersection numbers of the relevant
    locus in $Y = Y_0$ is $\mc L_0 - 1, \mc S_0 + 1$ respectively, since $\wt R$ is a $(-1)$ curve in the relevant locus which is not a toric boundary curve.
    By Lemma \ref{lemma: sum of smooth numbers}, the sum of self-intersection numbers of toric boundary divisors of $Y$ is $12-3N.$ Hence
    \begin{align*}
        &x+(a+b)(-2)+3(-1)+(\mc S_0+1) = 12-3(4+a+b+\mc L_0-1)\\
        &\implies x = -(a+b)-(\mc S_0+3\mc L_0)+5=-8-(s+2)+5=-(s+5)
    \end{align*}
    and we are done.
\end{proof}


We now have the setup to deal with Cases I and II from Figure \ref{fig:3 cases}. 
Both of them have a different analysis, however the key approach is the same; we attack them algebraically using the matrix equation  \ref{lem: matrix equation} and $\Vol(\Delta)=m^2$.

\subsection{Case I}
\label{subsec: A0 Case I}

Based on our analysis, every toric boundary curve is determined except the relevant locus. 
Our key claim, however, are the following matrix equations whose solutions bijectively correspond to semi-elliptic (smooth) toric surfaces (cf. Rem. \ref{remark: semi-elliptic surface}).

\begin{lemma}
\label{lemma: A0 case I both equations}
There is a bijection between semi-elliptic (smooth) toric surfaces in Case I and integral matrix solutions $\mf M,\mf N$ of the system of equations
\begin{equation}
\label{eq: A0 case I eq 1}
    \mf M \begin{pmatrix}
    0 & -1 \\
    1 & 1
\end{pmatrix} \mf N =\begin{pmatrix}
    -s & -1 \\
    s+1 & 1
\end{pmatrix} 
\end{equation}
\begin{equation}
    \label{eq: A0 case I eq 2}
        \mf N \begin{pmatrix}
        0&-1\\1&1
    \end{pmatrix}\mf M \begin{pmatrix}
        0&-1\\1&1
    \end{pmatrix} \begin{pmatrix}
        0&-1\\1&2
    \end{pmatrix}^a \begin{pmatrix}
        0&-1\\1&s+5
    \end{pmatrix}\begin{pmatrix}
        0&-1\\1&2
    \end{pmatrix}^b\begin{pmatrix}
        0&-1\\1&1
    \end{pmatrix}=\mf I
\end{equation}
where $s \in \ZZ_{>0}$ and $a+b=8$ are any non-negative integers.
\end{lemma}
\begin{proof}
Label $\mc S,\mc S',\mc D \subset Y$ as before (cf. Fig. \ref{fig:3 cases}).
Let $\wt R$ pass through $\mc C$ and $\mc C',$ which are curves in $\mc S, \mc S'$ respectively.
Recall the definition of inner and outer curve from Definition \ref{def: inner and outer curves}. Suppose the matrix from the inner curve to outer curve of $\mc S$ is $\mf A_1,$ and the matrix from the outer curve to inner curve of $\mc S'$ is $\mf A_2,$ see Figure \ref{fig:A0 case I}. Since the chain $\mc S \cup \wt R \cup \mc S'$ contracts to a point, Lemma \ref{lem: classification of matrices, -(s+5)} shows
\begin{equation}
\label{eq: A0 case I A_i equation}
    \mf A_i\begin{pmatrix}
    0 & -1 \\
    1 & 1
\end{pmatrix}\mf A_{(i+1)}=\begin{pmatrix}
    -s & -1 \\
    s+1 & 1
\end{pmatrix}
\end{equation}
for some $i \in \{1,2\},$ where indices are modulo $2.$ We can then define $\mf M:=\mf A_i$ and $\mf N:=\mf A_{i+1}$ so that Equation \ref{eq: A0 case I eq 1} holds.

\begin{figure}[h]
    \centering
    \includegraphics[scale=0.2]{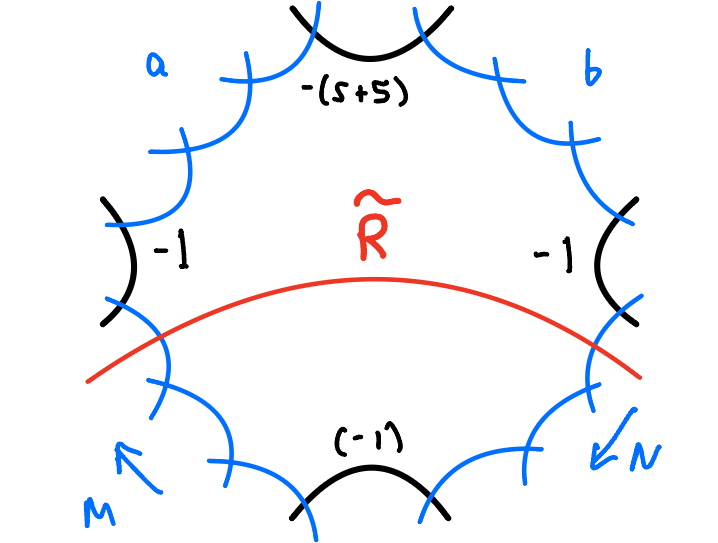}
    \caption{Schematic picture for this case. Note the positions of $a,b$ so that Equation \ref{eq: A0 case I eq 2} holds.}
    \label{fig:A0 case I}
\end{figure}

Outside the relevant locus, the exceptional locus of $Y \to X$ contains two chains of $(-2)$ curves. We can appropriately label the number of curves $a,b$ so that the Matrix equation \ref{lem: matrix equation} implies Equation \ref{eq: A0 case I eq 2}, see Figure \ref{fig:A0 case I}. Lastly, we need to argue why this is a bijection.
However, that is easy; for any chain of curves $\mc S, \mc S'$ with matrices $\mf M, \mf N$ satisfying Equation \ref{eq: A0 case I eq 1}, the divisor $\mc S \cup \wt R \cup \mc S'$ contracts to a point. Further, the matrix equation guarantees the existence of a toric surface.
\end{proof}

Write
\begin{equation}
\label{eq: definition of M}
    \mf M=\begin{pmatrix}
    x&y\\z&w
\end{pmatrix}.
\end{equation}
Observe here that $x < y<0 < z<w$ and $\det(\mf M)=xw-yz=1$ (cf. Lem. \ref{lem: matrix and chain}).

\begin{lemma}
    \label{lemma: pell's equation for A0}
    There are only two possibilities for the triple $(a,b,s),$ each one giving a Pell's equation in $z,w:$
    $$\begin{cases}
        (a,b,s) = (0,8,2), & 55z^{2}-95zw+41w^{2}+1=0\\
        (a,b,s) = (8,0,2), & 55z^{2}-15zw+w^{2}+1=0
    \end{cases}.$$
\end{lemma}
\begin{proof}
    Using Equation \ref{eq: A0 case I eq 1} to express $\mf N$ as a matrix in $x,y,z,w,$ Equation \ref{eq: A0 case I eq 2} gives $4$ equations in $\{x,y,z,w,a,b,s\}.$ As $a+b=8$ are non-negative (see \ref{lemma: key commutative diagram X,Y,Z,W, and 8 curves}), we can test all $9$ possible pairs. 
    In each case, we look at the minimal primes of the ideal generated by these $4$ equations (on \textsc{Macaulay2}). 
    Fascinatingly enough, there are only two non-negative integral values of $(a,b,s)$ that work, producing the claimed Pell's equation that $z,w$ must satisfy
    (see Computation \ref{comp: A0 case I 1}).
\end{proof}

These equations aren't enough to get a contradiction; turns out there are toric surfaces which satisfy the conditions of Lemma \ref{lemma: pell's equation for A0}. However, the corresponding quadrilaterals fail to satisfy the condition $\Vol(\Delta)=m^2.$ To proceed, we need to understand how this condition translates to the fan's structure:

\begin{lemma}
\label{lem: vol = m^2 equation}
    Consider the adjacent Du Val case, and assume the Du Val chains in the exceptional but non-relevant locus have lengths $a,b.$ Define $c:=a+1, d:=b+1.$ Consider the fan $v_0=(w,-z), v_1=(0,1), v_2=(-c,1), v_3$ so that $\measuredangle (v_1,v_2), \measuredangle (v_2,v_3)$ are Du Val of type $A_{a},A_{b}$ respectively, and $v_0, v_1$ are positive (cf. Lem. \ref{lemma: special fan}). 

    \begin{figure}[h]
        \centering
        \includegraphics[scale=0.3]{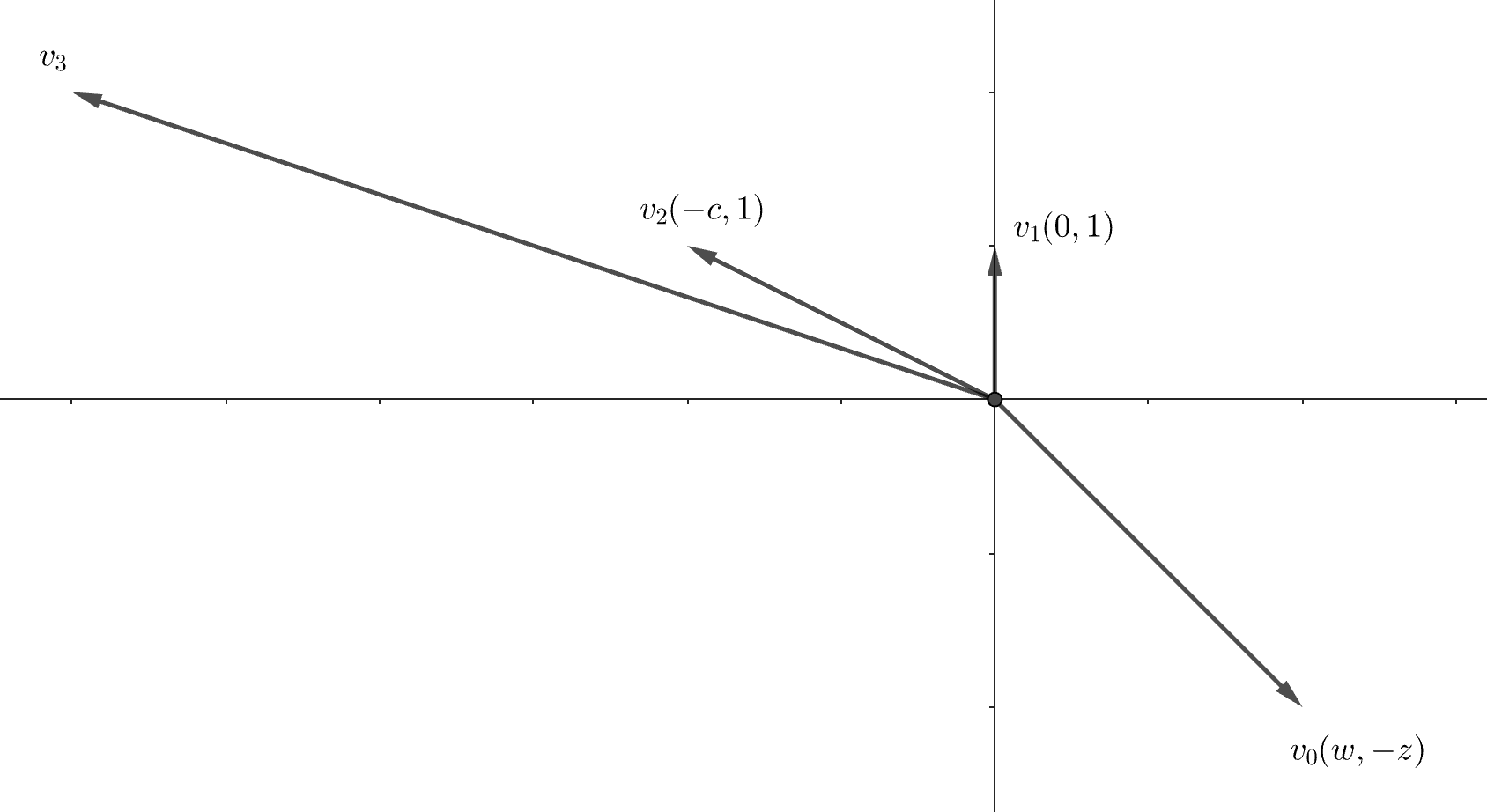}
        \includegraphics[scale=0.35]{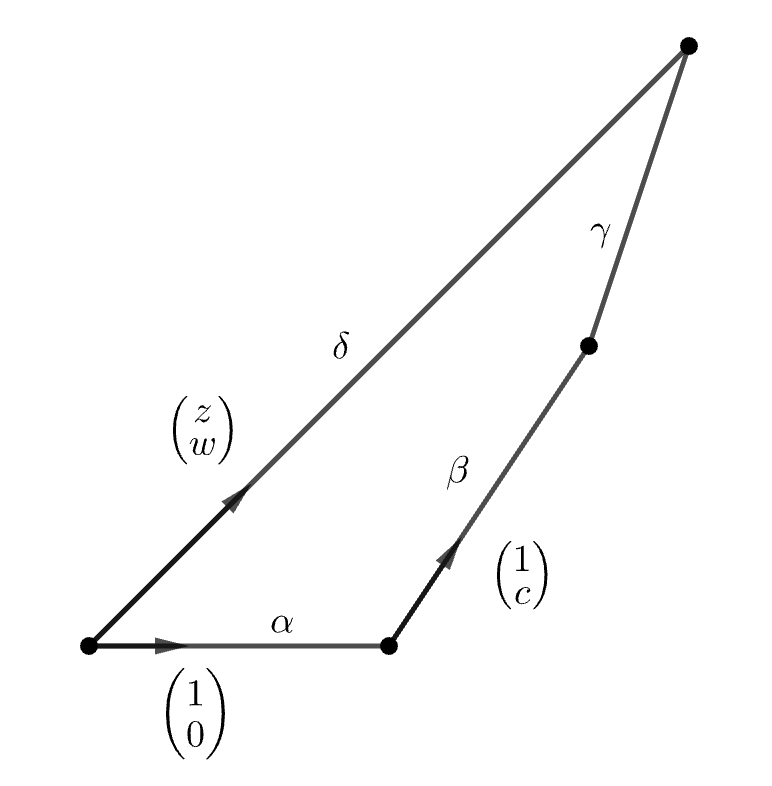}
        \caption{The fan and corresponding quadrilateral}
        \label{fig:A0 vol}
    \end{figure}
    
    Assume $\Delta$ has lattice side-lengths $\alpha,\beta,\gamma,\delta$ corresponding to the vectors $v_1,v_2,v_3,v_0$ respectively (see Figure \ref{fig:A0 vol}). Then the lattice perimeter is $m=\alpha+\beta+\gamma+\delta$ and 
    $$\Vol(\Delta) = \alpha \delta w + \beta \gamma d.$$
\end{lemma}
\begin{proof}
    The formula for $m$ follows from the definition. 
    To prove the volume formula, let $V_i$ be the edge vector in $\Delta$ corresponding to $v_i.$ Subdivide the quadrilateral into two triangles; one with vectors $V_0, V_1$ and the second with vectors $V_2,V_3.$ Observe that the first has volume\footnote{$v \times w$ is the cross product of the vectors $v,w$}
    $$\lt|\alpha V_0 \times \delta V_1\rt|=\alpha \delta |v_0 \times v_1|=\alpha \delta w$$
    since $|v_0 \times v_1| = |V_0 \times V_1|.$ Similarly the second triangle has volume $$|\beta V_2 \times \gamma V_3 |=\beta \gamma |v_2 \times v_3|$$
    We now claim that $|v_2 \times v_3|=|\det (v_2, v_3)|=d.$ Indeed, as $\measuredangle (v_2,v_3)$ is a Du Val singularity of type $b=d-1,$ up to $\GL_2(\ZZ)$ its fan is $(0,1),(-d,1).$ Then $|\det (v_2, v_3)|=d$ is clear, and invariance of determinant under $\GL_2(\ZZ)$ implies the desired result.
\end{proof}

\begin{remark}(Notational Consistency)
The configuration we consider in Figure \ref{fig:A0 case I} will have the fan \ref{fig:A0 vol}, i.e. the $a,b,w,z$ used in Lemma \ref{lem: vol = m^2 equation} are consistent with the orientation in Figure \ref{fig:A0 case I}.
To see this, let $v_1$ be the vector corresponding to the common $(-1)$ curve intersecting both the $A_{a}$ chain and the $\mf M$ chain.
Then $v_0=(w,-z)$ shows the matrix from $v_0$ to $v_1$ is exactly $\mf M$ (see (4) of Lemma \ref{lem: matrix and chain}).
Further, $v_2=(-c,1)$ shows $\measuredangle (v_1,v_2)$ corresponds to the $A_a$ chain.
\end{remark}

We can calculate $v_3$ explicitly:

\begin{lemma}
\label{lemma: value of v3}
    In notation of Lemma \ref{lem: vol = m^2 equation}, $v_3=\begin{pmatrix}
        -c-d-\ell cd \\1+\ell d
    \end{pmatrix},$ where $\ell = s+3.$ 
\end{lemma}
\begin{proof}
    Consider the shear transform $\mf A \in \SL_2(\ZZ)$ mapping $v_2$ to $(0,1).$ 
    As $\measuredangle (v_2,v_3)$ is an $A_{d-1}$ singularity, $v_3$ maps to the vector $(-d, 1+\ell d)$ for some $\ell.$ 
    Further $v_1$ maps to $(c,1).$ 
    We now show that $\ell=s+3$ by the fact that the divisor corresponding to $v_2$ has self-intersection $-(s+5)$ (cf. Lem. \ref{lem: classification of matrices, -(s+5)}).
    Let $u, w$ be the resolution vectors adjacent to $v_2,$ so that $u = (1,1).$ 
    It suffices to show that $w=(-1,1+\ell),$ since then $u+w = (s+5)v_2$ shows $2+\ell = s+5,$ i.e. $\ell=s+3.$

    \begin{figure}[h]
        \centering
        \includegraphics[scale=0.3]{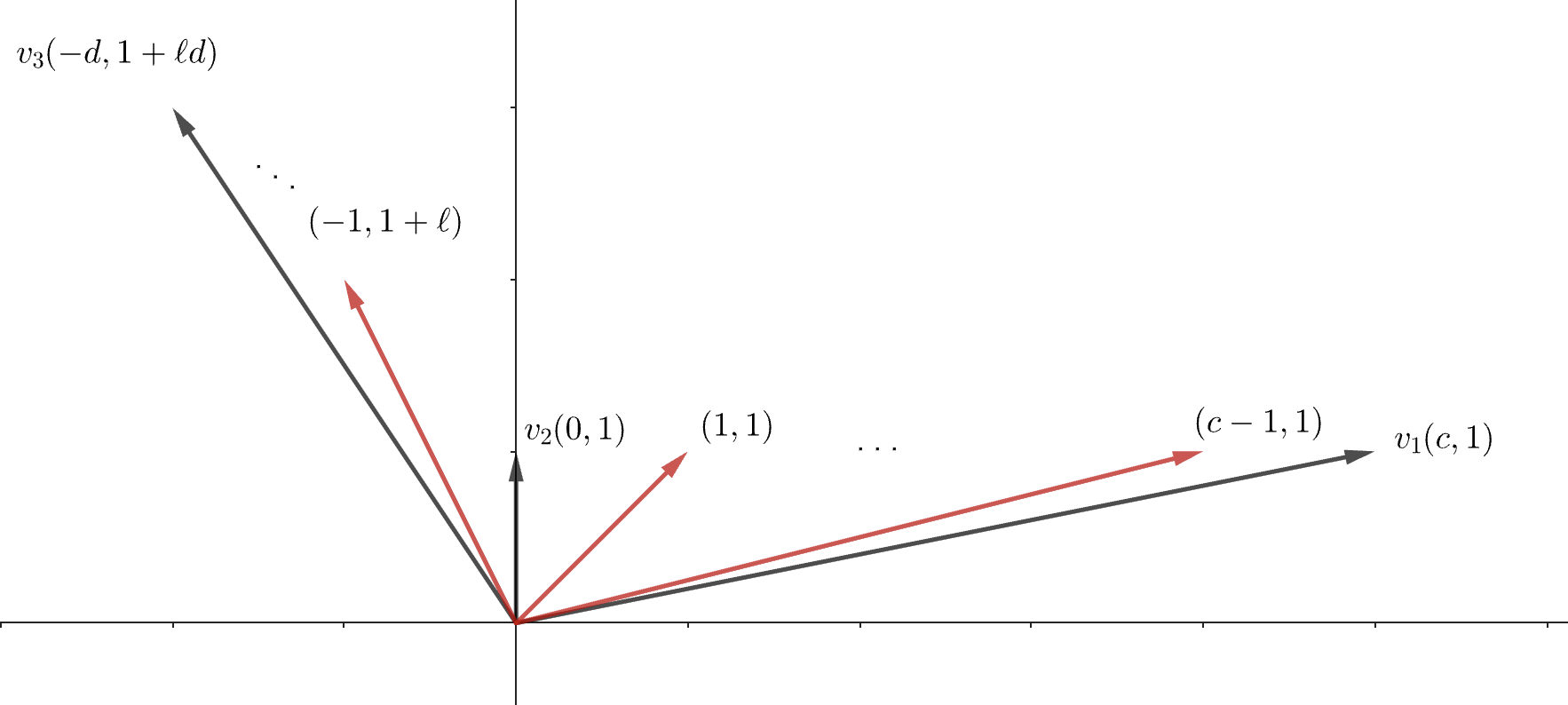}
        \caption{Fan after the shear transform $\mf A=\begin{pmatrix}
            1 & c \\ 0 & 1
        \end{pmatrix}$}
        \label{fig:enter-label}
    \end{figure}

    However this is easy; the shear transform $\mf B \in \SL_2(\ZZ)$ fixing $v_2=(0,1)$ and mapping $v_3$ to $(-d,1)$ must map $w$ to $(-1,1).$ Hence $w=\mf B^{-1}(-1,1)=(-1,1+\ell).$  

    \medskip

    To finish, $v_3=\mf A^{-1}(-d,1+\ell d)=(-d-c(1+\ell d),1+\ell d),$ as desired.
\end{proof}

\begin{remark}
\label{remark: -(s+5) also for An}
    Note that the only result about $s$ used in Lemma \ref{lemma: value of v3} was that the self-intersection of the divisor corresponding to $v_4$ is $-(s+5).$ 
    In the $A_n$ case analysis with $n>0,$ we will define $s$ differently, but the self-intersection will remain $-(s+5)$ (see \ref{lemma: An matrix classifcation}). Hence we will be able to reuse this lemma. 
\end{remark}

We can use this to get second equations in both the cases of Lemma \ref{lemma: pell's equation for A0}:

\begin{lemma}
\label{lem: A0 elliptic curve}
    There are only $2$ possibilities for the triple $(a,b,s).$ For each possibility, the points $(z,w)$ lie on an intersection of two quadric surfaces in $\mathbb{A}^3_{z,w,t}:$
    \begin{align*}
        (a,b,s) = (0,8,2): &\begin{cases}
        55z^{2}-95zw+41w^{2}+1=0 \\
        1980z^{2}-13320zw+9756w^{2}+3960z-3240w+1980+t^2=0
    \end{cases} \\ 
    (a,b,s) = (8,0,2): &\begin{cases}
        55z^{2}-15zw+w^{2}+1=0 \\
        1980z^{2}-10440zw+1116w^{2}+3960z-360w+1980+t^2=0
    \end{cases}
    \end{align*}
    where $t$ is an integer parameter.
    Both intersections are reducible and are the unions of two conics.
\end{lemma}
\begin{proof}
    We get the two cases and the first equation using Lemma \ref{lemma: pell's equation for A0}. 
    Let $\alpha,\beta,\gamma,\delta$ be the (integer) lattice lengths of $\Delta,$ and $v_0,v_1,v_2,v_3$ be vectors in the toric fan (as in Lemma \ref{lem: vol = m^2 equation}).
    Using Lemma \ref{lemma: value of v3} and the fact that $c+d=a+b+2=10$ (cf. Lem. \ref{lemma: key commutative diagram X,Y,Z,W, and 8 curves}),
    we find
    $$\alpha\begin{pmatrix}
        1\\0
    \end{pmatrix}+\beta\begin{pmatrix}
        1\\c
    \end{pmatrix}+\gamma \begin{pmatrix}
        1+\ell d \\ 10+\ell cd
    \end{pmatrix}=\delta \begin{pmatrix}
        z \\ w
    \end{pmatrix}.$$
    Hence, 
    \begin{equation}
    \label{eq: A0 a', b' expression}
    \begin{split}
        &c\beta = \delta w - \gamma(10+\ell cd) \\
        &\alpha = \delta z-\gamma(1+\ell d)-\beta.
        \end{split}
    \end{equation}
    We can hence express $\alpha,\beta$ as a homogeneous linear forms in $\gamma, \delta.$ 
    We can now use the $\Vol(\Delta)=m^2$ condition with Lemma \ref{lem: vol = m^2 equation} to obtain (see Computation \ref{comp: A0 case I 1}):
    \begin{equation}
    \label{eq: Case A0 I equations}
    \begin{split}
    (a,b,s)&=(0,8,2): \\
    &-(z^{2}-zw+w^2+2z+1)\delta^{2}+(90z+18w+90)\delta\gamma-2520\gamma^2=0 
    \\
    (a,b,s)&=(8,0,2): \\
    -&\lt(9z^{2}-9zw+w^{2}+18z+9\rt)\delta^2+\lt(90z+2w+90\rt)\delta \gamma-280\gamma^2 = 0.
    \end{split}
    \end{equation}
    Since $q$ has rational roots, its discriminant, a quadratic form in $z,w,$ is an integer square, say $t^2.$
    That gives the second equation in each case.
    Note that in Computation \ref{comp: A0 case I 1}, we dehomogenize the above to express $\texttt{a'} := \tfrac{\alpha}{\gamma},\texttt{b'} := \tfrac{\beta}{\gamma}$ as elements of $\QQ(\texttt{d'}) := \QQ(\tfrac{\delta}{\gamma}).$

    \begin{remark}[Note on scaling]
    \label{remark:scaling}
        We have to be careful with scaling the sides of $\Delta,$ since being elliptic in general is not a property invariant under scaling: consider $\Delta$ from Example \ref{example:random example}.
        Then $\Delta$ is elliptic, but $3\Delta$ isn't, as one checks $\dim \mc L_{3\Delta}(3m)=2$ using a computation similar to Computation \ref{comp:random example}. 
        However, note that the relations $\Vol(\Delta)=m^2$ and $\Width(\Delta) \ge m$ (for $m=|\partial \Delta \cap \ZZ^2|)$ stay invariant under scaling.
        Since we show none of the semi-elliptic quadrilaterals can satisfy both of these relations, we only ever deal with these two conditions. Hence we can freely scale $\Delta.$
    \end{remark}

    \medskip

    Finally, we use \texttt{minimalPrimes} on \textsc{Macaulay2} to find the minimal primes of the ideal generated by \texttt{disc=t}$^2$ and the Pell's equation. In both the cases, we find a product of two ideals, both given by the intersection of a hyperplane and quadric in $\mathbb{A}^3_{z,w,t};$ see Computation \ref{comp: A0 case I 1}.
\end{proof}

\subsubsection*{Case $(a,b,s)=(0,8,2)$}
\label{subsec: Case (a,b,s)=(0,8,2)}
The Pell's equation in this case is
\begin{equation}
\label{eq: A0 case I abs=082 pell}
55z^{2}-95zw+41w^{2}+1=0.
\end{equation}
Here, we get explicit roots to our quadratic in Equation \ref{eq: Case A0 I equations} in terms of $z,w$ because the intersection lies on one of the two hyperplanes (cf. Comp. \ref{comp: A0 case I 1}):
$$t=\pm (330z-270w-6).$$
However we rule out both of them: the first root will imply the lattice length of an edge of $\Delta$ is negative. For the second one, we will show either the lattice length of an edge is negative, or the width is less than $m.$ 

\medskip

The roots to the quadratic Equation \ref{eq: Case A0 I equations} are $\delta/\gamma = (U(z,t) \pm t)/V(z,t)$ for two polynomial expressions $U,V$ in $z,t.$ Up to sign, we know $t$ in terms of $z$ and $w.$ Hence we get two possibilities for $\delta/\gamma,$ which we treat separately. 
We will use very crude bounds to get contradictions.
Corresponding to our inequality constraints, we obtain open regions in $\RR^2_{z,w}.$ The approach to show they don't intersect is to find lines between them.
All (difficult) computations of this case are done in \ref{comp: A0 Case I (a,b,s)=(0,8,2)}

\begin{enumerate}
    \item Consider the first root corresponding to $+t.$ We then compute (cf. Comp. \ref{comp: A0 Case I (a,b,s)=(0,8,2)})
    \begin{equation}
    \label{eq: Case A0 abs=082 beta}
    \begin{split}
        \frac{\beta}{\gamma} = \frac{-110z^{2}-130zw+178w^{2}-220z+96w-110}{2(z+1)^2+2w(w-z)}
    \end{split}
    \end{equation}
    However, we show that $\beta/\gamma<0$ in this case, which would be our desired contradiction.

    \begin{claim}
    \label{claim: A0 case I abs=082 +t}
    Let $(z,w)$ be non-negative integers with $w>z.$ Then (cf. Rem. \ref{remark: motivation regions}, Fig. \ref{fig:regions})
    \begin{enumerate}
        \item if $z \le 5/6w,$ then there are no solutions to the Pell's equation \ref{eq: A0 case I abs=082 pell};
        \item if $z \ge 5/6w,$ then $\beta/\gamma < 0.$
    \end{enumerate}
    \end{claim}
    \begin{proof}
        \begin{enumerate}
            \item Observe that $0>55z^2-95zw+41w^2.$ Consider the quadratic form $f(t)=55t^2-95t+41,$ whose minimum is at $t=95/(2 \times 55)>5/6.$ Observe that $f(5/6)=1/36>0.$ Hence we must have $t>5/6$ for $f(t)=0.$
            \item Firstly, observe that the denominator of in Equation \ref{eq: Case A0 abs=082 beta} is positive since $w>z$ (see the comment below Equation \ref{eq: definition of M}). Hence it suffices to show the numerator is negative.

            \medskip
            
            The easiest way is to write $z=5/6w+\epsilon$ for some $\epsilon>0.$ Making this substitution,
            the numerator is (cf. Comp. \ref{comp: A0 Case I (a,b,s)=(0,8,2)})
            $$-\tfrac{121}{18}w^{2}-\tfrac{940}{3}w\epsilon^2-110\epsilon^{2}-\tfrac{262}{3}w-220\epsilon-
           110<0$$
           as desired.  \qedhere
        \end{enumerate} 
    \end{proof}

    \begin{remark}
    \label{remark: motivation regions}
        For some motivation, the idea is to graph the regions corresponding to the Pell's equation, and the relations $\alpha/\gamma,\beta/\gamma,\delta/\gamma>0.$
        If any two of these are (\textit{eventually}) disjoint, we can get get a contradiction.
        The easiest way to prove they are disjoint is by finding lines in $\RR^2$ separating them.
        Once we have the right guess, using a graphing software say, we can prove it works using \textsc{Macaulay2}.
        For instance, Claim \ref{claim: A0 case I abs=082 +t} is visualized in Figure \ref{fig:regions}.

        \begin{figure}[h]
            \centering
            \includegraphics[width=0.3\linewidth]{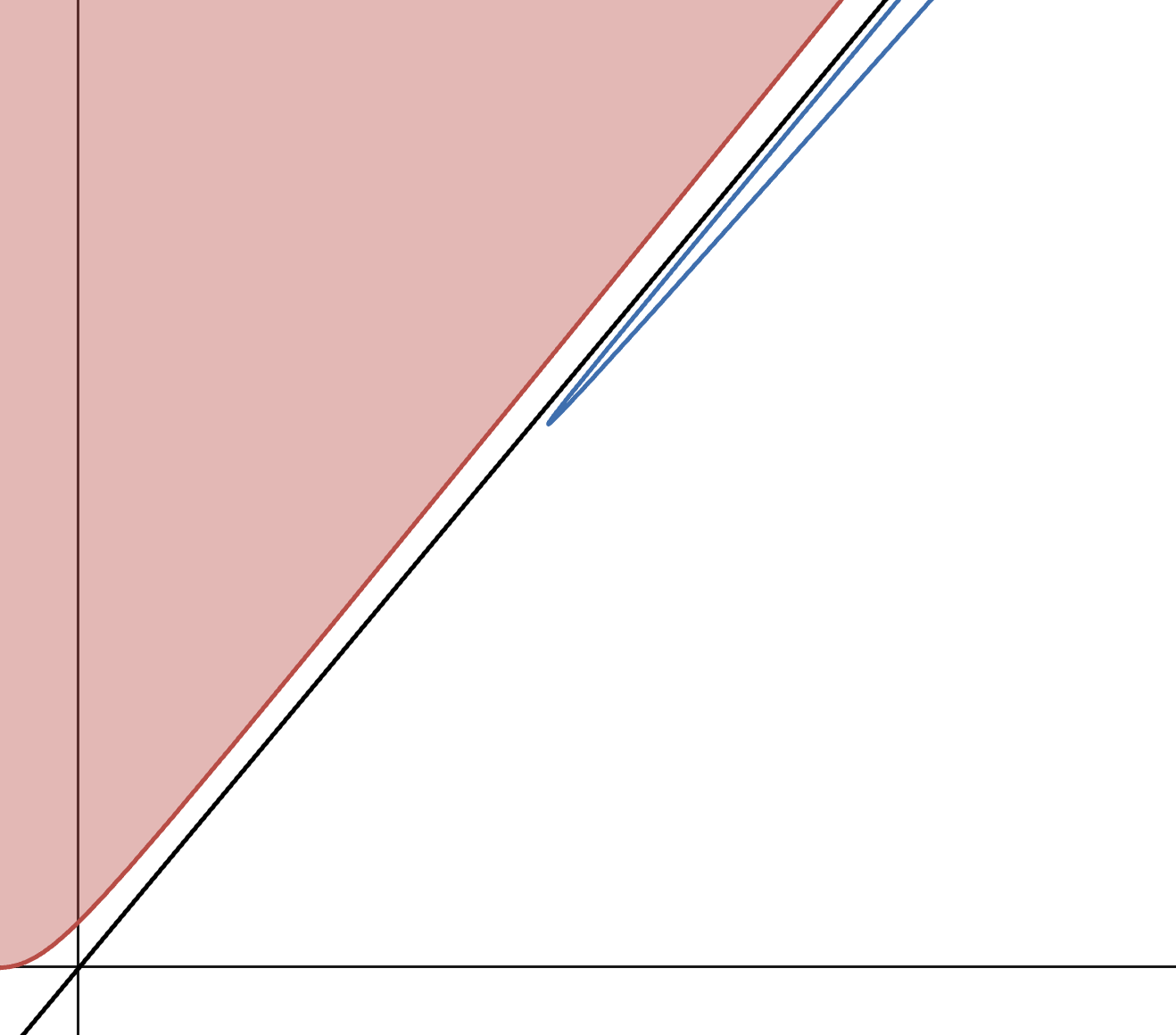}
            \caption{The red region is where $\beta/\gamma>0.$ The blue curve is the Pell's equation. The black line is $z=5/6w,$ the line separating these regions.}
            \label{fig:regions}
        \end{figure}
    \end{remark}


    \item Consider the $-t$ case. We can then compute (cf. Comp. \ref{comp: A0 Case I (a,b,s)=(0,8,2)})
    \begin{equation}
    \label{eq: A0 case I sides -t}
        \begin{split}
            \frac{\alpha}{\gamma}&=\frac{219z^{2}-345zw+135w^{2}+60z-42w+9}{z^{2}-zw+w^{2}+2z+1} \\
            &\frac{\beta}{\gamma} = \frac{-55z^{2}+265zw-181w^{2}-110z+42w-55}{z^{2}-zw+w^{2}+2z+1} \\
            \frac{\delta}{\gamma} &= \frac{210z-126w+42}{z^{2}-zw+w^{2}+2z+1}
        \end{split}
    \end{equation}
    where the denominator is positive as argued in the proof of Claim \ref{claim: A0 case I abs=082 +t}.
    We now prove that outside small values of $(z,w)$, either $(\alpha/\gamma)<0,$ or $\Width_{(1,-1)}(\Delta) < m,$ giving us a contradiction in either case. 

    \begin{figure}[h]
        \centering
        \includegraphics[scale=0.3]{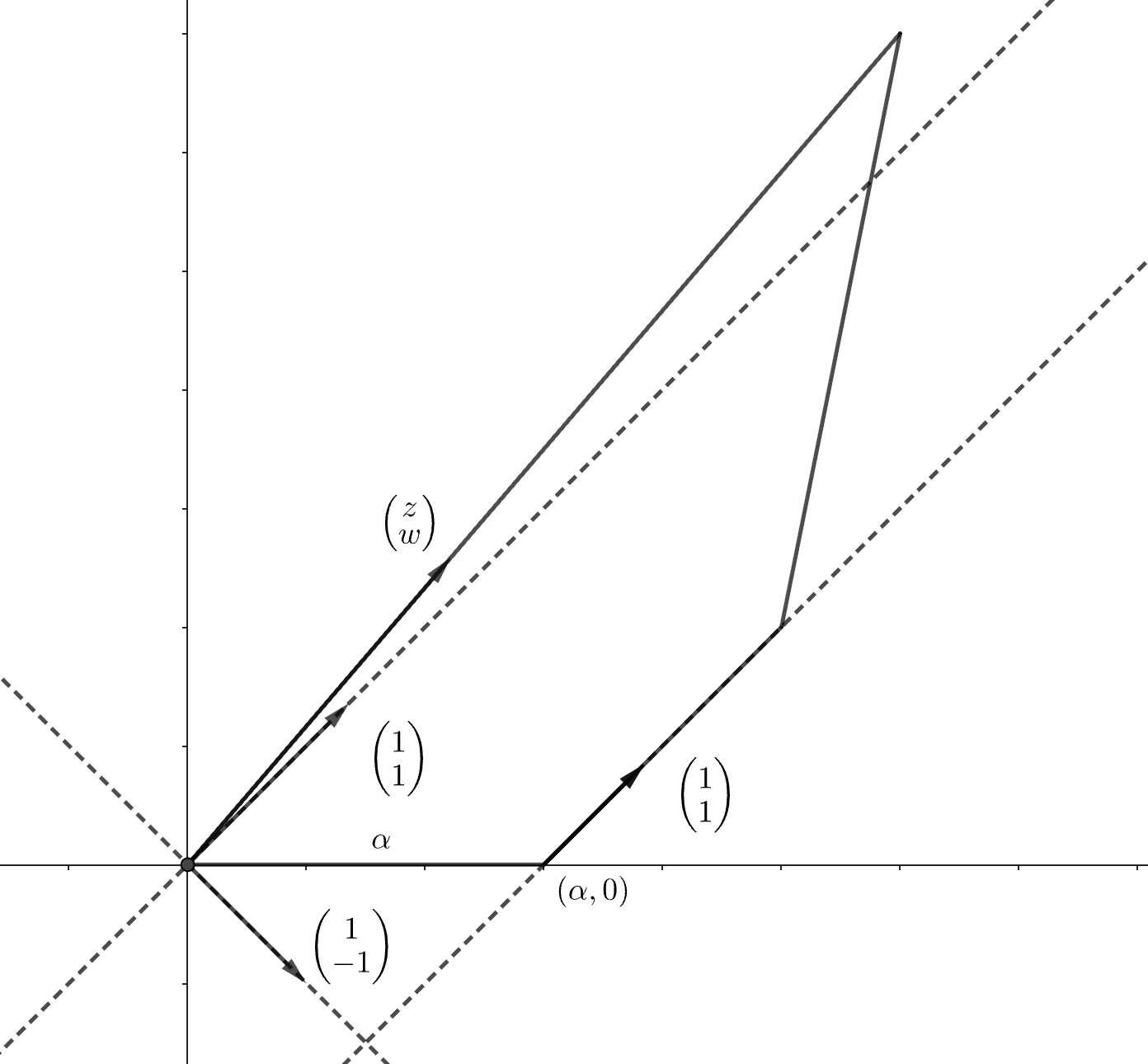}
        \caption{Width along $(1,-1)$}
        \label{fig:width(1,-1)}
    \end{figure}
    
    To compute the width, observe that $w>z$ (see comments after Equation \ref{eq: definition of M}) and $c = a+1= 1.$ As seen in Figure \ref{fig:width(1,-1)}, the width along $(1,-1)$ is 
    $$\lt\la \alpha\begin{pmatrix}
        1 \\ 0
    \end{pmatrix}, \begin{pmatrix}
        1 \\ -1
    \end{pmatrix}\rt \ra - \lt \la \delta\begin{pmatrix}
        z \\ w
    \end{pmatrix},\begin{pmatrix}
        1 \\ -1
    \end{pmatrix} \rt \ra=\alpha + (w-z)\delta.$$

    The condition $\Width_{(1,-1)}(\Delta) \ge m$ translates to
    \begin{equation}
    \label{eq: A0 case I -t width - m}
        (w-z-1)\lt(\frac{\delta}{\gamma}\rt)- \frac{\beta}{\gamma}-1 \ge 0
    \end{equation}
    Combining with Equation \ref{eq: A0 case I sides -t}, we find (cf. Comp. \ref{comp: A0 Case I (a,b,s)=(0,8,2)})
    \begin{equation} \label{eq: A0 Case I abs=082 width - m}
    \frac{-156z^{2}+72zw+54w^{2}-144z+126w+12}{z^{2}-zw+w^{2}+2z+1} \ge 0.
    \end{equation}

    \begin{claim}
    \label{claim: A0 abs = 082 -t}
    Suppose $(z,w)$ are non-negative integers with $z>84,w>4.$ Then
    \begin{enumerate}
        \item if $1.15z \ge w,$ then $\Width_{(1,-1)} < m;$
        \item if $1.18z \ge w \ge 1.15 z$ then the Pell's equation \ref{eq: A0 case I abs=082 pell} has no solutions;
        \item if $1.3z \ge w \ge 1.18z,$ then $\alpha/\gamma<0;$
        \item if $w \ge 1.3z,$ then the Pell's equation \ref{eq: A0 case I abs=082 pell} has no solutions.
    \end{enumerate}
    \end{claim}
    \begin{proof}
    We use the same technique for all; introduce a dummy variable $\epsilon$ and use \textsc{Macaulay2}. See Computation \ref{comp: A0 Case I (a,b,s)=(0,8,2)} for computations.
        \begin{enumerate}
        \item Write $z =  1/1.15w+\epsilon$ and $w=4+x$ for $\epsilon,x \ge 0.$ Then the numerator of Equation \ref{eq: A0 Case I abs=082 width - m} as a function of $x,\epsilon$ is
        $$-\tfrac{714}{529}x^{2}-\tfrac{4584}{23}x\epsilon-156\epsilon^{2}-\tfrac{5298}{529}x-\tfrac{21648}{23}\epsilon-\tfrac{3420}{529} < 0 \; \forall x,\epsilon \ge 0$$
        as desired.
             
        \item Write $w=(1.15+\epsilon)z$ and $z=84+x$ for $0.03>\epsilon>0$ and $x>0.$ Then the expression $55z^2-95zw+41w^2+1$ has a function of $x,\epsilon$ is
        $$f(x,\epsilon)=(41x^2+6888x+289296)\epsilon^2+\dots.$$
        As a quadratic in $\epsilon,$ this has positive leading coefficient. So to show $f<0,$ it suffices to check that only for $\epsilon=0,0.03.$ Indeed, we find
        \begin{align*}
            &f(x,0) = -\tfrac{11}{400}x^{2}-\tfrac{231}{50}x-\tfrac{4826}{25} < 0 \; \forall x\ge 0, \\
            &f(x,0.03) = -\tfrac{29}{2500}x^{2}-\tfrac{1218}{625}x-\tfrac{50531}{625} < 0 \; \forall x\ge 0.
        \end{align*}
        as desired.
        \item Write $z=84+x, w=(1.18+\epsilon)z$ so that $x>0$ and $0.12>\epsilon>0.$ 
        Then we find the numerator of $\alpha/\gamma$ in Equation \ref{eq: A0 case I sides -t} as function of $x,\epsilon$ is
        $$f(x,\epsilon)=(135x^2+22680x+952560)\epsilon^2+\dots.$$
        As a quadratic in $\epsilon,$ this has positive leading coefficient. So to show $f<0,$ it suffices to check that only for $\epsilon=0,0.12.$ Indeed, we find
        \begin{align*}
            &f(x,0) = -\tfrac{63}{500}x^{2}-\tfrac{1341}{125}x-\tfrac{387}{125} < 0  \forall x \ge 0 \\
            &f(x,0.12) = -\tfrac{27}{20}x^{2}-\tfrac{1107}{5}x-9063 < 0  \forall x \ge 0.
        \end{align*}
        This give the desired result.
        \item This follows from $(a)$ in Claim \ref{claim: A0 case I abs=082 +t}.  \qedhere
        \end{enumerate} 
    \end{proof}

    We only have to consider the cases when $w \le 4$ or $z \le 84.$ 
    The only roots of the Pell's equation in this range are
    $$(z,w) \in \{(6,7),(7,8),(11,12),(15,17), (27,32),(38,43),(70,83)\}.$$
    All of these can be manually ruled out, see Computation \ref{comp: A0 Case I (a,b,s)=(0,8,2)}.
\end{enumerate}


\subsubsection*{Case $(a,b,s)=(8,0,2)$}
\label{subsec: Case (a,b,s)=(8,0,2)}
The Pell's equation in this case is

\begin{equation}
\label{eq: A0 case I abs=802 pell}
55z^{2}-15zw+w^{2}+1=0.
\end{equation}

Here, we get explicit roots to our quadratic in Equation \ref{eq: Case A0 I equations} in terms of $z,w$ because the intersection lies on one of the two hyperplanes (cf. Comp. \ref{comp: A0 case I 1}):
$$t=\pm (330z-30w-6).$$
However we rule out both of them: the first root will imply the lattice length of an edge of $\Delta$ is negative. For the second one, we will show either the lattice length of an edge is negative, or the width is less than $m.$ 

\medskip

As before, there are two roots to the quadratic Equation \ref{eq: Case A0 I equations}, which we treat separately. 
See \ref{comp: A0 Case I (a,b,s)=(0,8,2)} for computations.

\begin{enumerate}
    \item Consider the first root. Here we find
    $$\frac{\beta}{\gamma} = \frac{-165z^{2}+125zw-13w^{2}-330z+16w-165}{27z^{2}-27zw+3w^{2}+54z+27}.$$
    Define the numerator as $\mk b$ and the denominator as $\mk c.$ We show that for any positive $(z,w)$ satisfying the Pell's equation, either $\mk b>0$ or $\mk c>0,$ but not both.

    \begin{claim}
    \label{claim: A0 abs=802 +t}
    Let $(z,w)$ be non-negative integers with $z>3.$ Then
    \begin{enumerate}
        \item if $4z-4 \ge w,$ then the Pell's equation \ref{eq: A0 case I abs=802 pell} has no solutions;
        \item if $7.8z-4 \ge w \ge 4z-4,$ then $\mk b>0, \mk c <0;$
        \item if $8.4z-4 \ge w \ge 7.8z-4,$ then the Pell's equation \ref{eq: A0 case I abs=802 pell} has no solutions;
        \item if $w \ge 8.4z-4,$ then $\mk b < 0, \mk c > 0.$
    \end{enumerate}
    \end{claim}
    \begin{proof}
    We use the same technique for all; introduce a dummy variable $\epsilon$ and use \textsc{MAcaulay2}. See Computation \ref{comp: A0 Case I (a,b,s)=(0,8,2)} for computations.
    \begin{enumerate}
        \item Observe that $0>55z^2-15zw$ shows $w/z>55/15=11/3>3,$ further showing $w \ge 4z > 4z-4.$
        \item Write $w=(7.8-\epsilon)z-4, z=4+x$ for $3.8 \ge \epsilon \ge 0$ and $x \ge 0.$ We then compute $\mk b,\mk c$ as a function of $x,\epsilon$ as
        \begin{align*}
            &\mk b(x,\epsilon)= (-13x^2-104x-208)\epsilon^2+\dots\\
            &\mk c(x,\epsilon)=(3x^2+24x+48)\epsilon^2+\dots 
        \end{align*}
        As a quadratic in $\epsilon,$ $\mk b, \mk c$ have negative, positive leading coefficients (resp.). Thus to show they are positive, negative (resp.), it suffices to check $\mk b>0, \mk c<0$ for $\epsilon=0,3.8.$ We can check
        \begin{align*}
            &\mk b(x,0) = \tfrac{477}{25}x^{2}+\tfrac{6466}{25}x+\tfrac{7307}{25}> 0 \text{ }\forall x \ge 0 \\
            &\mk b(x,3.8) = 127x^{2}+666x+195>0\text{ }\forall x \ge 0 \\
            &\mk c(x,0) = -\tfrac{27}{25}x^{2}-\tfrac{846}{25}x-\tfrac{1077}{25} < 0 \text{ }\forall x \ge 0 \\
            &\mk c(x,3.8) = -33x^{2}-198x-189<0\text{ }\forall x \ge 0
        \end{align*}
        giving the desired result.
        \item Write $z=3+x$ and $w=(7.8+\epsilon)z-4$ for $x \ge 0$ and $0.6>\epsilon>0.$ Then $55z^2-15zw+w^2+1$ as a function of $x,\epsilon$ is
        $$f(x,\epsilon)=(x^2+6x+9)\epsilon^2+\dots.$$
        As a quadratic in $\epsilon,$ this has leading coefficient positive, and hence we only need to show this is negative (for all $x \ge 0)$ for $\epsilon=0,0.6.$ Indeed,
        \begin{align*}
            &f(x,0) = -\tfrac{29}{25}x^{2}-\tfrac{234}{25}x-\tfrac{16}{25} < 0 \text{ }\forall x \ge 0 \\
            &f(x,3.8) = -\tfrac{11}{25}x^{2}-\tfrac{246}{25}x-\tfrac{214}{25}<0\text{ }\forall x \ge 0
        \end{align*}
        giving the desired result.
        \item Write $w=8.4z-4+\epsilon$ for $\epsilon \ge 0$  We then find 
        $$\mk b =-13\epsilon^2+(120-\tfrac{467}{5}z)\epsilon-\tfrac{807}{25}z^{2}+178z-437.$$
        As a quadratic in $\epsilon,$ the coefficients are all negative since $120-467/5z<0$ for $z \ge 2$ and $-\tfrac{807}{25}z^{2}+178z-437$ is always negative (its discriminant and leading term are both negative). 
        Hence $\mk b<0.$
        Furthermore,
        $$\mk c = 3\epsilon^{2}+( \tfrac{117}{5}z-24)\epsilon+ \tfrac{297}{25}z^{2}-\tfrac{198}{5}z+75.$$
        As a quadratic in $\epsilon,$ the coefficients are all positive since $\tfrac{117}{5}z-24>0$ for $z \ge 2$ and $\tfrac{297}{25}z^{2}-\tfrac{198}{5}z+75$ is always positive (its discriminant is negative but leading term positive). 
        Hence $\mk c>0.$  \qedhere
    \end{enumerate} 
    \end{proof}
    We only need to analyze the roots with $z \le 3.$ All such solutions of the Pell's equations are
    $$(z,w) \in \{(1,7), (1,8), (2,13),(2,17)\}.$$
    However, for all these values, $\beta/\gamma < 0$ as one can check. This finishes this case.
    

    \item Consider the second root. Here we find
    \begin{equation}
    \label{eq: A0 case I abs = 802 -t root}
    \begin{split}
        \frac{\alpha}{\gamma} &= \frac{633z^{2}-115zw+5w^{2}+132z-14w+3}{27z^{2}-27zw+3w^{2}+54z+27} \\
        &\frac{\beta}{\gamma} = \frac{-165z^{2}+235zw-23w^{2}-330z+14w-165}{27z^{2}-27zw+3w^{2}+54z+27} \\
        \frac{\delta}{\gamma} &= \frac{210z-14w+42}{9z^{2}-9zw+w^{2}+18z+9}.
    \end{split}
    \end{equation}    
    
    In this case, we show that any positive $(z,w)$ satisfying the Pell's equation in \ref{lemma: pell's equation for A0} will satisfy $\Width_{(9,-1)}(\Delta) < m,$ which would be a contradiction. 
    To calculate this width, we need to prove something first:

    \begin{claim}
    \label{claim: abs=802 root 2 claim 1}
        For any positive $(z,w)$ satisfying the Pell's equation, $w < 9z.$
    \end{claim}
    \begin{proof}
        We prove the contrapositive, write $w=9z+\epsilon.$ Then $55z^2-15zw+w^2+1$ equals
        $$z^2 + 3z\epsilon + \epsilon^2  + 1$$
        which is always nonzero.
    \end{proof}

    Now, observe the side with lattice length $\beta$ has vector $(1,c)^T=(1,9)^T,$ i.e. it is perpendicular to $(9,-1)^T.$
    Further, as $w/z<9$ by Claim \ref{claim: abs=802 root 2 claim 1}, the width along $(9,-1)$ is (also see Figure \ref{fig:width(1,-1)})
    $$\lt \la \alpha\begin{pmatrix}
        1\\0
    \end{pmatrix}, \begin{pmatrix}
        9 \\-1
    \end{pmatrix} \rt \ra=9\alpha.$$
    The condition $\Width_{(9,-1)}(\Delta) \ge m$ then implies
    $$8\lt(\frac{\alpha}{\gamma}\rt) -\frac{\beta}{\gamma}-\frac{\delta}{\gamma}-1 \ge 0,$$
    which with Equation \ref{eq: A0 case I abs = 802 -t root} is the same as
    \begin{equation}
    \label{eq: widthM abs=802 root 2}
        \frac{1734z^{2}-376zw+20w^{2}+234z-28w+12}{9z^{2}-9zw+w^{2}+18z+9} \ge 0.
    \end{equation}
    
    Let the numerator be $\mk b$ and denominator be $\mk c.$

    \begin{claim}
    \label{claim: 10 abs = 802 root 2 claim 4}
        For any positive $(z,w)$ satisfying the Pell's equation, if $\delta/\gamma > 0,$ then $\mk c > 0.$
    \end{claim}
    \begin{proof}
        Since $\delta/\gamma$ in Equation \ref{eq: A0 case I abs = 802 -t root} has the same denominator $\mk c,$
        it suffices to show the numerator is positive, i.e. $210z-14w+42 > 0.$ However, as $0=55z^2-15zw+w^2+1>w^2-15zw$ shows $15z>w,$ we conclude $210z-14w+42=14(15z)-14w+42>0.$
    \end{proof}

    \begin{claim}
    \label{claim: A0 abs=802 -t}
    Let $(z,w)$ be non-negative integers with $z \ge 3.$ Then
    \begin{enumerate}
        \item if $3z \ge w,$ then the Pell's equation \ref{eq: A0 case I abs=802 pell} has no solutions;
        \item if $6.5z \ge w \ge 3z,$ then $\mk c < 0;$
        \item if $8.4z \ge w \ge 6.5z,$ then the Pell's equation \ref{eq: A0 case I abs=802 pell} has no solutions;
        \item if $9z \ge w \ge 8.4z,$ then $\mk b<0;$
        \item if $w \ge 9z,$ then the Pell's equation \ref{eq: A0 case I abs=802 pell} has no solutions.
    \end{enumerate}
    \end{claim}
    \begin{proof}
    We use the same technique for all; introduce a dummy variable $\epsilon$ and use \textsc{MAcaulay2}. See Computation \ref{comp: A0 Case I (a,b,s)=(0,8,2)} for computations.
    \begin{enumerate}
        \item Observe that $0>55z^2-15zw$ shows $w/z>55/15=11/3>3,$ showing $w > 3z.$
        \item Write $w=(3+\epsilon)z$ for $3.5>\epsilon>0,$ and $z=3+x.$ As a function of $x,\epsilon,$ $\mk c$ equals
        $$f(x,\epsilon)=(x^2+6x+9)\epsilon^2+\dots.$$
        As a quadratic in $\epsilon,$ this has positive leading coefficient for all $x>0.$ 
        Thus to show it's negative for all $3.5>\epsilon>0,$ we only need to check that for $\epsilon =0,3.5.$ Indeed, we find
        \begin{align*}
            &f(x,0) = -9x^{2}-36x-18 < 0 \text{ } \forall x \ge 0\\
            &f(x, 3.5) = -\tfrac{29}{4}x^{2}-\tfrac{51}{2}x-\tfrac{9}{4} < 0 \text{ } \forall x \ge 0.
        \end{align*}
        Hence we have our claim.
        \item Write $w=(6.5+\epsilon)z$ for $1.9>\epsilon>0$ and $z=3+x$ for $x>0.$ As a function of $x,\epsilon,$ the expression $55z^2-15zw+w^2+1$ equals
        $$f(x,\epsilon)=(x^2+6x+9)\epsilon^2+\dots.$$
        As a quadratic in $\epsilon,$ this has positive leading coefficient for all $x>0.$ 
        Thus to show it's negative for all $1.9\ge \epsilon\ge0,$ we only need to check that for $\epsilon =0,1.9.$ Indeed, we find
        \begin{align*}
            &f(x,0) = -\tfrac{1}{4}x^{2}-\tfrac{3}{2}x-\tfrac{5}{4} < 0 \text{ } \forall x \ge 0\\
            &f(x, 1.9) = -\tfrac{11}{25}x^{2}-\tfrac{66}{25}x-\tfrac{74}{25} < 0 \text{ } \forall x \ge 0.
        \end{align*}
        Hence we have our claim.
        \item Write $w=(8.4+\epsilon)z$ with $0.6>\epsilon>0,$ and $z=2+x.$ We then find $\mk b$ as a function of $x,\epsilon:$
        $$f(x,\epsilon) = (20x^2+80x+80)\epsilon^2+\dots.$$
        As a quadratic in $\epsilon,$ this has positive leading coefficient for all $x>0.$ 
        Thus to show it's negative for all $0.6>\epsilon>0,$ we only need to check that for $\epsilon =0,0.6.$ Indeed, we find
        \begin{align*}
            &f(x,0) = -\tfrac{66}{5}x^{2}-54x-\tfrac{216}{5} < 0 \text{ } \forall x > 0\\
            &f(x, 0.6) = - 30x^2  - 138x - 144 < 0 \text{ } \forall x > 0.
        \end{align*}
        Hence we have our claim.
        \item This was Claim \ref{claim: abs=802 root 2 claim 1}.  \qedhere
    \end{enumerate} 
    \end{proof}

    Hence it suffices to consider roots of the Pell's equation with $z < 3.$ These are
    $$(z,w) \in \{(1,7),(1,8),(2,13),(2,17)\}.$$
    For the first two roots, we find $\beta/\gamma<0$ (see Computation \ref{comp: A0 Case I (a,b,s)=(0,8,2)}). For the last root, we find $\Width_{(9,-1)}<m.$
    For $(z,w)=(2,13),$ we however find $\Width_{(1,0)}(\Delta)<m,$ exhausting all possibilities.
\end{enumerate}

\subsection{Case II}
\label{subsec: A0 Case II}

We can do an analysis similar to that in Section \ref{subsec: A0 Case I}. We start by translating the combinatorial structure to an algebraic equation. Recall $\mf A^R$ stands for the reverse of a matrix (cf. Def. \ref{lemdef: reverse matrix}).

\begin{lemma}
\label{lemma: A0 case II both equations}
There is a bijection between semi-elliptic (smooth) toric surfaces in Case II and integral matrix solutions $\{\mf M,\mf N\}$ of the system
\begin{equation}
\label{eq: A0 case II eq 1}
    \mf M \begin{pmatrix}
    0 & -1 \\
    1 & 1
\end{pmatrix} \mf N^R =\begin{pmatrix}
    -s & -1 \\
    s+1 & 1
\end{pmatrix} 
\end{equation}
\begin{equation}
    \label{eq: A0 case II eq 2}
        \mf N \begin{pmatrix}
        0&-1\\1&1
    \end{pmatrix}\mf M \begin{pmatrix}
        0&-1\\1&1
    \end{pmatrix} \begin{pmatrix}
        0&-1\\1&2
    \end{pmatrix}^a \begin{pmatrix}
        0&-1\\1&s+5
    \end{pmatrix}\begin{pmatrix}
        0&-1\\1&2
    \end{pmatrix}^b\begin{pmatrix}
        0&-1\\1&1
    \end{pmatrix}=\mf I
    \end{equation}
where $s \in \ZZ_{>0}$ and $a+b=8$ are any non-negative integers.
\end{lemma}
\begin{proof}
Identical to the proof of Lemma \ref{lemma: A0 case I both equations}. The only difference would be to change $\mf A_{i+1}$ to $\mf A_{i+1}^R$ in Equation \ref{eq: A0 case I A_i equation}, see Figure \ref{fig:A0 case II}.
\end{proof}
\begin{figure}[h]
    \centering
    \includegraphics[scale=0.2]{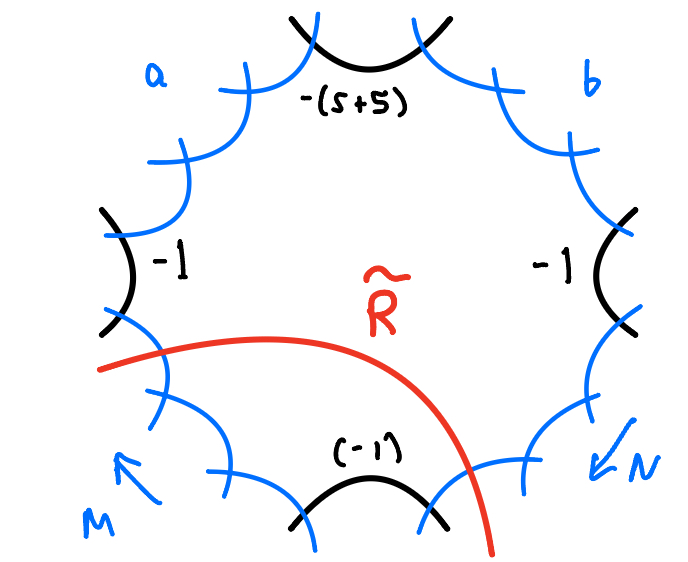}
    \caption{Schematic picture for this case. Note the positions of $a,b$ so that Equation \ref{eq: A0 case II eq 2} holds.}
    \label{fig:A0 case II}
\end{figure}
Write
\begin{equation}
\label{eq: definition of M 2}
    \mf M=\begin{pmatrix}
    x&y\\z&w
\end{pmatrix}.
\end{equation}
Observe here that $x < y<0 < z<w$ and $\det(\mf M)=xw-yz=1$ (cf. Lem. \ref{lem: matrix and chain}).
This time, unlike Case I, there are a finite number of solutions $\{\mf M,\mf N\}$ to the Equations \ref{eq: A0 case II eq 1}, \ref{eq: A0 case II eq 2}. 
However, none of them give an elliptic quadrilateral.
We study these solutions this in the following lemma:

\begin{lemma}
\label{lemma: A0 Case II key lemma}
For each non-negative $(a,b)$ pair with $a+b=8$, solutions $(z,w)$ to the system in Lemma \ref{lemma: A0 case II both equations} lie on a curve in $\mathbb{A}^2$ listed in Table \ref{table:A0 case II}. However, none of them give an elliptic quadrilateral.
\end{lemma}
\begin{proof}
\begin{table}[h]
    \centering
    \begin{tabular}{c|c|c}
    $a$ & $b$ & Equation \\
    \hline
    \hline
    $0$ & $8$ & $31w^{2}+(81s^{4}+810s^{3}+2619s^{2}+2970s+1089)$ \\
    \hline
    $1$ & $7$ & $(-8s^2-24s+31)w^{2}+(256s^{4}+2304s^{3}+6816s^{2}+7344s+2601)$ \\
    \hline
    $2$ & $6$ & $(-14s^{2}-42s+31)w^{2}+(441s^{4}+3822s^{3}+10927s^{2}+11466s+3969)$\\
    \hline
    $3$ & $5$ & $(-18s^{2}-54s+31)w^{2}+(576s^{4}+4896s^{3}+13716s^{2}+14076s+4761)$ \\
    \hline
    $4$ & $4$ & $(-20s^{2}-60s+31)w^{2}+(625s^{4}+5250s^{3}+14475s^{2}+14490s+4761)$ \\
    \hline
    $5$ & $3$ & $(-20s^{2}-60s+31)w^{2}+(576s^{4}+4800s^{3}+13024s^{2}+12600s+3969)$ \\
    \hline
    $6$ & $2$ & $(-18s^{2}-54s+31)w^{2}+(441s^{4}+3654s^{3}+9711s^{2}+8874s+2601)$ \\
    \hline
    $7$ & $1$ & $(-14s^{2}-42s+31)w^{2}+(256s^{4}+2112s^{3}+5412s^{2}+4356s+1089)$ \\
    \hline
    $8$ & $0$ & $(-8s^{2}-24s+31)w^{2}+(81s^{4}+666s^{3}+1531s^{2}+666s+81)$\\ \hline
    \hline
\end{tabular}
    \caption{Quartics in $s,w.$ Note they are all of the form $f(s)w^2+g(s).$}
    \label{table:A0 case II}
\end{table}
From Equation \ref{eq: A0 case II eq 1} and the formula for the reverse of a matrix (see \ref{lemdef: reverse matrix}), can express $\mf N$ in terms of $x,y,z,w,s,$ see Computation \ref{comp: A0 case II}. 
Equation \ref{eq: A0 case II eq 2} then gives $4$ equations in $\{x,y,z,w,a,b,s\}.$ As $a+b=8$ are non-negative, we can test all $9$ possible pairs. 
In each case, we look at the minimal primes of the ideal generated by these $4$ equations (on \textsc{Macaulay2}). 
As shown in Computation \ref{comp: A0 case II}, each ideal contains a quartic in $s,w$ of the form $f(s)w^2+g(s),$ where $\deg f=2, \deg g=4,$ see Table \ref{table:A0 case II}.

\medskip

The case $(a,b)=(0,8)$ clearly has no solutions in positive integers $s,w.$ For the remaining cases, we get the divisibility condition $f(s) \mid g(s).$ 
This is a Diophantine equation in $s,$ which has finitely many solutions unless the polynomial $f$ identically divides $g.$
We use \textsc{Mathematica} to analyze these (see Computation \ref{comp: A0 case II}) and find the only solutions with $s,w>0,$ which happen to be finite. Furthermore, we can use \texttt{minimalPrimes} again to find $z$ in each case. We find the only solutions to be 
$$(a,s,z,w) \in \{(1,1,63,139),(1,2,17,37),(2,1,11,35),(7,1,3,23),(8,1,7,55),(8,2,2,17)\}.$$
These do lead to valid toric surfaces, and so we next check if $\Vol(\Delta)=m^2$ has solutions.
This can be done as before; using Lemmas \ref{lem: vol = m^2 equation},\ref{lemma: value of v3} and Equations \ref{eq: A0 a', b' expression} (expression for $\alpha,\beta$ in terms of $\gamma, \delta, z$ and $w).$ If it does have a solution, we test if $\Width(\Delta) \ge m$ holds, which we find never holds true for our examples (cf. Comp. \ref{comp: A0 case II}).
\end{proof}

\section{The $A_n$ case}
\label{sec: the A_n case}

\subsection{Classification}

The approach and analysis is very similar to that in Section \ref{sec: the smooth case}.
Looking at the maps $Y \to Y_1 \to \dots W$ as a series of blowups of the relevant locus, we 
start with a chain of $n$ $(-2)$ curves, and initially perform some type II blowups (cf. Def. \ref{def: type I, type II blowups}), and then type I blowups (same proof as Lemma \ref{lemma: type II blowups then type I}).    
Hence, we get a similar tree of possibilities as in Section \ref{sec: the smooth case}. 

\medskip

This time, however, we define $s$ as two less than the number of blowups. 
Like in Figure \ref{fig:key tree A0}, there are two branches with reverse chains. 
Here, we only show one of the branches, corresponding to type II blowups at the right endchain (like in Lemma \ref{lem: classification of matrices, -(s+5)}), see Figure \ref{fig: key tree An}.

\begin{figure}[h]
    \centering
    \includegraphics[scale=0.35]{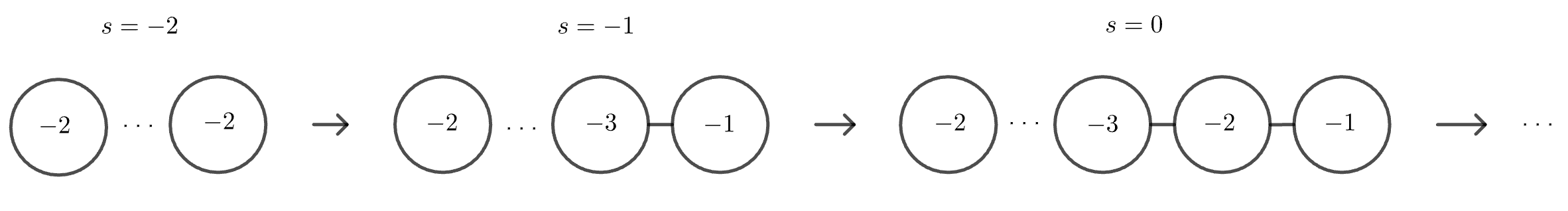}
    \caption{Tree of possibilities}
    \label{fig: key tree An}
\end{figure}

\begin{lemma}
\label{lemma: An matrix classifcation}
    Define $s \ge -2$ as in Figure \ref{fig: key tree An}, i.e. two less than the number of Type II blowups. 
    Up to orientation,  the matrix of the chain is (in our case, the ones shown in Figure \ref{fig: key tree An})
    $$\begin{pmatrix}
        -(ns+3n-1) & -n \\
        (n+1)s+3n+2 & n+1
    \end{pmatrix}.$$
    Furthermore, the heavy curve (cf. Def. \ref{def: heavy curve}) has self-intersection $-(s+5).$
\end{lemma}
\begin{proof}
    Recall that matrices stay invariant for blowups at intersection of curves, see Lemma \ref{lem: blowup at intersection}. So we only need to consider blowups of type II, i.e. at non-intersection points. 

    \medskip

    This formula holds for $s=-2$ since
    $$\begin{pmatrix}
        -(-2n+3n-1) & -n \\
        (n+1)(-2)+3n+2 & n+1
    \end{pmatrix}=\begin{pmatrix}
        0&-1\\
        1 & 2
    \end{pmatrix}^n.$$
    
    Assume $s>-2.$ 
    It is easy to see we obtain a chain of $n-1$ $(-2)$ curves, followed by a $(-3)$ curve, $s+1$ $(-2)$ curves and a $(-1)$ curve. 
    Hence the matrix is
    $$\begin{pmatrix}
        0&-1\\
        1&2
    \end{pmatrix}^{n-1}\begin{pmatrix}
        0&-1\\1&3
    \end{pmatrix}\begin{pmatrix}
        0&-1\\
        1&2
    \end{pmatrix}^{s+1}\begin{pmatrix}
        0&-1\\
        1&1
    \end{pmatrix}$$
    which can be checked to be the desired matrix.

    \medskip

    Let $\mc L_i, \mc S_i$ be the length, sum of self-intersection numbers of our the relevant locus in $Y_i.$ 
    We claim that $\mc S_i+3\mc L_i = s+n+2$ always. 
    Firstly, this holds for a chain of $n$ $(-2)$ curves.
    For a blowup at the intersection of two curves, observe that $\mc S_i$ decreases by $3$ and $\mc L_i$ increases by $1,$ showing $\mc S_i+3\mc L_i$ is invariant.

    \medskip
    
    For a type II blowup, the key observation is that each blowup increases $\mc L_i$ by $1,$ but decreases $\mc S_i$ by only $2.$ Hence $\mc S_i+3\mc L_i$ increases by $1,$ just like $s+n+2$ (which is linear in $s).$ 

    \medskip

    Let $x$ be the self-intersection number of the heavy curve.
    Consider the toric boundary curves of $Y.$
    There are at three $(-1)$ toric boundary curves (cf. discussion in \ref{def: heavy curve}), and two chains of $(-2)$ curves of lengths $a,b$ satisfying $a+b+n=8$ by Lemma \ref{lemma: key commutative diagram X,Y,Z,W, and 8 curves}.  
    Let $N$ be the number of toric boundary divisors in $Y.$
    Furthermore, the length, sum of self-intersection numbers of the relevant
    locus in $Y = Y_0$ is $\mc L_0 - 1, \mc S_0 + 1$ respectively, since $\wt R$ is a $(-1)$ curve in the relevant locus which is not a toric boundary curve.
    By Lemma \ref{lemma: sum of smooth numbers}, the sum of self-intersection numbers of toric boundary divisors of $Y$ is $12-3N.$ Hence
    \begin{align*}
        &x+3(-1)+(a+b)(-2)+(\mc S_0+1) = 12-3(4+a+b+\mc L_0-1) \\
        &\implies x = -(a+b)-\mc S_0-3\mc L_0+5 = (n-8)-(s+n+2)+5=-(s+5).
    \end{align*}
    and we are done.
\end{proof}

\begin{remark}
\label{remark: -(s+5) so can repeat results}
Since we get the same self-intersection number $-(s+5)$ in Lemma \ref{lemma: An matrix classifcation} as we did in the smooth case (cf. Lem. \ref{lem: classification of matrices, -(s+5)}), we can apply most of the results we did in that section by treating $s$ purely as an unknown variable, discarding it's geometric interpretation.
Also see Remark \ref{remark: -(s+5) also for An}.
\end{remark}

We now have the setup to deal with Cases I and II from Figure \ref{fig:3 cases}. Both of them have a different analysis, however the key approach is the same; we attack them algebraically using the matrix equation \ref{lem: matrix equation} and $\Vol(\Delta)=m^2.$ 
The approaches and results in both cases are very similar to corresponding cases in Section \ref{sec: the smooth case}.

\subsection{Case I} 

The same analysis that we did in the proof of Lemma \ref{lemma: A0 case I both equations} combined with Lemma \ref{lemma: An matrix classifcation} implies the following. (Also see Remark \ref{remark: -(s+5) so can repeat results} on we can reuse the results, and Remark \ref{remark:scaling} on scaling).

\begin{lemma}
\label{lemma: An case I both equations}
There is a bijection between semi-elliptic (smooth) toric surfaces in Case I and integral matrix solutions $\{\mf M,\mf N\}$ of the system
\begin{equation}
\label{eq: An case I eq 1}
    \mf M \begin{pmatrix}
    0 & -1 \\
    1 & 1
\end{pmatrix} \mf N =\begin{pmatrix}
        -(ns+3n-1) & -n \\
        (n+1)s+3n+2 & n+1
    \end{pmatrix}
\end{equation}
\begin{equation}
    \label{eq: An case I eq 2}
        \mf N \begin{pmatrix}
        0&-1\\1&1
    \end{pmatrix}\mf M \begin{pmatrix}
        0&-1\\1&1
    \end{pmatrix} \begin{pmatrix}
        0&-1\\1&2
    \end{pmatrix}^a \begin{pmatrix}
        0&-1\\1&s+5
    \end{pmatrix}\begin{pmatrix}
        0&-1\\1&2
    \end{pmatrix}^b\begin{pmatrix}
        0&-1\\1&1
    \end{pmatrix}=\mf I
\end{equation}
where $s \in \ZZ_{\ge -2}$ and $a+b=8-n$ are any non-negative integers.
\end{lemma}

Write 
\begin{equation}
\label{eq: An M matrix}
    \mf M = \begin{pmatrix}
    x&y\\
    z&w
\end{pmatrix}
\end{equation}
with $x<y<0<z<w$ and $\det(M)=xw-yz=1$ (cf. Lem. \ref{lem: matrix and chain}).
A similar computation as that done in the proof of Lemma \ref{lemma: pell's equation for A0} shows there are a finite possible values for $(s,n),$ however, each accompanied with a Pell's equation in $z,w.$ 
If we further add the algebraic constraint $\Vol(\Delta)=m^2,$ we get a pair of equations in each case.

\begin{lemma}
    \label{lemma: pell's equation for An}
    There are a finite number of possibilities for the triple $(n,s,a),$ listed in Table \ref{tab:An case I}. Further, any point $(z,w)$ lies on the intersection of two quadrics in $\mathbb{A}^3_{z,w,t}.$ 
    \begin{table}[h]
        \footnotesize
        \centering
        \begin{tabular}{c|c|c|c|c}
            $n$ & $s$ & $a$ & Pell's Equation&$\Vol(\Delta)=m^2$\\
            \hline
            \hline
            $1$ & $0$ & $0$ & $33z^{2}-54zw+22w^{2}+2$&$1056z^{2}-4896zw+3104w^{2}+2112z-1536w+1056+t^2$ \\
            \hline
            $1$ & $0$ & $7$ & $33z^{2}-12zw+w^{2}+2$ & $1056z^{2}-3552zw+416w^{2}+2112z-192w+1056+t^2$ \\
            \hline
            \hline
            $2$ & $-1$ & $0$ & $22z^{2}-34zw+13w^{2}+3$&$616z^{2}-2184zw+1204w^{2}+1232z-784w+616+t^2$ \\
            \hline
            $2$ & $-1$ & $6$ & $22z^{2}-10zw+w^{2}+3$&$616z^{2}-1512zw+196w^{2}+1232z-112w+616+t^2$ \\
            \hline
            \hline
            $3$ & $4$ & $1$ & $77z^{2}-98zw+29w^{2}+4$&$3080z^{2}-25480zw+11240w^{2}+6160z-2800w+3080+t^2$ \\
            \hline
            $3$ & $4$ & $4$ & $77z^{2}-56zw+8w^{2}+4$&$3080z^{2}-23800zw+4520w^{2}+6160z-1120w+3080+t^2$ \\
            \hline
            \hline
            $4$ & $-2$ & $0$ & $11z^2- 15zw + 5w^2  + 5$&$220z^{2}-520zw+220w^{2}+440z-200w+220+t^2$ \\
            \hline
            $4$ & $-2$ & $4$ & $11z^2 - 7zw + w^2  + 5$&$220z^{2}-360zw+60w^{2}+440z-40w+220+t^2$ \\
            \hline
            \hline
            $4$ & $-1$ & $1$ & $22z^{2}-26zw+7w^{2}+5$&$704z^{2}-2240zw+800w^{2}+1408z-512w+704+t^2$ \\
            \hline
            $4$ & $-1$ & $3$ & $22z^{2}-18zw+3w^{2}+5$&$704z^{2}-1984zw+416w^{2}+1408z-256w+704+t^2$ \\
            \hline
            \hline
            $4$ & $0$ & $2$ & $33z^{2}-33zw+7w^{2}+5$&$1188z^{2}-4752zw+1332w^{2}+2376z-648w+1188+t^2$ \\
            \hline
            \hline
        \end{tabular}
        \caption{List of possible $(n,s,a)$ and pairs of quadrics that $(z,w)$ lie on.}
        \label{tab:An case I}
    \end{table}
    
     Every intersection is a curve of geometric genus $0.$ The four curves for $(n,s) \in \{(1,0),(4,-2)\}$ are a union of two conics, each an intersection of a hyperplane and a quadric. Every other curve is irreducible but singular.
\end{lemma}

\begin{proof}
The Pell's equation is derived using the same method as Lemma \ref{lemma: pell's equation for A0}, see Computation \ref{comp: An case I general} for details. For $\Vol(\Delta)=m^2,$ firstly observe that the result of Lemma \ref{lemma: value of v3} holds here as well, as mentioned in Remark \ref{remark: -(s+5) also for An}.
Hence we can again express $\beta,\alpha$ in terms of $\gamma,\delta$ as in Equation \ref{eq: A0 a', b' expression}. 
Then $\Vol(\Delta)=m^2$ gives a quadratic equation in $\delta/\gamma,$
whose discriminant should be a square, say $t^2.$ This gives the second equation (cf. Comp. \ref{comp: An case I general}).

\medskip

Finally, we use \texttt{minimalPrimes} on \textsc{Macaulay2} to find the minimal primes of the ideal generated by \texttt{disc=t}$^2$ and the Pell's equation. 
In the $4$ special cases when $(n,s) \in \{(1,0),(4,-2)\}$, we find a product of two ideals, both given by the intersection of a hyperplane and quadric in $\mathbb{A}^3_{z,w,t};$ see Computation \ref{comp: An case I general}. In all other cases, we find an irreducible singular curve of geometric genus $0,$ as is checked on \textsc{Magma}. 
We remark here that since the intersection of two quadrics in $\PP^3$ has arithmetic genus $1,$ hence there is just one singular point.
\end{proof}

Here, it is fascinating how every curve we get has genus zero, which forms a (Zariski) closed subset of the moduli space of all quadric intersections in $\mathbb{A}^3_{z,w,t}$. 

\medskip

We now show that all cases in Lemma \ref{lemma: pell's equation for An} outside $(n,s) \in \{(1,0),(4,-2)\}$ aren't locally soluble, and hence can be ruled out.

\begin{lemma}
\label{lemma: An cases not locally soluble}
Consider the cases from Lemma \ref{lemma: pell's equation for An}. There are no integer points on the intersection of quadrics 
\begin{enumerate}
    \item modulo $13$ when $(n,s)=(2,-1);$ 
    \item modulo $7$ when $(n,s) \in \{(3,4),(4,-1)\};$ 
    \item modulo $17$ when $(n,s) = (4,0).$ 
\end{enumerate}
\end{lemma}
\begin{proof}
This is readily checked on \textsc{Mathematica}; see Computation \ref{comp: An case I singular}.
\end{proof}

We are left with the cases when $(n,s) \in \{(1,0),(4,-2)\}.$ 
Since the intersection is a union of two conics, both being the intersection of a hyperplane with a quadric, they have infinitely many integer points. 
However, we show that in each case, the quadrilateral either does not exist, or is not elliptic; a method similar to what we did in Section \ref{subsec: A0 Case I}.



\subsubsection*{Case $(n,s,a)=(1,0,0)$}
In this case, the Pell's equation is
\begin{equation}
\label{eq: An nsa=100 pell}
    33z^{2}-54zw+22w^{2}+2 = 0.
\end{equation}
Here, we find $t = \pm (132z - 96w - 8)$ as seen in Computation \ref{comp: An case I general}. As seen before, we find two roots for $\gamma/\delta$ in terms of $z,w.$
\begin{enumerate}
    \item Consider the first case. Here we find (cf. Comp. \ref{comp: An case I degenerate})
    \begin{equation}
    \label{eq: An nsa = 100}
        \frac{\beta}{\delta}=\frac{-33z^{2}-9zw+23w^{2}-66z+28w-33}{z^{2}-zw+w^{2}+2z+1}.
    \end{equation}

    However, we show that $\beta/\gamma <0$ in this case, which would be our desired contradiction.

    \begin{claim}
     \label{claim: An nsa = 100}
        Let $(z,w)$ be non-negative integers. Then
        \begin{enumerate}
            \item if $1.35z > w,$ then $\beta/\gamma > 0;$
            \item if $w \ge 1.35z,$ then there are no solutions to the Pell's equation \ref{eq: An nsa=100 pell}.
        \end{enumerate}
    \end{claim}
    \begin{proof}
    \begin{enumerate}
        \item As $w>z$ (see the comment below Equation \ref{eq: An M matrix}), we find the denominator in Equation \ref{eq: An nsa = 100} is $w(w-z)+(z+1)^2>0.$ Hence it suffices to show the numerator is negative. 
        Write $z = 1/1.35w+\epsilon$ for $\epsilon>0.$ Then we find the numerator equals (cf. Comp. \ref{comp: An case I degenerate})
        $$-\tfrac{431}{243}w^{2}-\tfrac{521}{9}w\epsilon-33\epsilon^{2}-\tfrac{188}{9}w-66\epsilon-33 < 0 \; \forall \epsilon>0$$
        giving the desired result.
        \item We prove that if $w > 1.35z,$ then $ 33z^{2}-54zw+22w^{2}+2>0.$ Indeed, write $z=1/1.35w-\epsilon.$ The expression then becomes (cf. Comp. \ref{comp: An case I degenerate})
        $$\tfrac{26}{243}w^{2}+\tfrac{46}{9}w\epsilon+33\epsilon^{2}+2>0 \; \forall \epsilon>0$$
        proving the desired claim. \qedhere
    \end{enumerate} 
    \end{proof}
    

    \item Consider the second case. Here, we find (cf. Comp. \ref{comp: An case I degenerate})
    \begin{equation}
    \label{eq: An nsa = 100 -t}
    \begin{split}
        \frac{\alpha}{\gamma} &= \frac{98z^{2}-138zw+48w^{2}+36z-20w+8}{z^{2}-zw+w^{2}+2z+1} \\
        &\frac{\beta}{\gamma} = \frac{-33z^{2}+123zw-73w^{2}-66z+20w-33}{z^{2}-zw+w^{2}+2z+1} \\
        \frac{\delta}{\gamma} &= \frac{90z-40w+20}{z^{2}-zw+w^{2}+2z+1}.
        \end{split}
    \end{equation}
    In this case we show either $\alpha/\gamma < 0$ or $\Width_{(1,-1)}<m,$ just like in the smooth case (cf. the case with Equations \ref{eq: A0 case I sides -t}). 
    The same analysis as in the smooth case works (cf. Fig. \ref{fig:width(1,-1)}), and so $\Width_{(1,-1)} \ge m$ translates to
    $$(w-z-1)\lt(\frac{\delta}{\gamma}\rt)- \frac{\beta}{\gamma}-1 \ge 0$$
    like in Equation \ref{eq: A0 case I -t width - m}. In our case, Equations \ref{eq: An nsa = 100 -t} translates this to (cf. Comp. \ref{comp: An case I degenerate})
    \begin{equation}
        \label{eq: An nsa = 100 width-m} \frac{-58z^{2}+8zw+32w^{2}-46z+40w+12}{z^{2}-zw+w^{2}+2z+1} \ge 0.
    \end{equation}

    \begin{claim}
    \label{claim: An nsa = 100 -t}
    Let $(z,w)$ be non-negative integers with $z \ge 37,w \ge 4.$ Then
    \begin{enumerate}
        \item if $1.2z>w,$ then $\Width_{(1,-1)}(\Delta)<m;$
        \item if $1.3z \ge w \ge 1.2z,$ then there are no solutions to the Pell's equation \ref{eq: An nsa=100 pell};
        \item if $1.35z \ge w \ge 1.3z,$ then $\alpha/\gamma < 0;$
        \item if $w>1.35z,$ then there are no solutions to the Pell's equation \ref{eq: An nsa=100 pell}.
    \end{enumerate}
    \end{claim}
    \begin{proof}
    We use the same technique for all; introduce a dummy variable $\epsilon$ and use \textsc{Macaulay2}. See Computation \ref{comp: An case I degenerate} for computations.
        \begin{enumerate}
            \item The denominator of Equation \ref{eq: An nsa = 100 width-m} is positive as $w>z.$ So consider the numerator.
            Write $z=10/12w +\epsilon$ and $w=4+x$ for $\epsilon, x \ge 0.$ Then the numerator is 
            $$\tfrac{1}{18}(-29x^{2}-1596x\epsilon-1044\epsilon^{2}-202x-7212\epsilon-128) < 0\; \forall x,\epsilon>0$$
            as desired.

            \item Write $w=(1.2+\epsilon)z$ and $z=11+x$ for $0.1>\epsilon>0$ and $x \ge 0$ (as $z \ge 11$ is sufficient here). As a function of $x,\epsilon,$ the expression $33z^2-54zw+22w^2+2$ becomes 
            $$f(x,\epsilon)=(22x^2+484x+2662)\epsilon^2+\dots.$$
            As a quadratic in $\epsilon,$ this has positive leading coefficient. Thus to show it's negative for all $0.1>\epsilon>0,$ we only need to check that $\epsilon=0,0.1.$ Indeed, we find
            \begin{align*}
            &f(x,0) = -\tfrac{3}{25}x^{2}-\tfrac{66}{25}x-\tfrac{313}{25}<0 \; \forall x\ge0\\
            &f(x,0.1) = -\tfrac{1}{50}x^{2}-\tfrac{11}{25}x-\tfrac{21}{50}<0 \; \forall x\ge0
            \end{align*}
            as desired.
            
            \item The denominator of $\alpha/\gamma$ in Equation \ref{eq: An nsa = 100 -t} is positive as $w>z.$ So consider the numerator.
            Write $w=(1.3+\epsilon)z$ and $z=37+x$ for $0.05>\epsilon>0$ and $x \ge 0.$ As a function of $x,\epsilon,$ we find the numerator equals
            $$f(x,\epsilon)=(48x^{2}+3552x+65712)\epsilon^{2}+\dots.$$
            As a quadratic in $\epsilon,$ this has positive leading coefficient. Thus to show it's negative for all $0.05>\epsilon>0,$ we only need to check that $\epsilon=0,0.05.$ Indeed, we find
            \begin{align*}
                &f(x,0) = \tfrac{1}{25}(-7x^{2}-268x-133)<0 \; \forall x\ge0\\
                &f(x,0.05) = \tfrac{1}{50}(-41x^{2}-2584x-39079)<0 \; \forall x\ge0.
            \end{align*}
            Hence we have our claim.
            \item This was $(b)$ in Claim \ref{claim: An nsa = 100}.  \qedhere
        \end{enumerate}
    \end{proof}

    We are thus left to analyze the cases when $z \le 36$ or $w \le 4.$ The only possible solutions to the Pell's equation \ref{eq: An nsa=100 pell} in this range are
    $$(z,w) \in \{(4,5),(6,7),(10,13),(20,23),(36,47)\}.$$
    As shown in Computation \ref{comp: An case I degenerate}, none of these work:
    for $(z,w) \in \{(4,5),(10,13)\},$ we find $\beta/\gamma < 0.$ For $(z,w) \in \{(6,7),(20,23)\},$ we find $\Width_{(1,-1)}(\Delta)<m.$ lastly, for $9z,w)=(36,47),$ we find $\alpha/\gamma < 0.$ This exhausts all the possibilities in this case.   
\end{enumerate}


\subsubsection*{Case $(n,s,a)=(1,0,7)$} In this case, the Pell's equation is
\begin{equation}
\label{eq: An nsa=107 pell}
    33z^{2}-12zw+w^{2}+2 = 0.
\end{equation}
Here $t=\pm (132z-12w-8)$ as seen in Computation \ref{comp: An case I general}. As seen before, we find two roots for $\gamma/\delta$ in terms of $z,w.$
\begin{enumerate}
    \item Consider the first case. Here, we find (cf. Comp. \ref{comp: An case I degenerate})
    \begin{equation}
    \label{eq: An nsa = 107 +t}
        \frac{\beta}{\gamma} = \frac{-132z^{2}+111zw-13w^{2}-264z+14w-132}{32z^{2}-32zw+4w^{2}+64z+32}.
    \end{equation}
    Call the numerator $\mk b$ and denominator $\mk c$ in Equation \ref{eq: An nsa = 107 +t}. We prove that $\beta/\gamma<0$ for large enough $(z,w)$ satisfying the Pell's equation. We do this by studying individual regions in $\RR^2_{z,w}:$
    \begin{claim}
    \label{claim: An nsa=107 +t}
        Assume $(z,w)$ are non-negative integers with $z \ge 3.$ Then
        \begin{enumerate}
            \item if $3z \ge w,$ then there are no solutions to the Pell's equation \ref{eq: An nsa=107 pell};
            \item if $5z \ge w \ge 3z,$ then $\mk b > 0$ and $\mk c < 0;$
            \item if $7.5z \ge w \ge 5z,$ then there are no solutions to the Pell's equation \ref{eq: An nsa=107 pell};
            \item if $w \ge 7.5z,$ then $\mk b<0$ and $\mk c > 0.$
        \end{enumerate}
    \end{claim}
    \begin{proof}
    We use the same technique for all; introduce a dummy variable $\epsilon$ and use \textsc{Macaulay2}. See Computation \ref{comp: An case I degenerate} for computations.
        \begin{enumerate}
            \item Write $w=3z-\epsilon$ for $\epsilon>0.$ Then $33z^{2}-12zw+w^{2}+2$ equals
            $$6z^2  + 6z\epsilon + \epsilon^2  + 2>0 \; \forall \epsilon>0$$
            as desired.
            \item Write $w=(3+\epsilon)z$ and $z=3+x$ for $2 \ge \epsilon \ge 0$ and $x \ge 0.$ Then as a function of $x,\epsilon,$
            \begin{align*}
                &\mk b(x,\epsilon) = (-13x^2-78x-117)\epsilon^2+\dots \\
                &\mk c(x,\epsilon) = (4x^2+24x+36)\epsilon^2+\dots.
            \end{align*}
            As a quadratic in $\epsilon,$ the leading coefficients of $\mk b,\mc c$ are negative, positive. So it suffices to check the conclusions at $\epsilon \in \{0,2\}.$ Indeed,
            \begin{align*}
            &\mk b(x,0) = 84x^{2}+282x-42>0 \; \forall x\ge0 \\ 
            &\mk b(x, 2) = -98x^{2}+394x+168>0 \; \forall x\ge0 \\ 
            &\mk c(x,0) = - 28x^2  - 104x - 28<0 \; \forall x\ge0\\
            &\mk c(x,2) = - 28x^2  - 104x - 28 <0 \; \forall x\ge0.
            \end{align*}
            \item Write $w=(5+\epsilon)z$ and $z=3+x$ for $2.5 \ge \epsilon \ge 0$ and $x \ge 0.$ Then $33z^{2}-12zw+w^{2}+2$ as a function of $x,\epsilon$ equals
            $$f(x,\epsilon)=(x^2+6x+9)\epsilon^2+\dots.$$
            We show this is negative in our range of values.
            As a quadratic in $\epsilon,$ the leading coefficient is positive. So we only need to check $\epsilon \in \{0,2.5\}.$ Indeed,
            \begin{align*}
            &\mk b(x,0) = - 2x^2  - 12x - 16<0 \; \forall x\ge0 \\ 
            &\mk b(x, 2) = -\tfrac{3}{4}x^{2}-\tfrac{9}{2}x-\tfrac{19}{4}<0 \; \forall x\ge0 \\ 
            \end{align*}
            as desired.
            \item Write $w=7.5z+\epsilon$ for $\epsilon\ge0.$ Then
            \begin{align*}
                &\mk b = -\tfrac{123}{4}z^{2}-84z\epsilon-13\epsilon^{2}-159z+14\epsilon-132 < 0 \; \forall \epsilon\ge0 \\
                &\mk c =17z^2  + 28z\epsilon + 4\epsilon^2  + 64z + 32 > 0 \; \forall \epsilon\ge0.
            \end{align*}
            where the former follows since $-13\epsilon^{2}+14\epsilon-132 < 0$ for all $\epsilon$ as its discriminant $(-6668)$ and leading coefficient are both negative.  \qedhere
        \end{enumerate} 
    \end{proof}
    We are thus left to analyze the case when $z \le 2.$ The only possible solutions to the Pell's equation \ref{eq: An nsa=107 pell} in this range are
    $$(z,w) \in \{(1,5),(1,7)\}.$$
    However, one can check $\beta/\gamma<0$ for both of these, exhausting this case.


    \item Consider the second case. Here we find (cf. Comp. \ref{comp: An case I degenerate})
    \begin{equation}
    \label{eq: An nsa = 107 -t}
        \begin{split}
            \frac{\alpha}{\gamma} &= \frac{364z^{2}-69zw+3w^{2}+88z-10w+4}{32z^{2}-32zw+4w^{2}+64z+32} \\
            &\frac{\beta}{\gamma} = \frac{-132z^{2}+177zw-19w^{2}-264z+10w-132}{32z^{2}-32zw+4w^{2}+64z+32} \\
            \frac{\delta}{\gamma} &= \frac{90z-5w+20}{8z^{2}-8zw+w^{2}+16z+8}.
        \end{split}
    \end{equation}

    \begin{claim}
    \label{claim: An nsa = 107 -t w<8z}
        Any $(z,w)$ satisfying the Pell's equation \ref{eq: An nsa=107 pell} satisfy $w<8z.$
    \end{claim}
    \begin{proof}
        We prove the contrapositive; write $w=8z+\epsilon$ with $\epsilon>0.$ Then the expression $33z^2-12zw+w^2+2$ equals
        $$z^2 + 4z\epsilon + \epsilon^2 + 2 > 0 \; \forall \epsilon>0$$
        as desired.
    \end{proof}

    Now, observe the side with lattice length $\beta$ has vector $(1,c)^T=(1,8)^T,$ i.e. it is perpendicular to $(8,-1)^T.$
    Further, as $w/z<8$ by Claim \ref{claim: An nsa = 107 -t w<8z}, the width along $(8,-1)$ is (also see Figure \ref{fig:width(1,-1)})
    $$\lt \la \alpha\begin{pmatrix}
        1\\0
    \end{pmatrix}, \begin{pmatrix}
        8 \\-1
    \end{pmatrix} \rt \ra=8\alpha.$$
    The condition $\Width_{(8,-1)}(\Delta) \ge m$ then implies
    $$7\lt(\frac{\alpha}{\gamma}\rt)-\frac{\beta}{\gamma}-\frac{\delta}{\gamma}-1 \ge 0,$$
    which by Equations \ref{eq: An nsa = 107 -t} is
    \begin{equation}
    \label{eq: An nsa = 107 -t width-m}
    \frac{662z^{2}-157zw+9w^{2}+114z-15w+12}{8z^{2}-8zw+w^{2}+16z+8} \ge 0.
    \end{equation}
    Let $\mk b, \mk c$ be the numerator, denominator in Equation \ref{eq: An nsa = 107 -t width-m}. Observe that $\mk c$ is the same denominator we have in Equations \ref{eq: An nsa = 107 -t} up to a constant. Hence we can prove the following:

    \begin{claim}
    \label{claim: An nsa = 107 mk c > 0}
        If $\delta/\gamma > 0,$ then $\mk c > 0.$
    \end{claim}
    \begin{proof}
    It suffices to show the numerator, i.e. $90z-5w+20,$ is positive. As $(z,w)$ satisfy the Pell's equation \ref{eq: An nsa=107 pell}, Claim \ref{claim: An nsa = 107 -t w<8z} shows $90z>88z>11w>5w>5w-20.$
    \end{proof}

    Claim \ref{claim: An nsa = 107 mk c > 0} and Equation \ref{eq: An nsa = 107 -t width-m} then imply $\mk b>0$ too. 
    However, we show this is impossible by studying individual regions in $\RR^2_{z,w}:$
    
    \begin{claim}
    \label{claim: An nsa=107 -t} 
    Let $(z,w)$ be non-negative integers with $z\ge 3.$ Then
        \begin{enumerate}
            \item if $3z \ge w,$ then there are no solutions to the Pell's equation \ref{eq: An nsa=107 pell};
            \item if $5z \ge w \ge 3z,$ then $\mk c < 0;$
            \item if $7.5z \ge w \ge 5z,$ then there are no solutions to the Pell's equation \ref{eq: An nsa=107 pell};
            \item if $8z \ge w \ge 7.5z,$ then $\mk b < 0;$
            \item if $w \ge 8z,$ then there are no solutions to the Pell's equation \ref{eq: An nsa=107 pell}.
        \end{enumerate}
    \end{claim}
    \begin{proof}
    We use the same technique for all; introduce a dummy variable $\epsilon$ and use \textsc{Macaulay2}. See Computation \ref{comp: An case I degenerate} for computations.
    \begin{enumerate}
        \item This was $(a)$ in Claim \ref{claim: An nsa=107 +t}.
        \item Write $w=(3+\epsilon)z$ and $z=3+x$ for $2 \ge \epsilon \ge 0$ and $x \ge 0.$ Then $\mk c$ as a function of $x,\epsilon$ equals
        $$\mk c(x,\epsilon)=(x^2+6x+9)\epsilon^2+\dots.$$
        As a quadratic in $\epsilon$, this has positive leading coefficient. Hence we need to check the result only for $\epsilon \in \{0,2\}.$ Indeed,
            \begin{align*}
            &\mk c(x,0) = \mk c(x,2) = - 7x^2  - 26x - 7<0 \; \forall x\ge0
            \end{align*}
            as desired.
        \item This was $(c)$ in Claim \ref{claim: An nsa=107 +t}.
        \item Write $w=(7.5+\epsilon)z$ and $z=3+x$ with $1.5 \ge \epsilon \ge 0$ and $x \ge 0.$ Then $\mk b$ as a function of $x,\epsilon$ equals
        $$\mk b(x,\epsilon) = (9x^2+54x+81)\epsilon^2+\dots.$$
        As a quadratic in $\epsilon$, this has negative leading coefficient. Hence we need to check the result only for $\epsilon \in \{0,1.5\}.$ Indeed,
            \begin{align*}
            &\mk b(x,0) = -\tfrac{37}{4}x^{2}-54x-\tfrac{267}{4}<0 \; \forall x\ge0 \\
            &\mk b(x,0.5) = -18x^{2}-114x-168<0 \; \forall x\ge0
            \end{align*}
            as desired.
        \item This was proven in Claim \ref{claim: An nsa = 107 -t w<8z}.  \qedhere
    \end{enumerate} 
    \end{proof}

    We are thus left to analyze the case when $z \le 2.$ The only possible solutions to the Pell's equation \ref{eq: An nsa=107 pell} in this range are
    $$(z,w) \in \{(1,5),(1,7)\}.$$
    However, one can check $\beta/\gamma<0$ for both of these, exhausting this case.
\end{enumerate}



\subsubsection*{Case $(n,s,a)=(4,-2,0)$}
In this case, the Pell's equation is
\begin{equation}
\label{eq: An nsa=4-20 pell}
11z^2-15zw+5w^2+5 = 0.
\end{equation}
Here, $t = \pm(22z - 10w - 10)$ as seen in Computation \ref{comp: An case I general}.
As seen before, we find two roots for $\gamma/\delta$ in terms of $z,w.$
\begin{enumerate}
    \item Consider the first case. We find (cf. Comp. \ref{comp: An case I degenerate})
    \begin{equation}
    \label{eq: An nsa=4-20}
    \frac{\beta}{\delta} = \frac{-11z^{2}+5zw-w^{2}-22z+10w-11}{w(w-z)+(z+1)^2}.
    \end{equation}
    We show that for $(z,w)$ satisfying the Pell's equation \ref{eq: An nsa=4-20 pell}, $\beta/\gamma<0.$ We can in fact prove this for non-negative reals $z,w,$ not just integers.

    \begin{claim}
    Let $(z,w)$ be any non-negative real numbers.
    \begin{enumerate}
        \item if $z < 3,$ then there are no solutions to the Pell's equation \ref{eq: An nsa=4-20 pell};
        \item if $z \ge 3$ then $\beta/\gamma < 0.$
    \end{enumerate}
    \begin{proof}
    \begin{enumerate}
        \item As a quadratic in $w,$ the discriminant is $(15z)^2-4(5)(11z^2+5)=5z^2-100 < 0$ for $z \le 4.$
        \item The denominator in Equation \ref{eq: An nsa=4-20} is always positive as $w>z$ (see the comment below Equation \ref{eq: An M matrix}). So consider the numerator, which in fact is an ellipse. As a quadratic in $w,$ this is $-w^2+5(z+2)w-11(z+1)^2.$
        The discriminant is $25(z+2)^2-44(z+1)^2=- 19z^2 + 12z + 56.$ We can easily check this is negative for all $z \ge 3.$  \qedhere
    \end{enumerate} 
    \end{proof}
    \end{claim}


    \item Consider the second case. We find (cf. Comp. \ref{comp: An case I degenerate})
    \begin{equation}
    \label{eq: An nsa=4-20 -t}
    \begin{split}
        &\frac{\alpha}{\gamma} = \frac{21z^{2}-21zw+5w^{2}+10z+5}{z^{2}-zw+w^{2}+2z+1} \\
        \frac{\beta}{\gamma} &= \frac{-11z^{2}+27zw-11w^{2}-22z-11}{z^{2}-zw+w^{2}+2z+1} \\
        &\frac{\delta}{\gamma} = \frac{16z}{z^{2}-zw+w^{2}+2z+1}.
    \end{split}
    \end{equation}
    In this case we show either $\alpha/\gamma < 0$ or $\Width_{(1,-1)}<m,$ just like in the smooth case (cf. the case with Equations \ref{eq: A0 case I sides -t}). 
    The same analysis as in the smooth case works (cf. Fig. \ref{fig:width(1,-1)}), and so $\Width_{(1,-1)}-m \ge 0$ translates to
    $$(w-z-1)\lt(\frac{\delta}{\gamma}\rt)- \frac{\beta}{\gamma}-1 \ge 0$$
    like in Equation \ref{eq: A0 case I -t width - m}. In our case, Equation \ref{eq: An nsa=4-20 -t} translates this to
    \begin{equation}
        \label{eq: An nsa=4-20 -t width-m}
        \frac{-6z^{2}-10zw+10w^{2}+4z+10}{z^{2}-zw+w^{2}+2z+1} \ge 0.
    \end{equation}

    \begin{claim}
    \label{claim: An nsa=4-20 -t}
        Assume $(z,w)$ are non-negative integers with $w > z$ (see the comment below Equation \ref{eq: An M matrix}) and $z \ge 31$ Then
        \begin{enumerate}
            \item if $1.3z \ge w \ge z,$ then $\Width_{(1,-1)}<m;$
            \item if $1.72z \ge w \ge 1.3z,$ then there are no solutions to the Pell's equation \ref{eq: An nsa=4-20 pell};
            \item if $2z \ge w \ge 1.72z,$ then here $\alpha/\gamma < 0;$
            \item if $w \ge 2z,$ then are no solutions to the Pell's equation \ref{eq: An nsa=4-20 pell}.
        \end{enumerate}
    \end{claim}
    \begin{proof}
     We use the same technique for all; introduce a dummy variable $\epsilon$ and use \textsc{Macaulay2}. See Computation \ref{comp: An case I degenerate} for computations.
     \begin{enumerate}
         \item As $w>z,$ the denominator of Equation \ref{eq: An nsa=4-20 -t width-m} is positive, and so consider the numerator, say $\mk b$.
         Write $w=(1+\epsilon)z$ and $z=31+x$ where $0.3 \ge \epsilon \ge 0$ and $x \ge 0.$ Then $\mk b$ as a function of $x,\epsilon$ is
         $$\mk b(x,\epsilon) = (10x^2+620x+9610)\epsilon^2+\dots$$
         As a quadratic in $\epsilon,$ this has positive leading coefficient. So we only need to check this is negative for $\epsilon \in \{0,0.3\}.$ Indeed,
         \begin{align*}
             &\mk b(x,0) = -6x^{2}-368x-5632 < 0 \; \forall x \ge 0 \\
             &\mk b(x,0.3) -\tfrac{21}{10}x^{2}-\tfrac{631}{5}x-\tfrac{18841}{10} < 0 \; \forall x \ge 0
         \end{align*}
         as desired.
         \item Write $w=(1+\epsilon)z$ and $z=31+x$ where $0.3 \ge \epsilon \ge 0$ and $x \ge 0.$ Then $11z^2-15zw+5w^2+5$ as a function of $x,\epsilon$ is
         $$f(x,\epsilon) = (5x^2+310x+4805)\epsilon^2+\dots.$$
         As a quadratic in $\epsilon,$ this has positive leading coefficient. So we only need to check this is negative for $\epsilon \in \{0,0.42\}.$ Indeed,
         \begin{align*}
             &\mk b(x,0) = -\tfrac{1}{20}x^{2}-\tfrac{31}{10}x-\tfrac{861}{20} < 0 \; \forall x \ge 0 \\
             &\mk b(x,0.42) -\tfrac{1}{125}x^{2}-\tfrac{62}{125}x-\tfrac{336}{125} < 0 \; \forall x \ge 0
         \end{align*}
         as desired.
         \item As $w>z,$ the denominator of Equation \ref{eq: An nsa=4-20 -t} is positive, and so consider the numerator, say $\mk a$.
         Write $w=(1.72+\epsilon)z$ and $z=31+x$ where $0.28 \ge \epsilon \ge 0$ and $x \ge 0.$ 
         Then $\mk a$ as a function of $x,\epsilon$ is
         $$\mk a(x,\epsilon) = (5x^2+310x+4805)\epsilon^2+\dots$$
         As a quadratic in $\epsilon,$ this has positive leading coefficient. So we only need to check this is negative for $\epsilon \in \{0,0.28\}.$ Indeed,
         \begin{align*}
             &\mk a(x,0) = -\tfrac{41}{125}x^{2}-\tfrac{1292}{125}x-\tfrac{26}{125} < 0 \; \forall x \ge 0 \\
             &\mk a(x,0.28) =-x^{2}-52x-646 < 0 \; \forall x \ge 0
         \end{align*}
         as desired.
         \item Write $w=2z+\epsilon$ for $\epsilon\ge0.$ Then $11z^2-15zw+5w^2+5$ as a function of $x,\epsilon$ is
         $$z^{2}+5z\epsilon+5\epsilon^{2}+5>0 \; \forall\epsilon\ge0$$
         as desired.  \qedhere
     \end{enumerate}
    \end{proof}
    We are thus left to analyze the case when $z < 31.$ The only possible solutions to the Pell's equation \ref{eq: An nsa=4-20 pell} in this range are
    $$(z,w) \in \{(5,7),(5,8),(10,13),(10,17),(25,32),(25,43)\}.$$
    As shown in Computation \ref{comp: An case I degenerate}, $\beta/\gamma<0$ for $(z,w) = (5,8),$ and $\Width_{(1,-1)}<m$ for $(z,w) \in \{(10,13),(25,32)\}.$
    Of the remaining cases, we can manually check that there is only polygon (up to scaling) which satisfies $\Width(\Delta) \ge m.$ 
    For that quadrilateral, $\mc L_{\Delta}(m)$ has a reducible curve, showing it's not elliptic. 
    \end{enumerate}


\subsubsection*{Case $(n,s,a)=(4,-2,4)$}
In this case, the Pell's equation is
\begin{equation}
\label{eq: An nsa=4-24 pell}
11z^2-7zw+w^2+5 = 0.
\end{equation}
Here, $t = \pm(22z - 2w - 10)$ as seen in Computation \ref{comp: An case I general}.
As seen before, we find two roots for $\gamma/\delta$ in terms of $z,w.$
\begin{enumerate}
    \item Consider the first case. We find (cf. Comp. \ref{comp: An case I degenerate})
    \begin{equation}
    \label{eq: An nsa=4-24 +t}
        \frac{\beta}{\gamma} = \frac{-55z^{2}+49zw-9w^{2}-110z+10w-55}{25z^{2}-25zw+5w^{2}+50z+25}.
    \end{equation}
    Let $\mk b, \mk c$ be the numerator, denominator in Equation \ref{eq: An nsa=4-24 +t}. We show $\mk b < 0$ or $\mk c < 0$ for any $(z,w)$ satisfying the Pell's equation.
        \begin{claim}
        \label{claim: An nsa=4-24 +t}
        Let $(z,w)$ be non-negative integers with $z \ge 14.$ Then
        \begin{enumerate}
            \item if $2z \ge w,$ then there are no solutions to the Pell's equation \ref{eq: An nsa=4-24 pell};
            \item if $2.5z \ge w \ge 2z,$ then $\mk c < 0$ but $\mk b>0;$
            \item if $4.4z \ge w \ge 2.5z,$ then there are no solutions to the Pell's equation \ref{eq: An nsa=4-24 pell};
            \item if $w \ge 4.4z,$ then $\mk b < 0$ but $\mk c>0.$
        \end{enumerate}
    \end{claim}
    \begin{proof}
         We use the same technique for all; introduce a dummy variable $\epsilon$ and use \textsc{Macaulay2}. See Computation \ref{comp: An case I degenerate} for computations.
     \begin{enumerate}
        \item Write $w=2z-\epsilon$ for $\epsilon\ge0.$ Then $11z^2-7zw+w^2+5$ as a function of $z,\epsilon$ is
        $$z^{2}+3z\epsilon+\epsilon^{2}+5 > 0 \; \forall z,\epsilon\ge0$$
        showing it's nonzero.
         \item Write $w=(2+\epsilon)z$ and $z=14+x$ for $0.5 \ge \epsilon \ge 0$ and $x \ge 0.$ Then $\mk b, \mk c$ as a function of $x,\epsilon$ are
         \begin{align*}
            &\mk b(x,\epsilon)=(-9x^2- 252x-1764)\epsilon^2+\dots \\
            &\mk c(x,\epsilon)=(5x^2+140x+980)\epsilon^2+\dots.   
         \end{align*}
         As a quadratic in $\epsilon,$ they have negative, positive leading coefficients respectively. So we only need to check the result for $\epsilon \in \{0,0.5\}.$ Indeed,
         \begin{align*}
             &\mk b(x,0) = 7x^{2}+106x+57 > 0 \; \forall x \ge0 \\
             &\mk b(x,0.5)= \tfrac{45}{4}x^{2}+230x+960 > 0 \; \forall x \ge 0 \\
             &\mk c(x,0) = -5x^{2}-90x-255 < 0 \; \forall x \ge 0 \\
             &\mk c(x,0.5)= -\tfrac{25}{4}x^{2}-125x-500 < 0 \; \forall x \ge 0
         \end{align*}
         as desired.   
         \item Write $w=(2.5+\epsilon)z$ and $z=14+x$ for $1.9 \ge \epsilon \ge 0.$
         Then $11z^2-7zw+w^2+5$ as a function of $z,\epsilon$ is
        \begin{align*}
            f(x,\epsilon)=(x^2+28x+196)\epsilon^2+\dots    
         \end{align*}
         As a quadratic in $\epsilon,$ this has positive leading coefficient. So we only need to check the result for $\epsilon \in \{0,1.5\}.$ Indeed,
         \begin{align*}
             &f(x,0) = -\tfrac{1}{4}x^{2}-7x-44 < 0 \; \forall x \ge0 \\
             &f(x,1.9)= -\tfrac{11}{25}x^{2}-\tfrac{308}{25}x-\tfrac{2031}{25} < 0 \; \forall x \ge 0
         \end{align*}
         as desired.   
         \item Write $w=4.4z+\epsilon$ and $z=14+x$ for $\epsilon,x \ge 0.$ Then $\mk b, \mk c$ as functions of $x,\epsilon$ are
         {\footnotesize
         \begin{align*}
            \mk b(x,\epsilon)&=-9x^{2}\epsilon^{2}-\tfrac{151}{5}x^{2}\epsilon-252x\epsilon^{2}-\tfrac{341}{25}x^{2}-\tfrac{4178}{5}x\epsilon \\
            &-1764\epsilon^{2}-\tfrac{11198}{25}x-\tfrac{28896}{5}\epsilon
       -\tfrac{91311}{25} < 0 \; \forall z,\epsilon\ge0 \\
            \mk c(x,\epsilon)&=5x^{2}\epsilon^{2}+19x^{2}\epsilon+140x\epsilon^{2}+\tfrac{59}{5}x^{2} \\
            &+532x\epsilon+980
       \epsilon^{2}+\tfrac{1902}{5}x+3724\epsilon+\tfrac{15189}{5} > 0 \; \forall z,\epsilon\ge0.
         \end{align*}
         }
         as desired.  \qedhere
     \end{enumerate}
    \end{proof}
     We are thus left to analyze the case when $z < 14.$ The only possible solutions to the Pell's equation \ref{eq: An nsa=4-20 pell} in this range are
    $$(z,w) \in \{(2,7),(3,8),(3,13),(7,17),(7,32)\}.$$
    However, we can check in each case that $\beta/\gamma<0.$


    \item  Consider the second case. Here we find (cf. Comp. \ref{comp: An case I degenerate})
    \begin{equation}
    \label{eq: An nsa=4-24 -t}
        \begin{split}
            \frac{\alpha}{\gamma} &= \frac{85z^{2}-21zw+w^{2}+10z+5}{25z^{2}-25zw+5w^{2}+50z+25} \\
            &\frac{\beta}{\gamma} = \frac{-55z^{2}+71zw-11w^{2}-110z-55}{25z^{2}-25zw+5w^{2}+50z+25} \\
            \frac{\delta}{\gamma} &= \frac{16z}{5z^{2}-5zw+w^{2}+10z+5}.
        \end{split}
    \end{equation}

    \begin{claim}
    \label{claim: An nsa = 4-24 -t w<5z}
        Any $(z,w)$ satisfying the Pell's equation \ref{eq: An nsa=4-24 pell} satisfy $w<5z.$
    \end{claim}
    \begin{proof}
        We prove the contrapositive; write $w=5z+\epsilon$ for $\epsilon>0.$ Then the  expression $11z^2-7zw+w^2+5$ is 
        $$z^{2}+3z\epsilon+\epsilon^{2}+5>0 \; \forall z,\epsilon \ge 0$$
        as desired.
    \end{proof}

    Now, observe the side with lattice length $\beta$ has vector $(1,c)^T=(1,5)^T,$ i.e. it is perpendicular to $(5,-1)^T.$
    Further, as $w/z<5$ by Claim \ref{claim: An nsa = 107 -t w<8z}, the width along $(5,-1)$ is (also see Figure \ref{fig:width(1,-1)})
    $$\lt \la \alpha\begin{pmatrix}
        1\\0
    \end{pmatrix}, \begin{pmatrix}
        5 \\-1
    \end{pmatrix} \rt \ra=5\alpha.$$
    The condition $\Width_{(5,-1)}(\Delta) \ge m$ then implies

    $$4\lt(\frac{\alpha}{\gamma}\rt)-\frac{\beta}{\gamma}-\frac{\delta}{\gamma}-1 \ge 0,$$
    which by Equations \ref{eq: An nsa=4-24 -t} is
    \begin{equation}
    \label{eq: An nsa=4-24 -t width-m}
    \frac{74z^{2}-26zw+2w^{2}+4z+10}{5z^{2}-5zw+w^{2}+10z+
       5} \ge 0.
    \end{equation}
    Let $\mk b, \mk c$ be the numerator, denominator in Equation \ref{eq: An nsa = 107 -t width-m}. Observe that $\mk c$ is the same denominator we have in Equations \ref{eq: An nsa = 107 -t} up to a constant. Hence we can prove the following:

    \begin{claim}
    \label{claim: An nsa = 4-24 mk c > 0}
        If $\delta/\gamma > 0,$ then $\mk c > 0.$
    \end{claim}
    \begin{proof}
    This follows since the numerator of $\delta/\gamma$ in Equation \ref{eq: An nsa=4-24 -t} is $16z>0.$
    \end{proof}

    Claim \ref{claim: An nsa = 4-24 mk c > 0} and Equation \ref{eq: An nsa=4-24 -t width-m} then imply $\mk b>0$ too. 
    However, we show this is impossible by studying individual regions in $\RR^2_{z,w}:$

    \begin{claim}
    \label{claim: An nsa=4-24 -t}
    Let $(z,w)$ be non-negative integers with $z \ge 14.$ Then
    \begin{enumerate}
        \item if $2.5z \ge w \ge 2z,$ then $\mk c < 0;$
        \item if $4.4z \ge w \ge 2.5z,$ then there are no solutions to the Pell's equation \ref{eq: An nsa=4-24 pell};
        \item if $5z \ge w \ge 4.4z,$ then $\mk b <0;$
        \item if $w \ge 5z,$ then there are no solutions to the Pell's equation \ref{eq: An nsa=4-24 pell}.
    \end{enumerate}
    \end{claim}
    \begin{proof}
     We use the same technique for all; introduce a dummy variable $\epsilon$ and use \textsc{Macaulay2}. See Computation \ref{comp: An case I degenerate} for computations.
     \begin{enumerate}
         \item This was $(a)$ in Claim \ref{claim: An nsa=4-24 +t}.
         \item This was $(b)$ in Claim \ref{claim: An nsa=4-24 +t}.
         \item Write $w=(4.4+\epsilon)z$ with $0.6 \ge \epsilon \ge 0$ and $z=14+x.$ Then $\mk b$ as a function of $x,\epsilon$ is
         $$\mk b(x,\epsilon) = (2x^2+56x+ 392)\epsilon^2+\dots.$$
         As a quadratic in $\epsilon,$ this has positive leading coefficient. Hence we only need to check this is negative for $\epsilon\in \{0,0.6\}.$ Indeed,
         \begin{align*}
             &\mk b(x,0) = -\tfrac{42}{25}x^{2}-\tfrac{1076}{25}x-\tfrac{6582}{25} < 0 \; \forall x \ge 0\\
             &\mk b(x,0.6) = -6x^{2}-164x-1110 < 0 \; \forall x \ge 0
         \end{align*}
         as desired.
         \item This was Claim \ref{claim: An nsa = 4-24 -t w<5z}. \qedhere
     \end{enumerate}
    \end{proof}

     We are thus left to analyze the case when $z < 14.$ The only possible solutions to the Pell's equation \ref{eq: An nsa=4-20 pell} in this range are
    $$(z,w) \in \{(2,7),(3,8),(3,13),(7,17),(7,32)\}.$$
    As shown in Computation \ref{comp: An case I degenerate}, $\beta/\gamma<0$ for $(z,w) = (2,7),$ and $\Width_{(5,-1)}<m$ for $(z,w) =(7,32).$
    Of the remaining cases, we can manually check that there is only polygon (up to scaling) which satisfies $\Width(\Delta) \ge m.$ 
    For that quadrilateral, $\mc L_{\Delta}(m)$ has a reducible curve, showing it's not elliptic.  
\end{enumerate}

\subsection{Case II}

This is similar to Case II of the smooth case, see Section \ref{subsec: A0 Case II} (also see Remark \ref{remark: -(s+5) so can repeat results}). We start off by establishing a similar equivalence between surfaces in this case and solutions $\{\mf M,\mf N\}$ to a pair of matrix equations.

\begin{lemma}
\label{lemma: An case II both equations}
There is a bijection between semi-elliptic (smooth) toric surfaces in Case II and integral matrix solutions $\{\mf M,\mf N\}$ of the system

\begin{equation}
\label{eq: An case II eq 1}
    \mf M \begin{pmatrix}
    0 & -1 \\
    1 & 1
\end{pmatrix} \mf N^R =\begin{pmatrix}
        -(ns+3n-1) & -n \\
        (n+1)s+3n+2 & n+1
    \end{pmatrix} 
\end{equation}
\begin{equation}
    \label{eq: An case II eq 2}
        \mf N \begin{pmatrix}
        0&-1\\1&1
    \end{pmatrix}\mf M \begin{pmatrix}
        0&-1\\1&1
    \end{pmatrix} \begin{pmatrix}
        0&-1\\1&2
    \end{pmatrix}^a \begin{pmatrix}
        0&-1\\1&s+5
    \end{pmatrix}\begin{pmatrix}
        0&-1\\1&2
    \end{pmatrix}^b\begin{pmatrix}
        0&-1\\1&1
    \end{pmatrix}=\mf I
    \end{equation}
where $s \in \ZZ_{>0}$ and $a+b=8-n$ are any non-negative integers.
\end{lemma}

\begin{table}[h]
\begin{tabular}{cc}
    \begin{minipage}{.5\linewidth}
    \tiny
        \begin{tabular}{c|c|c|c}
    $n$ & $a$ & $b$ & Equation \\ 
    \hline \hline
    $1$&$2$&$5$&$(36s^{2}+210s+295)^{2}-w^2(6s^{2}+30s+25)$ \\ \hline
    $1$&$3$&$4$&$(40s^{2}+230s+319)^{2}-w^2(10s^{2}+50s+49)$ \\ \hline
    $1$&$4$&$3$&$(40s^{2}+228s+313)^{2}-w^2(12s^{2}+60s+61)$ \\ \hline
    $1$&$5$&$2$&$(36s^{2}+204s+277)^{2}-w^2(12s^{2}+60s+61)$\\ \hline
    $1$&$6$&$1$&$(28s^{2}+158s+211)^{2}-w^2(10s^{2}+50s+49)$ \\ \hline
    $1$&$7$&$0$&$(16s^{2}+90s+115)^{2}-w^2(6s^{2}+30s+25)$ \\ \hline
    $2$&$1$&$5$&$(36s^{2}+222s+331)^{2}-w^2(6s^{2}+42s+61)$ \\ \hline
    $2$&$2$&$4$&$(45s^{2}+270s+394)^{2}-w^2(15s^{2}+90s+124)$ \\ \hline
    $2$&$3$&$3$&$(48s^{2}+284s+409)^{2}-w^2(20s^{2}+116s+157)$ \\ \hline
    $2$&$4$&$2$&$(45s^{2}+264s+376)^{2}-w^2(21s^{2}+120s+160)$ \\ \hline
    $2$&$5$&$1$&$(36s^{2}+210s+295)^{2}-w^2(18s^{2}+102s+133)$ \\ \hline
    $2$&$6$&$0$&$(21s^{2}+122s+166)^{2}-w^2(11s^{2}+62s+76)$ \\ \hline
    $3$&$1$&$4$&$(40s^{2}+250s+379)^{2}-w^2(10s^{2}+70s+109)$ \\ \hline
    $3$&$2$&$3$&$(48s^{2}+292s+433)^{2}-w^2(20s^{2}+124s+181)$ \\ \hline
    $3$&$3$&$2$&$(48s^{2}+288s+421)^{2}-w^2(24s^{2}+144s+205)$ \\ \hline
    $3$&$4$&$1$&$(2s+7)^{2}(20s+49)^{2}-w^2(22s^{2}+130s+181)$ \\ \hline
    \hline
    \end{tabular}
    \end{minipage} &

    \begin{minipage}{.5\linewidth}
    \tiny
        \begin{tabular}{c|c|c|c}
    $n$ & $a$ & $b$ & Equation \\ 
    \hline \hline
    $3$&$5$&$0$&$(24s^{2}+142s+199)^{2}-w^2(14s^{2}+82s+109)$ \\ \hline
    $4$&$0$&$4$&$(25s^{2}+170s+274)^{2}+w^2(5s^{2}+10s-4)$ \\ \hline
    $4$&$1$&$3$&$(40s^{2}+252s+385)^{2}-w^2(12s^{2}+84s+133)$ \\ \hline
    $4$&$2$&$2$&$(45s^{2}+276s+412)^{2}-w^2(21s^{2}+132s+196)$ \\ \hline
    $4$&$3$&$1$&$(2s+5)^{2}(20s+71)^{2}-w^2(22s^{2}+134s+193)$ \\ \hline
    $4$&$4$&$0$&$(25s^{2}+150s+214)^{2}-w^2(15s^{2}+90s+124)$ \\ \hline
    $5$&$0$&$3$&$(24s^{2}+164s+265)^{2}+w^2(4s^{2}+4s-13)$ \\ \hline
    $5$&$1$&$2$&$(36s^{2}+228s+349)^{2}-w^2(12s^{2}+84s+133)$ \\ \hline
    $5$&$2$&$1$&$(36s^{2}+222s+331)^{2}-w^2(18s^{2}+114s+169)$ \\ \hline
    $5$&$3$&$0$&$(24s^{2}+146s+211)^{2}-w^2(14s^{2}+86s+121)$ \\ \hline
    $6$&$0$&$2$&$(21s^{2}+144s+232)^{2}+w^2(3s^{2}-16)$ \\ \hline
    $6$&$1$&$1$&$(28s^{2}+178s+271)^{2}-w^2(10s^{2}+70s+109)$ \\ \hline
    $6$&$2$&$0$&$(21s^{2}+130s+190)^{2}-w^2(11s^{2}+70s+100)$ \\ \hline
    $7$&$0$&$1$&$(2s+5)^{2}(8s+35)^{2}+w^2(2s^{2}-2s-13)$ \\ \hline
    $7$&$1$&$0$&$(16s^{2}+102s+151)^{2}-w^2(6s^{2}+42s+61)$ \\ \hline
    $8$&$0$&$0$&$(9s^{2}+62s+94)^{2}+w^2(s^{2}-2s-4)$ \\
    \hline\hline
    \end{tabular}
    \end{minipage} 
\end{tabular}
 \centering
    \caption{Quartics in $s,w.$ Note they are all of the form $f(s)w^2+g(s).$}
    \label{table:An case II}
\end{table}

Write 
\begin{equation}
\label{eq: An case II M matrix}
    \mf M = \begin{pmatrix}
    x&y\\
    z&w
\end{pmatrix}
\end{equation}
with $x<y<0<z<w$ and $\det(M)=xw-yz=1$ (cf. Lem. \ref{lem: matrix and chain}).
We now again use \textsc{Macaulay2} to solve the system of Equations \ref{eq: An case II eq 1}, \ref{eq: An case II eq 2} and get a finite number of possibilities like in Section \ref{subsec: A0 Case II}, many more this time, however. 

\begin{lemma}
For each non-negative $(a,b)$ pair with $a+b=8-n$, solutions $(z,w)$ to the system in Lemma \ref{lemma: An case II both equations} lie on curves in $\mathbb{A}^2$ of bidegree $(2,4),$ listed in Table \ref{table:An case II}. However, none of them give an elliptic quadrilateral.
\end{lemma}
\begin{proof}
The same method used in Lemma \ref{lemma: A0 Case II key lemma} works, see Computation \ref{comp: An Case II}. 
Out of the terms $f(s)w^2+g(s),$ if $f,g$ are positive for all $s,$ we discard those in Table \ref{table:An case II}.
Using \textsc{Mathematica}, we get $31$ admissible tuples $(z,w).$ All of them can be used to construct potential elliptic quadrilaterals. However, some of them either have a negative side, or fail the $\Width(\Delta) \ge m$ condition, see Computation \ref{comp: An Case II} for details.
\end{proof}

\nocite{Magma}
\nocite{M2}
\nocite{Mathematica}

\bibliographystyle{alpha}
\bibliography{toric_bib}

\clearpage

\appendix

\section{Toric Surfaces}

Here, we present and prove some key facts about toric surfaces that we use in the paper. 
We skip the proof of some well known results, although most of these can be found in a standard text such as \cite{fulton}.
We also discuss some non-standard terminologies we use in the paper.

\begin{lemma}
\label{lemma: convexity of (-2) curves}
If the vectors $v_i,\dots,v_{i+k},$ generating rays of the toric fan $\Ti$ of $X,$ all correspond to curves with self-intersection $\le (-2),$ then $\measuredangle (v_i, v_{i+k}) < \pi,$ \footnote{In this paper, $\measuredangle$ means directed angle, measured counter-clockwise} i.e. they form a convex cone. 
Conversely, if $\measuredangle (v_i, v_{i+k}) < \pi,$ then the minimal resolution of the cone by subdivision correspond to curves of self-intersection $\le (-2).$
\end{lemma}

\begin{lemma}
    \label{lemma: sum of smooth numbers}
    Suppose the toric fan of $X$ (smooth and projective) has $n$ primitive vectors (or equivalently $X$ has $n$ boundary divisors). The sum of self-intersection numbers of the boundary divisors of $X$ equals $12-3n.$
\end{lemma}

\begin{definition}
\label{def: matrix}
    Given a curve $\mc C$ on a smooth projective surface, we define the \textit{matrix corresponding to a curve} $\mc C$ as the matrix
    $$\mf M_{\mc C} := \begin{pmatrix}
        0 & -1 \\ 1 & -\mc C^2
    \end{pmatrix}.$$
    The \textit{matrix corresponding to a chain} $\mc C_1,\dots, \mc C_n$ is defined as the product
     $$\mf M := \begin{pmatrix}
        0 & -1 \\ 1 & -\mc C_1^2
    \end{pmatrix} \cdots \begin{pmatrix}
        0 & -1 \\ 1 & -\mc C_n^2
    \end{pmatrix}.$$
    Observe here that $\mf M \in \SL_2(\ZZ).$
\end{definition}

\begin{lemma}
\label{lem: matrix and chain}
Let 
$$\mf M=\begin{pmatrix}
    x&y\\z&w
\end{pmatrix} \in \SL_2(\ZZ)$$
be the matrix of a chain of curves $\mc C_1,\dots, \mc C_n.$ Assume $\mc C_i^2 \le -2$ for all $i.$
\begin{enumerate}
    \item Then $x,y \le 0 \le z,w$ and $|x|<|y|,|z| < w.$
    \item The Hirzebruch-Jung fraction of the chain is $[-\mc C_1^2,\dots,-\mc C_n^2]=\frac{w}{(-y)}.$
    \item The cone generated by $v_0=(0,1)$ and $v_1=(w,y)$ has the resolution given by the curves $\mc C_1,\dots, \mc C_n$ (clockwise).
    \item The cone generated by $v_0=(0,1)$ and $v_1=(w,-z)$ has the resolution given by the curves $\mc C_n,\dots, \mc C_1$ (also follows from (3) and Lemma \ref{lemdef: reverse matrix}).
\end{enumerate}
\end{lemma}

\begin{corollary}
\label{cor: chains up to equivalence}
Define an equivalence relation $\sim$ between two convex cones $(v_0,v_1)$ and $(v_0',v_1')$ iff there is a $\GL_2(\ZZ)$ transform $\mf M$ with $\mf Mv_0=v_0'$ and $\mf Mv_1=v_1'.$ Then there is a bijection
$$\lt\{\substack{\text{Chains }\mc C_1,\dots, \mc C_n \\
\text{with }\mc C_i^2 \le -2}\rt\} \longleftrightarrow \lt\{\substack{\text{Convex cones cones}\\ \text{in }\ZZ^2}\rt\}/\sim.$$
In particular, the cone generated by $v_0=(0,1)$ and $v_1=(x,y)$ corresponds to the chain with Hirzebruch-Jung fraction $\frac{w}{(-y)}$ if and only if $$(x,y) \in \{(w,y+kw),(-w,y+kw)\}$$
for some $k \in \ZZ.$
\end{corollary}
\begin{proof}
Every rational convex cone corresponds to a chain of curves of self-intersection $\le (-2),$ see Lemma \ref{lemma: convexity of (-2) curves}.
Hence it suffices to show injectivity, i.e., the chain you get from a cone depends only on the class of the cone.
Consider a cone $(v_0,v_1),$ and map $v_0$ to $(0,1)$ and $v_1$ to $(w,y)$ with $y<0<w$ and $|y|<|w|.$ 
By Lemma \ref{lem: matrix and chain} and invariance of Hirzebruch-Jung fractions by $\GL_2(\ZZ),$ the cone $(v_0,v_1)$ corresponds to the fraction $\frac{w}{(-y)}.$
Since the Hirzebruch-Jung fraction uniquely determines the minimal chain (i.e. without $(-1)$ curves), the result follows.
\end{proof}


\begin{example}
\label{example: du val matrix}
Consider the resolution of an $A_{n}$ singularity.
It is easy to check that $v_0=(0,1)$ and $v_1=(n+1,1)$ has a resolution with $n$ $(-2)$ curves, namely the vectors $(a,1)$ for $0<a \le n.$
By uniqueness of resolution and Lemma \ref{lem: matrix and chain}, the Hirzebruch-Jung fraction is $\frac{n+1}{n}.$
We can then find the matrix $\mf M$ using the facts that $\mf M \in \SL_2(\ZZ)$ and $(1)$ of Lemma \ref{lem: matrix and chain}. This gives a combinatorial proof of the identity
$$\mf M = \begin{pmatrix}
    0&-1\\
    1&2
\end{pmatrix}^{n} = \begin{pmatrix}
    -(n-1)&-n\\
    n&n+1
\end{pmatrix}.$$
Corollary \ref{cor: chains up to equivalence} also shows that every cone with $v_0=(0,1)$ corresponds to the resolution of an $A_n$ singularity if and only if $v_1=(n+1,1+k(n+1))$ or $v_1=(-n-1,1+k(n+1))$ for some $k \in \ZZ.$
\end{example}

\begin{lemma}
\label{lem: blowup at intersection}
    Consider the chain of curves $\mc C_1, \mc C_2$ and the blowup $\pi$ at the point $\mc C_1 \cap \mc C_2.$ Then the matrix of the new chain $\ol{\mc C_1}, E, \ol{\mc C_2}$ of the strict transforms and the exceptional curve is the same as the matrix of the chain $\mc C_1, \mc C_2.$
\end{lemma}
\begin{proof}
    Since self-intersection numbers of the strict transforms $\mc C_1, \mc C_2$ will be one lesser than that of $\mc C_1, \mc C_2,$ hence the identity $$\begin{pmatrix}
        0&-1\\
        1&a
    \end{pmatrix}\begin{pmatrix}
        0&-1\\
        1&b
    \end{pmatrix}=\begin{pmatrix}
        0&-1\\
        1&a+1
    \end{pmatrix}\begin{pmatrix}
        0&-1\\
        1&1
    \end{pmatrix}\begin{pmatrix}
        0&-1\\
        1&b+1
    \end{pmatrix}$$
    proves the desired result.
\end{proof}

\begin{lemdef}
\label{lemdef: reverse matrix}
    If the matrix corresponding to the chain $\mc C_1,\dots, \mc C_n$ is 
    $$\mf M= \begin{pmatrix}
        x&y\\z&w
    \end{pmatrix},$$
    then the matrix corresponding to the reverse chain $\mc C_n,\dots, \mc C_1$ is 
    $$\begin{pmatrix}
        x&-z\\-y&w
    \end{pmatrix}.$$
    Call this the \textit{reverse matrix} of $\mf M,$ which we will denote by $\mf M^R.$
\end{lemdef}
\begin{proof}
Firstly, we prove that $(AB)^R=B^RA^R.$ While this can be manually done for $2\times 2$ matrices, another way is to observe that
$$(xw-yz)\begin{pmatrix}0&1\\ 1&0\end{pmatrix}\begin{pmatrix}x&y\\ z&w\end{pmatrix}^{-1}\begin{pmatrix}0&1\\ 1&0\end{pmatrix} = \begin{pmatrix}x&-z\\ -y&w\end{pmatrix}.$$
Hence if $E$ is the column swap elementary matrix, then $\mf M^R=(\det \mf M) E \mf M^{-1} E,$ from which the desired fact follows. 

\medskip

Next, to prove the result, we use induction, the base case for $n=1$ being clear.
Assuming the result holds for a chain of length $n,$ consider a chain $\mc C_1,\dots, \mc C_{n+1}.$ The matrix corresponding to the reverse of this chain would be the product of the matrices corresponding to $\mc C_{n+1}$ and the reverse of the chain $\mc C_1,\dots, \mc C_n.$ 
By the induction hypothesis, this becomes the product $$\mf M_{\mc C_{n+1}}^R(\mf M_{\mc C_1} \dots \mf M_{\mc C_n})^R = (\mf M_{\mc C_1} \dots \mf M_{\mc C_n} \mf M_{\mc C_{n+1}})^R$$
and hence the induction is complete.
\end{proof}



\begin{lemma}[Matrix Equation]
\label{lem: matrix equation}
    An ordered chain of curves $C_1,\dots, C_n$ are the boundary curves of a smooth projective toric surface if and only if $\sum_i \mc C_i^2=12-3n$ and their matrices multiply to the identity matrix, i.e.
    $$\begin{pmatrix}
        0 & -1 \\ 1 & -\mc C_1^2
    \end{pmatrix} \cdots \begin{pmatrix}
        0 & -1 \\ 1 & -\mc C_n^2
    \end{pmatrix}=\begin{pmatrix}
        1 & 0 \\ 0 & 1
    \end{pmatrix}.$$
\end{lemma}

\clearpage

\clearpage

\section{Computer Computations}

We use the same \textsc{Magma} package used in the paper \cite{jenia_blowup}, which can be found here: \url{https://github.com/alaface/non-polyhedral}. The \textsc{Python} code used can be found here: \url{https://github.com/Wizard-32/Toric-Varieties}. The Python code helps visualize toric resolutions given the coordinates of a polygon, as well as draw the fan of a cyclic quotient singularity.
While these aid certain calculations, all of the code used in this paper is contained in this section, except in Computation \ref{comp: DE case}, where we use a function from the \textsc{Python} file.

\begin{computation}
\label{comp:random example}
We verify the quadrilateral given in Example \ref{example:random example} is indeed a good quadrilateral.
\begin{tbox} {\footnotesize \begin{verbatim}
> p := Polytope([ [0 , 0] , [1 , 0] , [20 , 14], [7 , 5] ]);  
> #BoundaryPoints(p);    
    4
> Volume(p);
    16
> #FindCurves(p, 4, QQ);  
    1
> IsIrreducible(FindCurve(p,4,QQ));
    true
\end{verbatim}
    } 
    \end{tbox}
The linear system, however, has two curves as can be seen by the following:
\begin{tbox} {\footnotesize \begin{verbatim}
> #PolsAdjSys(p);      
    2
\end{verbatim}
} 
\end{tbox}
\end{computation}


\begin{computation}
    \label{comp: DE case}

    Consider Equation \ref{eq: DE eq 1}. As $a+b+n=8$ and $p+q+r=n-1,$ we can consider all the possibilities on \textsc{Macaulay2}. 
    The strategy is to use equation \ref{eq: DE eq 1} to get $4$ equations. We then consider the ideal generated by them, and use \texttt{minimalPrimes} to find the minimal primes of the ideal.

    \begin{tbox} {\footnotesize \begin{verbatim}
R = QQ[a,b,s,x,y,z,w,p,q,r,l,n];
M1 = matrix{{0,-1},{1,1}};
Ma = matrix {{1-a,-a},{a,a+1}};
Mb = matrix {{1-b,-b},{b,b+1}};
cc = k -> matrix{{0,-1},{1,-k}}; -- "canonical matrix"
exp2 = k -> matrix{{1-k,-k},{k,k+1}}; -- matrix of (-2) to power k

-- Dn CASE I
eq = 
exp2(p)*cc(-3-l)*exp2(q)*cc(-1)*exp2(l)*cc(-3)*
exp2(r-1)*cc(-1)*Mb*cc(-3)*Ma*cc(-1) - matrix{{1,0},{0,1}}

arr = [[1,1,n-3], [1,n-3,1], [n-3,1,1]];
for aa from 0 to 8 do{
for elem in arr do{
pp = elem_0; qq = elem_1; rr = elem_2;
bb = 7 - (aa + pp + qq + rr);
print("CASE a,b,p,q,r");
print(aa, bb, pp, qq, rr);
print("");
eq2 = sub(eq,{p => pp, q => qq, r => rr, a => aa, b => bb});
I = minimalPrimes ideal (eq2_(0,0), eq2_(0,1), eq2_(1,0), eq2_(1,1));
print(I);
print("");
print("");
}}
\end{verbatim}
} 
\end{tbox}

To test the possible tuples we get, we run the following code from the \textsc{Python} file ``veccheck.py''.

\begin{tbox} {\footnotesize \begin{verbatim}
'''D_n case eq 1'''
vec = [-1,-2,-4,-2,-1,-2,-3,-1,-2,-2,-2,-2,-3]
list_of_black = [0,4,7,11]
vec_check(vec,list_of_black)

vec = [-1,-2,-3,-2,-1,-3,-2,-2,-2,-1,-3,-2]
list_of_black = [0,4,9,10]
vec_check(vec,list_of_black)
\end{verbatim}
} 
\end{tbox}

For the $E_n$ case, we run the same code as above, just change \texttt{arr} to

\begin{tbox} {\footnotesize \begin{verbatim}
arr = [[1,2,n-4], [2,1,n-4], [1,n-4,2], 
[2,n-4,1], [n-4,1,2], [n-4,2,1]];
\end{verbatim}
} 
\end{tbox}

To test the possible tuple we get, we run the following code:

\begin{tbox} {\footnotesize \begin{verbatim}
'''E_n case eq 1'''
vec = [-1,-2,-3,-2,-2,-1,-3,-2,-1,-2,-2,-3]
list_of_black = [0,5,8,11]
vec_check(vec,list_of_black)
\end{verbatim}
} 
\end{tbox}

For equation \ref{eq: DE eq 2}, we run the same code as above for $D_n,$ only changing \texttt{eq} to

\begin{tbox} {\footnotesize \begin{verbatim}
eq = 
exp2(p)*cc(-3-l)*exp2(q)*cc(-1)*exp2(r-1)*cc(-3)*
exp2(l)*cc(-1)*Mb*cc(-3)*Ma*cc(-1) - matrix{{1,0},{0,1}};
\end{verbatim}
} 
\end{tbox}

We again run the following code from the same \textsc{Python} file.

\begin{tbox} {\footnotesize \begin{verbatim}
'''D_n case eq 2'''
vec = [-1,-2,-3,-2,-1,-2,-3,-1,-2,-2,-2,-3]
list_of_black = [0,4,7,11]
vec_check(vec,list_of_black)
\end{verbatim}
} 
\end{tbox}

For the $E_n$ case, we run the same code as above, and change \texttt{arr} as before.
\end{computation}


\begin{computation}
\label{comp: A0 case I 1}

We use the following code in \textsc{Macaulay2} to get Lemma \ref{lemma: pell's equation for A0}.

\begin{tbox} {\footnotesize \begin{verbatim}
R = QQ[a,b,s,x,y,z,w];

-- Standard matrices
M = matrix{{x,y},{z,w}};
M1 = matrix{{0,-1},{1,1}}; 
Ma = matrix {{1-a,-a},{a,a+1}}; 
Mb = matrix {{1-b,-b},{b,b+1}}; 

-- matrix for self-intersection k
cc = k -> matrix{{0,-1},{1,-k}}; 

-- matrix for chain contracting to A0 (up to orientation)
tc = s -> matrix{{-s, -1},{s+1, 1}};

-- Computing N using equation 1
-- inverse of M*M1;
MM1inv = matrix{{-z+w,x-y},{-w,y}};
N = MM1inv * tc(s);

-- equation 2
eq = N * M1 * M * M1 * Ma * cc(-(s+5)) * Mb * M1 - matrix {{1,0},{0,1}};

-- running a from 0 to 8 and using a+b=8
for a' from 0 to 8 do{
  eq' = sub(eq, {a => a', b => 8-a'});
  print(a');
  print(minimalPrimes ideal (eq'_(0,0), eq'_(0,1), 
  eq'_(1,0), eq'_(1,1), x*w-y*z-1));
  print("");
}
\end{verbatim}
} 
\end{tbox}

    Based on the output, we see that we get $s \ge 0$ only for $a=0$ and $8.$ Further, we get corresponding Pell's equations in terms of $z,w.$

    \medskip

    Next, we use Lemma \ref{lem: vol = m^2 equation} to get the second equation. 
    This is a quadratic in $\delta/\gamma,$ with discriminant \texttt{disc}. This must be a square, which we call $t^2$ for some $t \in \ZZ.$

    \medskip
    
    The function \texttt{idealCheck} takes as input $(n,s,a)$ (where $n=0$ in the smooth case) and the Pell's equation, and outputs the minimal primes of the ideal generated by the Pell's equation and \texttt{disc-t}$^2.$ Note that $\alpha/\gamma,\beta/\gamma,\delta/\gamma$ are replaced by \texttt{a', b', d'}.

    \begin{tbox} {\footnotesize \begin{verbatim}
R = QQ[x,y,z,w,s,a',b',c',d',c,d,l,t,e];
idealCheck = (nn,ss,aa, pell) -> {
    c=aa+1; d=(8-aa-nn)+1; l=ss+3;
    b' = (d'*w-(c+d+l*c*d))/c; 
    a' = z*d'-(1+l*d)-b';
    
    m = a'+b'+1+d';
    vol = (a')*(d')*w+(b')*d;
    eq = vol - m^2;

-- want integer coefficients
    eq = eq * c;

    f0 = sub(eq, {d' => 0});
    f1 = sub((eq-f0)//d', {d' => 0});
    f2 = sub((eq-f0-f1*d')//(d')^2, {d' => 0});
-- NOTE: we use // instead of / so the result is the same ring, 
-- not its fraction field
-- eq = f2(d')^2 + f1(d') + f0;
    print("disc");
    disc = (f1)^2 - 4*(f0)*(f2);
-- note we are printing -disc
    print(-disc);
    print("");
    print("ideal");
    
    J = minimalPrimes ideal (disc - t^2, pell);
    return J;}
\end{verbatim}
}
\end{tbox}

This prints \texttt{-disc} which must be a square. Furthermore, it prints minimal primes of the ideal generated by \texttt{disc-t}$^2$ and the Pell's equation. We use the function for the two cases we have:

\begin{tbox} {\footnotesize \begin{verbatim}
idealCheck(0,2,0,55*z^2-95*z*w+41*w^2+1)
idealCheck(0,2,8,55*z^2-15*z*w+w^2+1)
\end{verbatim}
} 
\end{tbox}

This shows the intersection is the union of two conics in both the cases, both conics being the intersection of a hyperplane with a quadric.
\end{computation}


\begin{computation}
    \label{comp: A0 Case I (a,b,s)=(0,8,2)}
Consider the case $(a,b,s)=(0,8,2).$ Here $t = \pm (330z - 270w - 6),$
\begin{enumerate}
    \item Consider the $+t$ root. We find $\beta/\gamma$ in Equation \ref{eq: Case A0 abs=082 beta} using
\begin{tbox} {\footnotesize \begin{verbatim}
-- initialise f0, f1, f2
idealCheck(0,2,0,55*z^2-95*z*w+41*w^2+1)

tt = 330*z - 270*w - 6
-- the root is d' = (-f1+t)/2f1.
-- up to scaling, we can write 
d'' = (f1 - tt)/(-2*f2);
a'' = sub(a', {d' => d''});
b'' = sub(b', {d' => d''});
\end{verbatim}
    } 
    \end{tbox}
    For Claim \ref{claim: A0 case I abs=082 +t}, we use
\begin{tbox} {\footnotesize \begin{verbatim}
bb = -110*z^2-130*z*w+178*w^2-220*z+96*w-110;
sub(bb,{z=>(5/6)*w+e})
\end{verbatim}
    } 
    \end{tbox}
    \item Consider the $-t$ root. We find $\alpha/\gamma,\beta/\gamma,\delta/\gamma$ in Equation \ref{eq: A0 case I sides -t} using
     \begin{tbox} {\footnotesize \begin{verbatim}
-- initialise f0, f1, f2
idealCheck(0,2,0,55*z^2-95*z*w+41*w^2+1)

tt = 330*z - 270*w - 6
-- the root is d' = (-f1+t)/2f1.
-- up to scaling, we can write 
d'' = (f1 + tt)/(-2*f2);
a'' = sub(a', {d' => d''});
b'' = sub(b', {d' => d''});

-- width - m quantity
widthM = (w-z-1)*d''-b''-1
\end{verbatim}
    } 
    \end{tbox}

For Claim \ref{claim: A0 abs = 082 -t}, we use
    \begin{tbox} {\footnotesize \begin{verbatim}
---- Part (a) ----
wnum = -156*z^2+72*z*w+54*w^2-144*z+126*w+12;
w1 = sub(wnum, {z => 100/115*w+e});
w2 = sub(w1, {w=>4+x})

---- Part (b) ----
pell = 55*z^2-95*z*w+41*w^2+1;
p1 = sub(pell, w=>(115/100+e)*z);
p2 = sub(p1, z=>84+x)

sub(p2, {e=>0})
sub(p2, {e=>3/100})

---- Part (c) ----
anum = 219*z^2-345*z*w+135*w^2+60*z-42*w+9
a1 = sub(anum, {w=>(118/100+e)*z});
a2 = sub(a1, {z=>84+x});
sub(a2, {e=>0})
sub(a2, {e=>12/100})
\end{verbatim}
    } 
    \end{tbox}
    The final cases we rule out manually on \textsc{Magma}. First, the following function takes the values $(z,w)=(\texttt{zt},\texttt{wt})$ and spits out the coordinates of the corresponding polytope, which we can test in \textsc{Magma}.
    
    \begin{tbox} {\footnotesize \begin{verbatim}
findCoordinates = (zt,wt) -> {
-- find integer side lengths upto scaling
      alpha = a'' * (-2*f2); beta = b'' * (-2*f2); 
      gamma = (-2*f2); delta = d'' * (-2*f2);
      at = sub(alpha, {z=>zt, w=>wt});
      bt = sub(beta, {z=>zt, w=>wt});
      ct = sub(gamma, {z=>zt, w=>wt});
      dt = sub(delta, {z=>zt, w=>wt});
      print([[0,0],[at,0],[at+bt, c*bt],[dt*zt,dt*wt]]);
  }

arr = [(6,7),(7,8),(11,12),(15,17),(27,32),(38,43),(70,83)];
for elem in arr do
    findCoordinates(elem);
\end{verbatim}
    } 
    \end{tbox}
    The output is
{\footnotesize  \begin{verbatim}
[[0, 0], [168, 0], [-112, -280], [5040, 5880]]
[[0, 0], [288, 0], [432, 144], [7056, 8064]]
[[0, 0], [1128, 0], [4128, 3000], [18480, 20160]]
[[0, 0], [1020, 0], [4820, 3800], [31500, 35700]]
[[0, 0], [192, 0], [3872, 3680], [90720, 107520]]
[[0, 0], [5208, 0], [38192, 32984], [197904, 223944]]
[[0, 0], [-1224, 0], [36720, 37944], [599760, 711144]]
\end{verbatim}
}
    We thus test all these quadrilaterals (except the first and last) 
    on \textsc{Magma}:

    \begin{tbox} {\footnotesize \begin{verbatim}
function checkAll(arr)
    for elem in arr do
        p := Polytope(elem);
        m := #BoundaryPoints(p);
        if Width(p) ge m then
        print(m);
        print("Yes"); print(p);
        end if; 
    end for;
    return 0;
end function;
\end{verbatim}
    } 
    \end{tbox}
    This function can be used to check that none of the quadrilaterals in the list above satisfy $\Width \ge m.$ Interestingly enough, we can check that $\Width_{(1,-1)}(\Delta)<m$ in all but one cases.
\end{enumerate}


Consider the case $(a,b,s)=(8,0,2).$ Here $t = \pm (330z - 30w - 6),$

\begin{enumerate}
    \item 
Consider the $+t$ root. For Claim \ref{claim: A0 abs=802 +t},
    \begin{tbox} {\footnotesize \begin{verbatim}
-- initialize f0, f1, f2
idealCheck(0,2,8,55*z^2-15*z*w+w^2+1)

tt = 330*z - 30*w - 6;
-- the root is d' = (-f1+t)/2f1.
-- up to scaling, we can write 
d'' = (f1 - tt)/(-2*f2);
a'' = sub(a',{d' => d''}); 
b'' = sub(b',{d' => d''});

---- Part (b) ----
bb = - 165*z^2  + 125*z*w - 13*w^2  - 330*z + 16*w - 165;
bb1 = sub(bb,{w=>(78/10-e)*z-4});
bb2 = sub(bb1,{z=>4+x});

cc = 27*z^2  - 27*z*w + 3*w^2 + 54*z + 27;
cc1 = sub(cc,{w=>(78/10-e)*z-4});
cc2 = sub(cc1,{z=>4+x});

sub(bb2,{e=>0})
sub(bb2,{e=>38/10})

sub(cc2,{e=>0})
sub(cc2,{e=>38/10})

---- Part (c) ----
pellexp = 55*z^2-15*z*w+w^2+1;
p1 = sub(pellexp,{w=>(78/10+e)*z-4});
p2 = sub(p1, {z=>3+x});

sub(p2, {e=>0})
sub(p2, {e=>6/10})

---- Part (d) ----
b1 = sub(bb, {w=>84/10*z-4+e});
c1 = sub(cc, {w=>84/10*z-4+e});
\end{verbatim}
    }
    \end{tbox}

To finish and exhaust the remaining cases, we can use

\begin{tbox} {\footnotesize \begin{verbatim}
arr = [[1,7],[1,8],[2,13],[2,17]]
for elem in arr do
      print(sub(b'', {z=>elem_0, w=>elem_1}));
\end{verbatim}
} 
\end{tbox}

\item Consider the $-t$ root.

        \begin{tbox} {\footnotesize \begin{verbatim}
-- initialize f0,f1,f2
idealCheck(0,2,8,55*z^2-15*z*w+w^2+1);

tt = -(330*z - 30*w - 6);
-- the root is d' = (-f1+t)/2f1.
-- up to scaling, we can write
d'' = (f1 - tt)/(-2*f2);
a'' = sub(a',{d' => d''});
b'' = sub(b',{d' => d''});
widthM = 8*a''-b''-1-d'' -- width(9,-1) - m
\end{verbatim}
    } 
    \end{tbox}

    For Claim \ref{claim: abs=802 root 2 claim 1}
    \begin{tbox} {\footnotesize \begin{verbatim}
pellexp = 55*z^2-15*z*w+w^2+1;
sub(pellexp,{w=>9*z+e})
\end{verbatim}
    } 
    \end{tbox}

For Claim \ref{claim: A0 abs=802 -t}, we use
    \begin{tbox} {\footnotesize \begin{verbatim}
---- Part (b) ----
cc = 9*z^2 - 9*z*w + w^2  + 18*z + 9;
c1 = sub(cc, {w=>(3+e)*z});
c2 = sub(c1, {z=>3+x});

sub(c2, {e=>0})
sub(c2, {e=>35/10})

---- Part (c) ----
pellexp = 55*z^2-15*z*w+w^2+1;
p1 = sub(pellexp, {w=>(65/10+e)*z});
p2 = sub(p1, {z=>3+x});

sub(p2, {e=>0})
sub(p2, {e=>19/10})

---- Part (d) ----
bb := 1734*z^2  - 376*z*w + 20*w^2  + 234*z - 28*w + 12;
b1 = sub(bb, {w=>(84/10+e)*z});
b2 = sub(b1, {z=>2+x});

sub(b2, {e=>0})
sub(b2, {e=>6/10})
\end{verbatim}
    } 
    \end{tbox}










    For the remaining cases, we first find the following negative:

\begin{tbox} {\footnotesize \begin{verbatim}
sub(b'',{z=>1,w=>7})
sub(b'',{z=>1,w=>8})
sub(widthM,{z=>2,w=>17})
\end{verbatim}
} 
\end{tbox}

The last case is $(z,w)=(2,13).$ In this case, we find $\lt(\tfrac{\alpha}{\gamma},\tfrac{\beta}{\gamma},\tfrac{\delta}{\gamma}\rt) = \lt(\tfrac{59}{6},\tfrac{115}{6},\tfrac{35}{2}\rt).$

\begin{tbox} {\footnotesize \begin{verbatim}
at = sub(a'', {z=>2, w=>13});  
bt = sub(b'', {z=>2, w=>13});  
dt = sub(d'', {z=>2, w=>13});  
\end{verbatim}
} 
\end{tbox}

Up to scaling, we consider a quadrilateral with side lengths $(59,115,6,105).$ The coordinates are

\begin{tbox} {\footnotesize \begin{verbatim}
> [[0,0],[at,0],[at+bt, c*bt],[dt*zt,dt*wt]];
[[0, 0], [59, 0], [174, 1035], [210, 1365]]
\end{verbatim}
} 
\end{tbox}

However, \textsc{Magma} shows $\Width_{(1,0)}=210<m=285.$
\end{enumerate}



\end{computation}


\begin{computation}
\label{comp: A0 case II}
The basic setup is similar to what we have seen before. The only difference is that we have to reverse the matrix $N$ to express it in terms of $x,y,z,w.$
\begin{tbox} {\footnotesize \begin{verbatim}
R = QQ[a,b,s,x,y,z,w];
M1 = matrix{{0,-1},{1,1}};
Ma = matrix {{1-a,-a},{a,a+1}};
Mb = matrix {{1-b,-b},{b,b+1}};
cc = k -> matrix{{0,-1},{1,-k}}; -- "canonical matrix"
M = matrix {{x,y},{z,w}};
-- reverse function
rev = A -> matrix{{A_(0,0),-A_(1,0)},{-A_(0,1),A_(1,1)}}

--inverse of M*M1
MM1inv = matrix{{-z+w,x-y},{-w,y}};
--check: MM1inv * M * M1 = I modulo xw-yz=1
N = rev(MM1inv*matrix{{-s,-1},{s+1,1}})

Neq 
= N * M1 * M * M1 * Ma * cc(-(s+5)) * Mb * M1 - matrix{{1,0},{0,1}};
-- This should be zero in order for the surface to be toric
\end{verbatim}
} 
\end{tbox}

The key now is to use the \texttt{eliminate} function to find an element in the ideal involving only $w,s.$
\begin{tbox} {\footnotesize \begin{verbatim}
for a' from 0 to 8 do{
  Neq' = sub(Neq, {a => a', b => 8-a'});
  I = minimalPrimes ideal (Neq'_(0,0), Neq'_(0,1), Neq'_(1,0), 
  Neq'_(1,1), x*w - y*z - 1);
  J = eliminate({x,y,z},I_0);
  print(J);
}
\end{verbatim}
} 
\end{tbox}

We can make a list of all cases on \textsc{Mathematica}. However, we replace $w^2$ by $k$ for it to work.

\begin{tbox} {\footnotesize \begin{verbatim}
(*(a,b) = (1,7), k = w^2*)
f = (-8s^2-24s+31)k+(256s^{4}+2304s^{3}+6816s^{2}+7344s+2601);
Sol = Solve[{f == 0}, 
{s,k}, Integers]

(*(a,b) = (2,6), k = w^2*)
f = (-14s^{2}-42s+31)k+(441s^{4}+3822s^{3}+10927s^{2}+11466s+3969);
Sol = Solve[{f==0}, {s,k}, Integers]

(*(a,b) = (3,5), k = w^2*)
f = (-18s^{2}-54s+31)k+(576s^{4}+4896s^{3}+13716s^{2}+14076s+4761);
Sol = Solve[{f==0}, {s,k}, Integers]

(*(a,b) = (4,4), k = w^2*)
f = (-20s^{2}-60s+31)k+(625s^{4}+5250s^{3}+14475s^{2}+14490s+4761);
Sol = Solve[{f==0}, {s,k}, Integers]

(*(a,b) = (5,3), k = w^2*)
f = (-20s^{2}-60s+31)k+(576s^{4}+4800s^{3}+13024s^{2}+12600s+3969);
Sol = Solve[{f == 0}, {s,k}, Integers]

(*(a,b) = (6,2), k = w^2*)
f = (-18s^{2}-54s+31)k+(441s^{4}+3654s^{3}+9711s^{2}+8874s+2601);
Sol = Solve[{f == 0}, {s,k}, Integers]

(*(a,b) = (7,1), k = w^2*)
f = (-14s^{2}-42s+31)k+(256s^{4}+2112s^{3}+5412s^{2}+4356s+1089);
Sol = Solve[{f == 0}, {s,k}, Integers]

(*(a,b) = (8,0), k = w^2*)
f = (-8s^{2}-24s+31)k+(81s^{4}+666s^{3}+1531s^{2}+666s+81)
Sol = Solve[{f == 0}, {s,k}, Integers]
\end{verbatim}
}
\end{tbox}

For each $(a,s,w)$ triple, we can use 
\begin{tbox} {\footnotesize \begin{verbatim}
arr = [[1,1,139],[1,2,37],[2,1,35],[7,1,23],[8,1,55],[8,2,17]]

for elem in arr do{
a' = elem_0; b' = 8-a'; s' = elem_1; w' = elem_2;
Neq' = sub(Neq, {a => a', b => b', s => s', w => w'});
I = minimalPrimes ideal (Neq'_(0,0), Neq'_(0,1), Neq'_(1,0), 
Neq'_(1,1), x*w - y*z - 1);
print(I);}
\end{verbatim}
} 
\end{tbox}
This shows that there is a unique $(x,y,z)$ solution for each $(a,s,w)$ triple. 
We use the value of $z.$
To use $\Vol(\Delta)=m^2$ to find coordinates of a potential elliptic quadrilateral for each tuple $(a,s,z,w),$ we use the following code.

\begin{tbox} {\footnotesize \begin{verbatim}
QQ[x,y,z,w,s,a',b',c',d',c,d,l,t,e];

arr = [[1,1,63,139],[1,2,17,37],[2,1,11,35],[7,1,3,23],
[8,1,7,55],[8,2,2,17]];

-- input: (a,s,z,w) tuple
-- prints discriminant
f = (arr) ->{
at = arr_0; st = arr_1; zt = arr_2; wt = arr_3;
c = at+1; d = 10-c; l = st+3;
b' = (d'*wt-(10+l*c*d))/c; 
a' = zt*d'-(1+l*d)-b';

m = a'+b'+1+d';
vol = (a')*(d')*wt+(b')*d;
eq = vol - m^2;

eq = eq * c;

f0 = sub(eq, {d' => 0});
f1 = sub((eq-f0)//d', {d' => 0});
f2 = sub((eq-f0-f1*d')//(d')^2, {d' => 0});

disc = (f1)^2 - 4*(f0)*(f2);
print(eq);
print(disc);}

for elem in arr do{
f(elem);}
\end{verbatim}
} 
\end{tbox}

The discriminants are $$2903616, 
153664, 
235956, 
87616, 
367236, 
20736,$$ 
all of which except $235956$ is a square, namely $1704^2,392^2,296^2,606^2,144^2$ respectively.

\medskip

We further observe here that the leading coefficient is negative for each quadratic in $\delta/\gamma.$ Hence we can solve for $\delta/\gamma,$ and use that to solve for $\alpha/\gamma,\beta/\gamma.$
We can use these to find $\Delta$ up to scaling, and then check if the quadrilateral satisfies $\Width(\Delta) \ge m$ (which is relation invariant under scaling).

\medskip
First, a function to print the coordinates given its sidelengths $\alpha,\beta,\gamma,\delta.$
\begin{tbox} {\footnotesize \begin{verbatim}
giveCoordinatesHelper = (zt,wt,a'',b'',c'',d'') -> {
      at = sub(a'', {z=>zt, w=>wt});
      bt = sub(b'', {z=>zt, w=>wt});
      ct = sub(c'', {z=>zt, w=>wt});
      dt = sub(d'', {z=>zt, w=>wt});
      print([[0,0],[at,0],[at+bt, bt*c],[dt*zt,dt*wt]]);
  }
\end{verbatim}
} 
\end{tbox}
Then, we use the function to write a function which gives the coordinates given the tuple $(a,s,z,t).$ 
\begin{tbox} {\footnotesize \begin{verbatim}
-- t is sqrt(disc)
giveCoordinates = (arr) -> {
at = arr_0; st = arr_1; zt = arr_2;
wt = arr_3; t = arr_4;
c = at+1; d = 10-c; l = st+3;
b' = (d'*wt-(10+l*c*d)*c')/c; 
a' = zt*d'-(1+l*d)*c'-b';
-- Note: this is different from before because not dehomogenized

m = a'+b'+c'+d';
vol = (a')*(d')*wt+(b')*(c')*d;
eq = vol - m^2;

eq = eq * c;

f0 = sub(eq//(c')^2, {d' => 0});
f1 = sub((eq-f0*(c')^2)//(c'*d'), {d' => 0});
f2 = sub((eq-f0*(c')^2-f1*(c'*d'))//(d')^2, {d' => 0});

-- FIRST ROOT
-- side lengths up to scaling
d'' = f1-t;
c'' = -2*f2;
a'' = sub(a', {d' => d'', c' => c''});
b'' = sub(b', {d' => d'', c' => c''});

giveCoordinatesHelper(zt,wt,a'',b'',c'',d'');

-- SECOND ROOT
-- side lengths up to scaling
d'' = f1+t;
c'' = -2*f2;
a'' = sub(a', {d' => d'', c' => c''});
b'' = sub(b', {d' => d'', c' => c''});

giveCoordinatesHelper(zt,wt,a'',b'',c'',d'');}
\end{verbatim}
} 
\end{tbox}

Combining this with the square roots of our discriminants, we run the following to get a list of quadrilaterals, which we test on \textsc{Magma}.

\begin{tbox} {\footnotesize \begin{verbatim}
arr = [[1,1,63,139,1704],[1,2,17,37,392],[7,1,3,23,296],
[8,1,7,55,606],[8,2,2,17,144]];

for elem in arr do
giveCoordinates(elem)
\end{verbatim}
} 
\end{tbox}

We can then run the following where \texttt{arr} is the list of quadrilaterals with non-negative integer coordinates. The code returns nothing, showing that $\Width(\Delta) \ge m$ fails in each case.

\begin{tbox} {\footnotesize \begin{verbatim}
function checkAll(arr)
    for elem in arr do
        p := Polytope(elem);
        m := #BoundaryPoints(p);
        if Width(p) ge m then
        print(m);
        print("Yes"); 
        print(p);
        end if; 
    end for;
    return 0;
end function;
\end{verbatim}
} 
\end{tbox}
\end{computation}


\begin{computation}
\label{comp: An case I general}

To get Lemma \ref{lemma: pell's equation for An}, we use the function \texttt{minimalPrimes} on \textsc{Macaulay2} to get the minimal primes of the ideal generated by the $4$ equations we get from the matrix equation. 

\begin{tbox} {\footnotesize \begin{verbatim}
R = QQ[a,b,s,x,y,z,w,n];

-- Standard matrices
M = matrix{{x,y},{z,w}};
M1 = matrix{{0,-1},{1,1}}; -- matrix of (-1) curve
Ma = matrix {{1-a,-a},{a,a+1}}; -- a'th power of matrix of (-2)
Mb = matrix {{1-b,-b},{b,b+1}}; -- b'th power of matrix of (-2)

-- matrix for self-intersection k
cc = k -> matrix{{0,-1},{1,-k}}; 

-- matrix respresenting k (-2) curves
exp2 = k -> matrix {{1-k,-k}, {k,k+1}} 

-- matrix for chain contracting to An (up to orientation)
tc = s -> matrix{{-(n*s+3*n-1), -n},{(n+1)*s+3*n+2, n+1}};

-- Computing N using equation 1
-- inverse of M*M1;
MM1inv = matrix{{-z+w,x-y},{-w,y}};
N = MM1inv * tc(s);

-- equation 2
eq = N * M1 * M * M1 * Ma * cc(-(s+5)) * Mb * M1 - matrix {{1,0},{0,1}};

-- running n, a from 0 to 8 and using a+b=8-n

-- note: Macaulay2 for loops include start and end too
for n' from 1 to 8 do{  -- note, n > 0.
    print("n ", n');
        for a' from 0 to (8-n') do{
          print("a ",a');
          eq' = sub(eq, {n => n', a => a', b => 8-n'-a'});
          print(minimalPrimes ideal (eq'_(0,0), eq'_(0,1), 
          eq'_(1,0), eq'_(1,1), x*w - y*z - 1));
          print("");
          };
    print("");
    print("");
} 
\end{verbatim}
} 
\end{tbox}

To get the second equation using $\Vol(\Delta)=m^2,$ we reuse the function \texttt{idealCheck} from Computation \ref{comp: A0 case I 1}.

\medskip

We now use it on those cases we get from the matrix equation:

\begin{tbox} {\footnotesize \begin{verbatim}
---- (n,s) = (1,0) ----
idealCheck(1,0,0, 33*z^2-54*z*w+22*w^2+2)
idealCheck(1,0,7, 33*z^2-12*z*w+w^2+2)

---- (n,s) = (2,-1) ----
idealCheck(2,-1,0, 22*z^2-34*z*w+13*w^2+3)
idealCheck(2,-1,6, 22*z^2-10*z*w+w^2+3)

---- (n,s) = (3,4) ----
idealCheck(3,4,1, 77*z^2-98*z*w+29*w^2+4)
idealCheck(3,4,4, 77*z^2-56*z*w+8*w^2+4)

---- (n,s) = (4,-2) ----
idealCheck(4,-2,0, 11*z^2-15*z*w+5*w^2+5)
idealCheck(4,-2,4, 11*z^2-7*z*w+w^2+5)

---- (n,s) = (4,-1) ----
idealCheck(4,-1,1, 22*z^2-26*z*w+7*w^2+5)
idealCheck(4,-1,3, 22*z^2-18*z*w+3*w^2+5)

---- (n,s) = (4,0) ----
idealCheck(4,0,2, 33*z^2-33*z*w+7*w^2+5)
\end{verbatim}
} 
\end{tbox}

To check which curves are irreducible and find the genus, we can use the following code on \textsc{Magma}:

\begin{tbox} {\footnotesize \begin{verbatim}
A3<z,w,t> := AffineSpace(Rationals(),3);

// input: equations of the two quadrics
// output: prints if the intersections is reducible,
// singular, and the genus
CheckSingular := function(f, g)
C := Curve(A3, [f,g]);
printf "irreducible: "; print IsIrreducible(C);
if IsIrreducible(C) then
printf "singular: "; print IsSingular(C);
printf "genus = "; print Genus(C);
end if;
return "";
end function;

arr := [[33*z^2-54*z*w+22*w^2+2,
1056*z^2-4896*z*w+3104*w^2+2112*z-1536*w+1056+t^2],

[33*z^2-12*z*w+w^2+2,
1056*z^2-3552*z*w+416*w^2+2112*z-192*w+1056+t^2],

[22*z^2-34*z*w+13*w^2+3,
616*z^2-2184*z*w+1204*w^2+1232*z-784*w+616+t^2],

[22*z^2-10*z*w+w^2+3, 
616*z^2-1512*z*w+196*w^2+1232*z-112*w+616+t^2],

[77*z^2-98*z*w+29*w^2+4, 
3080*z^2-25480*z*w+11240*w^2+6160*z-2800*w+3080+t^2],

[77*z^2-56*z*w+8*w^2+4, 
3080*z^2-23800*z*w+4520*w^2+6160*z-1120*w+3080+t^2],

[11*z^2-15*z*w+5*w^2+5, 
220*z^2-520*z*w+220*w^2+440*z-200*w+220+t^2],

[11*z^2-7*z*w+w^2+5, 
220*z^2-360*z*w+60*w^2+440*z-40*w+220+t^2],

[22*z^2-26*z*w+7*w^2+5, 
704*z^2-2240*z*w+800*w^2+1408*z-512*w+704+t^2],

[22*z^2-18*z*w+3*w^2+5, 
704*z^2-1984*z*w+416*w^2+1408*z-256*w+704+t^2],

[33*z^2-33*z*w+7*w^2+5, 
1188*z^2-4752*z*w+1332*w^2+2376*z-648*w+1188+t^2]];

for elem in arr do
CheckSingular(elem[1], elem[2]);
end for;

\end{verbatim}
} 
\end{tbox}
\end{computation}


\begin{computation}
\label{comp: An case I singular}
The following code on \textsc{Mathematica} can be used to solve the equations modulo the claimed primes in Lemma \ref{lemma: An cases not locally soluble}. We find in each case that the solution set is empty.

\begin{tbox} {\footnotesize \begin{verbatim}
f = 22z^2-34z*w+13w^2+3
g = 616z^2-2184z*w+1204w^2+1232z-784w+616+t^2
Solve[{f==0,g==0},{z,w,t},Modulus->13]

f= 22z^2-10z*w+w^2+3
g = 616z^{2}-1512z*w+196w^{2}+1232z-112w+616+t^2
Solve[{f==0,g==0},{z,w,t},Modulus->13]

f = 77z^{2}-98z*w+29w^{2}+4
g = 3080z^{2}-25480z*w+11240w^{2}+6160z-2800w+3080+t^2 
Solve[{f==0,g==0},{z,w,t},Modulus->7]

f = 77z^{2}-56z*w+8w^{2}+4
g = 3080z^{2}-23800z*w+4520w^{2}+6160z-1120w+3080+t^2
Solve[{f==0,g==0},{z,w,t},Modulus->7]

f = 22z^{2}-26z*w+7w^{2}+5
g = 704z^{2}-2240z*w+800w^{2}+1408z-512w+704+t^2
Solve[{f==0,g==0},{z,w,t},Modulus->7]

f = 22z^{2}-18z*w+3w^{2}+5
g = 704z^{2}-1984z*w+416w^{2}+1408z-256w+704+t^2
Solve[{f==0,g==0},{z,w,t},Modulus->7]

f = 33z^{2}-33z*w+7w^{2}+5
g = 1188z^{2}-4752z*w+1332w^{2}+2376z-648w+1188+t^2 
Solve[{f==0,g==0},{z,w,t},Modulus->17]
\end{verbatim}
}
\end{tbox}
\end{computation}


\begin{computation}
\label{comp: An case I degenerate}
For computations here, first initialize the function \texttt{idealCheck} in Computation \ref{comp: An case I general}.

\medskip

Consider first $(n,s,a)=(1,0,0).$ Here $t = \pm (132z-96w-8).$
\begin{enumerate}
\item Consider the $+t$ root. We find $\beta/\gamma$ in Equation \ref{eq: An nsa = 100} using
\begin{tbox}
{\footnotesize
\begin{verbatim}
-- initialize f0, f1, f2
idealCheck(1,0,0, 33*z^2-54*z*w+22*w^2+2)

tt = 132*z - 96*w - 8;
-- the root is d' = (-f1+t)/2f1.
-- up to scaling, we can write 
d'' = (f1 - tt)/(-2*f2);
a'' = sub(a',{d' => d''}); 
b'' = sub(b',{d' => d''});
\end{verbatim}
}
\end{tbox}

For Claim \ref{claim: An nsa = 100}, we use

\begin{tbox} {\footnotesize \begin{verbatim}
---- Part (a) ----
bb = -33*z^2-9*z*w+23*w^2-66*z+28*w-33;
sub(bb, {z=>100/135*w+e})

---- Part (b) ----
pell = 33*z^2-54*z*w+22*w^2+2;
sub(pell, {z => 100/135*w-e})
\end{verbatim}
}
\end{tbox}

\item We get Equation \ref{eq: An nsa = 100 -t} using

\begin{tbox} {\footnotesize \begin{verbatim}
-- initialize f0, f1, f2
idealCheck(1,0,0, 33*z^2-54*z*w+22*w^2+2)

tt = -(132*z - 96*w - 8);
d'' = (f1 - tt)/(-2*f2);
a'' = sub(a',{d' => d''});
b'' = sub(b',{d' => d''});
\end{verbatim}
} 
\end{tbox}

We get Equation \ref{eq: An nsa = 100 width-m} using
\begin{tbox} {\footnotesize \begin{verbatim}
widthM = (w-z-1)*d''-b''-1;
\end{verbatim}
} 
\end{tbox}

For Claim \ref{claim: An nsa = 100 -t}, we use

\begin{tbox} {\footnotesize \begin{verbatim}
---- Part (a) ----
wnum = widthM * (z^2-z*w+w^2+2*z+1);
w1 = sub(wnum, {z=>10/12*w+e});
w2 = sub(w1, {w => 4+x})

---- Part (b) ----
pell = 33*z^2-54*z*w+22*w^2+2;
p1 = sub(pell, {w=>(12/10+e)*z});
p2 = sub(p1, {z=>11+x})

sub(p2, {e=>0})
sub(p2, {e=>1/10})

---- Part (c) ----
anum = a''* (z^2-z*w+w^2+2*z+1);
a1 = sub(anum, {w=>(13/10+e)*z});
a2 = sub(a1, {z => 37+x})

sub(a2, {e=>0})
sub(a2, {e=>5/100})
\end{verbatim}
} 
\end{tbox}

For the remaining cases, we use

\begin{tbox} {\footnotesize \begin{verbatim}
sub(b'', {z=>4,w=>5})
sub(widthM, {z=>6,w=>7})
sub(b'', {z=>10,w=>13})
sub(widthM,{z=>20,w=>23})
sub(a'', {z=>36,w=>47})
\end{verbatim}
} 
\end{tbox}
\end{enumerate}

Consider next $(n,s,a)=(1,0,7).$ Here, $t=\pm (132z - 12w - 8).$
\begin{enumerate}
\item Consider the $+t$ root. We find $\beta/\gamma$ in Equation \ref{eq: An nsa = 107 +t} using
\begin{tbox} {\footnotesize \begin{verbatim}
-- initialize f0, f1, f2
idealCheck(1,0,7,33*z^2-12*z*w + w^2+2)

tt = 132*z - 12*w - 8;
-- the root is d' = (-f1+t)/2f1.
-- up to scaling, we can write 
d'' = (f1 - tt)/(-2*f2);
a'' = sub(a',{d' => d''}); 
b'' = sub(b',{d' => d''});
\end{verbatim}
} 
\end{tbox}

The following computations are used in Claim \ref{claim: An nsa=107 +t}.

\begin{tbox} {\footnotesize \begin{verbatim}
---- part (a) -----
pell = 33*z^2-12*z*w+w^2+2;
sub(pell, {w=>3*z-e});

---- part (b) -----
bb = -132*z^2 +111*z*w - 13*w^2 - 264*z + 14*w - 132;
b1 = sub(bb, {w=>(3+e) * z});
b2 = sub(b1, {z=>3+x});

sub(b2, {e=>0})
sub(b2, {e=>2})

cc = 32*z^2 - 32*z*w + 4*w^2  + 64*z + 32;
c1 = sub(cc, {w=>(3+e) * z});
c2 = sub(c1, {z=>3+x});

sub(c2, {e=>0})
sub(c2, {e=>2})

---- part (c) -----
pell = 33*z^2-12*z*w+w^2+2;
p1 = sub(pell, {w=>(5+e)*z});
p2 = sub(p1, {z=>3+x});

sub(p2, {e=>0});
sub(p2, {e=>25/10});

---- part (d) -----
bb = -132*z^2 +111*z*w - 13*w^2 - 264*z + 14*w - 132;
sub(bb, {w=>75/10*z+e})

cc = 32*z^2 - 32*z*w + 4*w^2  + 64*z + 32;
sub(cc, {w=>75/10*z+e})
\end{verbatim}
} 
\end{tbox}
\item Consider the $-t$ root. We find $\alpha/\gamma,\beta/\gamma,\delta/\gamma$ in Equation \ref{eq: An nsa = 107 -t} using

\begin{tbox} {\footnotesize \begin{verbatim}
-- initialize f0, f1, f2
idealCheck(1,0,7,33*z^2-12*z*w + w^2+2);

tt = -(132*z - 12*w - 8);
-- the root is d' = (-f1+t)/2f1.
-- up to scaling, we can write 
d'' = (f1 - tt)/(-2*f2);
a'' = sub(a',{d' => d''}); 
b'' = sub(b',{d' => d''});
\end{verbatim}
} 
\end{tbox}

For Claim \ref{claim: An nsa = 107 -t w<8z}, we use

\begin{tbox} {\footnotesize \begin{verbatim}
pell = 33*z^2-12*z*w+w^2+2;
sub(pell, {w=>8*z+e})
\end{verbatim}
} 
\end{tbox}

To calculate Equation \ref{eq: An nsa = 107 -t width-m}, we use
\begin{tbox} {\footnotesize \begin{verbatim}
widthM = 7*a''-b''-d''-1
\end{verbatim}
} 
\end{tbox}

For Claim \ref{claim: An nsa=107 -t}, we use

\begin{tbox} {\footnotesize \begin{verbatim}
----Part (b) ----
cc =  8*z^2 - 8*z*w + w^2 + 16*z + 8;
c1 = sub(cc, {w=>(3+e)*z});
c2 = sub(c1, {z=>3+x})

sub(c2, {e=>0})
sub(c2, {e=>2})

---- Part (d) ----
bb = 662*z^2 - 157*z*w + 9*w^2  + 114*z - 15*w + 12;
b1 = sub(bb, {w=>(75/10+e)*z});
b2 = sub(b1, {z=>3+x})

sub(b2, {e=>0})
sub(b2, {e=>5/10})
\end{verbatim}
} 
\end{tbox}
\end{enumerate}

Consider next $(n,s,a)=(4,-2,0).$ Here, $t=\pm (22z - 10w - 10).$

\begin{enumerate}
    \item Consider $t=22z - 10w - 10.$ To get Equation \ref{eq: An nsa=4-20}, we use

    \begin{tbox} {\footnotesize \begin{verbatim}
-- initialize f0, f1, f2
idealCheck(4,-2,0,11*z^2-15*z*w+5*w^2+5)

tt = 22*z - 10*w - 10;
-- the root is d' = (-f1+t)/2f1.
-- up to scaling, we can write 
d'' = (f1 - tt)/(-2*f2);
a'' = sub(a',{d' => d''}); 
b'' = sub(b',{d' => d''});
    \end{verbatim}
    } 
    \end{tbox}
    \item Consider $t=-(22z - 10w - 10).$ To get Equation \ref{eq: An nsa=4-20 -t}, we use

    \begin{tbox} {\footnotesize \begin{verbatim}
-- initialize f0, f1, f2
idealCheck(4,-2,0,11*z^2-15*z*w+5*w^2+5)

tt = -(22*z - 10*w - 10);
-- the root is d' = (-f1+t)/2f1.
-- up to scaling, we can write 
d'' = (f1 - tt)/(-2*f2);
a'' = sub(a',{d' => d''}); 
b'' = sub(b',{d' => d''});
    \end{verbatim}
    } 
    \end{tbox}

    To find the difference $\Width_{(1,-1)}-m$ in Equation \ref{eq: An nsa=4-20 -t width-m}, we use
    \begin{tbox} {\footnotesize \begin{verbatim}
widthM = (w-z-1)*d''-b''-1
    \end{verbatim}
    } 
    \end{tbox}

    For Claim \ref{claim: An nsa=4-20 -t}, we use 

    \begin{tbox} {\footnotesize \begin{verbatim}
---- Part (a) -----
wnum = -6*z^2 - 10*z*w + 10*w^2 + 4*z + 10;
w1 = sub(wnum, {w=>(1+e)*z});
w2 = sub(w1, {z=>31+x})

sub(w2,{e=>0})
sub(w2,{e=>3/10})

---- Part (b) -----
pell = 11*z^2-15*z*w+5*w^2+5;
p1 = sub(pell, {w=>(13/10+e)*z});
p2 = sub(p1, {z=>31+x})

sub(p2,{e=>0})
sub(p2,{e=>42/100}) 

---- Part (c) -----
anum = 21*z^2-21*z*w+5*w^2+10*z+5;
a1 = sub(anum, {w=>(172/100+e)*z});
a2 = sub(a1, {z=>31+x})

sub(a2,{e=>0})
sub(a2,{e=>28/100}) 

---- Part (d) -----
pell = 11*z^2-15*z*w+5*w^2+5;
p1 = sub(pell, {w=>2*z+e})
    \end{verbatim}
    } 
    \end{tbox}

The following code shows $\beta/\gamma<0$ for $(z,w) \in (5,8),$ and $\Width_{(1,-1)}<m$ for $(z,w) \in \{(10,13),(25,32)\}:$
    \begin{tbox} {\footnotesize \begin{verbatim}
sub(b'',{z=>5,w=>8})
sub(widthM,{z=>10,w=>13})
sub(widthM,{z=>25,w=>32})
    \end{verbatim}
    } 
    \end{tbox}

To analyze the remaining cases, we compute the coordinates of $\Delta$ using the \texttt{findCoordinates} function from Computation \ref{comp: A0 Case I (a,b,s)=(0,8,2)}:

    \begin{tbox} {\footnotesize \begin{verbatim}
arr = [(5,7),(10,17),(25,43)]
for elem in arr do
    findCoordinates(elem);
\end{verbatim}
    } 
    \end{tbox}
The output is

\begin{tbox} {\footnotesize \begin{verbatim}
[[0, 0], [180, 0], [200, 20], [800, 1120]]
[[0, 0], [160, 0], [320, 160], [3200, 5440]]
[[0, 0], [100, 0], [2600, 2500], [20000, 34400]]
\end{verbatim}
} 
\end{tbox}

We can test all of these quadrilaterals on \textsc{Magma} using the \texttt{checkAll} function from Computation \ref{comp: A0 Case I (a,b,s)=(0,8,2)}. The first quadrilateral, call $\Delta,$ is the only which satisfies $\Width(\Delta) \ge m.$
On \textsc{Magma}, consider the following code:

\begin{tbox} {\footnotesize \begin{verbatim}
p := Polytope([[0, 0], [9, 0], [10, 1], [40, 56]]);
m := #BoundaryPoints(p);
#FindCurves(p, m, Rationals());
IsIrreducible(FindCurve(p, m, Rationals()));
\end{verbatim}
} 
\end{tbox}

Hence the polygon \texttt{p} satisfies $\dim \mc L_{p}(|\partial p \cap \ZZ^2|)=1,$ but the generator is reducible. 
Since one gets $\Delta$ on scaling \texttt{p} by $20$, the linear system $\dim \mc L_{\Delta}(|\partial \Delta \cap \ZZ^2|)$ also contains a reducible element, showing it is not elliptic.


\end{enumerate}

Consider next $(n,s,a)=(4,-2,4).$ Here, $t=\pm (22z - 2w - 10).$

\begin{enumerate}
    \item Consider $t=22z - 2w - 10.$ To get Equation \ref{eq: An nsa=4-24 +t}, we use

\begin{tbox} {\footnotesize \begin{verbatim}
-- initialize f0, f1, f2
idealCheck(4,-2,4, 11*z^2-7*z*w+w^2+5)

tt = 22*z - 2*w - 10;
-- the root is d' = (-f1+t)/2f1.
-- up to scaling, we can write 
d'' = (f1 - tt)/(-2*f2);
a'' = sub(a',{d' => d''}); 
b'' = sub(b',{d' => d''});
\end{verbatim}
    } 
    \end{tbox}

    For Claim \ref{claim: An nsa=4-24 +t}, we use

    \begin{tbox} {\footnotesize \begin{verbatim}
---- Part (a) ----
pell = 11*z^2-7*z*w+w^2+5;
p1 = sub(pell, {w=>2*z-e})
    
---- Part (b) ----
bb =  -55*z^2 + 49*z*w - 9*w^2 - 110*z + 10*w - 55;
b1 = sub(bb, {w=>(2+e)*z});
b2 = sub(b1, {z=>14+x})

cc = 25*z^2 - 25*z*w + 5*w^2 + 50*z + 25;
c1 = sub(cc, {w=>(2+e)*z});
c2 = sub(c1, {z=>14+x})

sub(b2, {e=>0})
sub(b2, {e=>5/10})

sub(c2, {e=>0})
sub(c2, {e=>5/10})

---- Part (c) ----
p1 = sub(pell, {w=>(25/10+e)*z})
p2 = sub(p1, {z=>14+x})

sub(p2, {e=>0})
sub(p2, {e=>19/10})

---- Part (d) ----
b1 = sub(bb, {w=>(44/10+e)*z});
b2 = sub(b1, {z=>14+x})

c1 = sub(cc, {w=>(44/10+e)*z});
c2 = sub(c1, {z=>14+x})
\end{verbatim}
    } 
    \end{tbox}

    \item Consider $t=(22z - 2w - 10).$ To get Equation \ref{eq: An nsa=4-24 -t}, we use

    \begin{tbox} {\footnotesize \begin{verbatim}
-- initialize f0, f1, f2
idealCheck(4,-2,4, 11*z^2-7*z*w+w^2+5)

tt = -(22*z - 2*w - 10);
-- the root is d' = (-f1+t)/2f1.
-- up to scaling, we can write 
d'' = (f1 - tt)/(-2*f2);
a'' = sub(a',{d' => d''}); 
b'' = sub(b',{d' => d''});
\end{verbatim}
    } 
    \end{tbox}

    For Claim \ref{claim: An nsa = 4-24 -t w<5z}, we use

\begin{tbox} {\footnotesize \begin{verbatim}
pell = 11*z^2-7*z*w+w^2+5;
sub(pell, {w=>5*z+e})
\end{verbatim}
} 
\end{tbox}

    For Equation \ref{eq: An nsa=4-24 -t width-m}, we use

\begin{tbox} {\footnotesize \begin{verbatim}
widthM = 4*a''-b''-d''-1
\end{verbatim}
} 
\end{tbox}

    For Claim \ref{claim: An nsa=4-24 -t}, we use

\begin{tbox} {\footnotesize \begin{verbatim}
bb = 74*z^2-26*z*w+2*w^2+4*z+10;
b1 = sub(bb, {w=>(44/10+e)*z});
b2 = sub(b1, {z=>14+x})

sub(b2, {e=>0})
sub(b2, {e=>6/10})
\end{verbatim}
} 
\end{tbox}

The following code can be used to show $\beta/\gamma<0$ for $(z,w)=(2,7),$ and $\Width_{(5,-1)}<m$ for $(z,w) =(7,32),$ 
\begin{tbox} {\footnotesize \begin{verbatim}
sub(b'', {z=>2,w=>7})
sub(widthM, {z=>7, w=>32})
\end{verbatim}
} 
\end{tbox}

To analyze the remaining cases, we compute the coordinates of $\Delta$ using the \texttt{findCoordinates} function from Computation \ref{comp: A0 Case I (a,b,s)=(0,8,2)}:

\begin{tbox} {\footnotesize \begin{verbatim}
arr = [(3,8),(3,13),(7,17)]
for elem in arr do
    findCoordinates(elem);
\end{verbatim}
    } 
    \end{tbox}

The output is

\begin{tbox} {\footnotesize \begin{verbatim}
[[0, 0], [144, 0], [192, 240], [288, 768]]
[[0, 0], [60, 0], [72, 60], [288, 1248]]
[[0, 0], [812, 0], [1512, 3500], [1568, 3808]]
\end{verbatim}
} 
\end{tbox}

We can test all of these quadrilaterals on \textsc{Magma} using the \texttt{checkAll} function from Computation \ref{comp: A0 Case I (a,b,s)=(0,8,2)}. 
The second quadrilateral, call $\Delta,$ is the only which satisfies $\Width(\Delta) \ge m.$
On \textsc{Magma}, consider the following code:

\begin{tbox} {\footnotesize \begin{verbatim}
p := Polytope([[0, 0], [5, 0], [6, 5], [24, 104]]);
m := #BoundaryPoints(p);
#FindCurves(p, m, Rationals());
IsIrreducible(FindCurve(p, m, Rationals()));
\end{verbatim}
} 
\end{tbox}

Hence the polygon \texttt{p} satisfies $\dim \mc L_{p}(|\partial p \cap \ZZ^2|)=1,$ but the generator is reducible. 
Since one gets $\Delta$ on scaling \texttt{p} by $12$, the linear system $\dim \mc L_{\Delta}(|\partial \Delta \cap \ZZ^2|)$ also contains a reducible element, showing it is not elliptic.
\end{enumerate}
\end{computation}

\begin{computation}
\label{comp: An Case II}

The following code uses \texttt{eliminate} to find an expression involving $w,s$ in the ideal.
\begin{tbox} {\footnotesize \begin{verbatim}
R = QQ[a,b,s,x,y,z,w,n];
M1 = matrix{{0,-1},{1,1}};
Ma = matrix {{1-a,-a},{a,a+1}};
Mb = matrix {{1-b,-b},{b,b+1}};
cc = k -> matrix{{0,-1},{1,-k}}; -- "canonical matrix"
M = matrix {{x,y},{z,w}};

-- Function to "reverse" a matrix
rev = A -> matrix{{A_(0,0),-A_(1,0)},{-A_(0,1),A_(1,1)}}

-- inverse of M * M1
MM1inv = matrix{{-z+w,x-y},{-w,y}};

prod = matrix{{-(n*s+3*n-1), -n},{(n+1)*s+3*n+2,n+1}};
N' = MM1inv * prod;

N = rev(N');

Neq = N*M1*M*M1*Ma*cc(-(s+5))*Mb*M1 - matrix{{1,0},{0,1}};

for n' from 1 to 8 do{
    for a' from 0 to 8-n' do {
        print(n', a', 8-a'-n');
        Neq' = sub(Neq, {a=>a', b=>8-n'-a', n=>n'});
        Iarray = minimalPrimes ideal (Neq'_(0,0), Neq'_(1,0), 
        Neq'_(0,1), Neq'_(1,1), x*w-y*z-1);

-- we observe only one ideal in each case
-- but to confirm, we add an if statement
        if #Iarray > 1 then{
            print("Many ideals");
        }; 
        
        J = eliminate({x,y,z},Iarray_0);
        f0 = sub(J_0, {w=>0});
        f1 = sub((J_0-f0)//w, {w=>0});
        f2 =  sub((J_0-f0-w*f1)//w^2, {w=>0});
-- expression is f0+f1*w+f2*w^2 = 0
        print(factor f0);
        print(factor f1);
        print(factor f2);
        print("");
        print("");
    }
}
\end{verbatim}
} 
\end{tbox}
    
For each $(n,a,b,s,w)$ triple, we can use

\begin{tbox} {\footnotesize \begin{verbatim}
arr = [[1,2,5,-1,121],[1,2,5,0,59],[1,3,4,-1,43],[1,3,4,8,143],
[1,6,1,-1,27],[1,6,1,8,99],[1,7,0,-1,41],[1,7,0,0,23],
[2,1,5,-2,31],[2,1,5,-1,29],[2,1,5,19,319],[2,2,4,-2,17],
[2,4,2,-2,14],[2,5,1,-2,19],[2,6,0,-1,13],[3,1,4,-2,13],
[3,4,1,-2,9],[3,5,0,-2,11],[4,0,4,-2,17],[4,0,4,-1,43],
[4,0,4,0,137],[4,2,2,-2,10],[4,3,1,-1,17],[4,4,0,-2,7],
[6,1,1,-2,9],[6,2,0,-2,7],[6,2,0,0,19],[6,2,0,30,209],
[7,1,0,-2,11],[7,1,0,-1,13],[7,1,0,19,143]]

for elem in arr do{
n' = elem_0, a' = elem_1; b' = elem_2; s' = elem_3; w' = elem_4;
Neq' = sub(Neq, {n => n', a => a', b => b', s => s', w => w'});
I = minimalPrimes ideal (Neq'_(0,0), Neq'_(0,1), Neq'_(1,0), 
Neq'_(1,1), x*w - y*z - 1);
print(I);}
\end{verbatim}
} 
\end{tbox}

This shows that there is a unique $(x, y, z)$ solution for each $(n,a,b,s,w)$ tuple. 
We use the value of $z.$
To use $\Vol(\Delta)=m^2$ to find coordinates of a potential elliptic quadrilateral for each tuple $(n, a, b, s, z, w),$ we use the following code. 

\begin{tbox} {\footnotesize \begin{verbatim}
R = QQ[x,y,z,w,s,a',b',c',d',c,d,l,t,e];

-- format: (n,a,b,s,z,w)
arr = [[1,2,5,-1,36,121],[1,2,5,0,18,59],[1,3,4,-1,10,43],
[1,3,4,8,35,143], [1,6,1,-1,4,27],[1,6,1,8,14,99],
[1,7,0,-1,6,41],[1,7,0,0,3,23],[2,1,5,-2,12,31],
[2,1,5,-1,12,29],[2,1,5,19,156,319], [2,2,4,-2,5,17],
[2,4,2,-2,3,14],[2,5,1,-2,4,19],[2,6,0,-1,2,13],
[3,1,4,-2,5,13],[3,4,1,-2,2,9],[3,5,0,-2,3,11],
[4,0,4,-2,10,17],[4,0,4,-1,30,43], [4,0,4,0,105,137],
[4,2,2,-2,3,10],[4,3,1,-1,4,17],[4,4,0,-2,2,7],
[6,1,1,-2,4,9],[6,2,0,-2,3,7],[6,2,0,0,6,19],
[6,2,0,30,69,209], [7,1,0,-2,7,11],[7,1,0,-1,6,13],
[7,1,0,19,70,143]]

-- map to change ring to integers
-- will be used to find square root of discriminant
mapToZ = map(ZZ,R,{});

-- input: (n,a,b,s,z,w) tuple
-- first solves vol = m^2 and finds roots
-- uses that to get coordinates of sides
-- if quadratic has no integer solutions,
-- then will print non-integer coordinates (since
-- it approximates square root of discriminant)

f = (elem) ->{
n' = elem_0, at = elem_1; bt = elem_2; st = elem_3; 
zt = elem_4; wt = elem_5;
c = at+1; d = bt+1; l = st+3;
b' = (d'*wt-(c+d+l*c*d))/c; 
a' = zt*d'-(1+l*d)-b';

m = a'+b'+1+d';
vol = (a')*(d')*wt+(b')*d;
eq = vol - m^2;

eq = eq * c;

f0 = sub(eq, {d' => 0});
f1 = sub((eq-f0)//d', {d' => 0});
f2 = sub((eq-f0-f1*d')//(d')^2, {d' => 0});

disc = (f1)^2 - 4*(f0)*(f2);
-- First change ring of disc to integers
-- then take square root
print(elem);
print(sqrt mapToZ(disc));
print("");
print("");}

for elem in arr do{
f(elem);}
\end{verbatim}
} 
\end{tbox}

First, a function to print the coordinates given its sidelengths $\alpha,\beta,\gamma,\delta.$

\begin{tbox} {\footnotesize \begin{verbatim}
-- given z,t and side lengths, give coordinates
-- of quadrilateral
giveCoordinatesHelper = (zt,wt,a'',b'',c'',d'') -> {
      at = sub(a'', {z=>zt, w=>wt});
      bt = sub(b'', {z=>zt, w=>wt});
      ct = sub(c'', {z=>zt, w=>wt});
      dt = sub(d'', {z=>zt, w=>wt});
      print([[0,0],[at,0],[at+bt, bt*c],[dt*zt,dt*wt]]);
}
\end{verbatim}
} 
\end{tbox}

Then, we use the function to write a function which gives the coordinates given the tuple $(a,b,s,z,t),$ where $t$ is the square root of the discriminant.

\begin{tbox} {\footnotesize \begin{verbatim}
giveCoordinates = (elem) -> {
n' = elem_0; at = elem_1; bt = elem_2; st = elem_3;  
zt = elem_4; wt = elem_5; tt = elem_6;
c = at+1; d = bt+1; l = st+3;
-- observe that we scale by c'f
b' = (d'*wt-(c+d+l*c*d)*c')/c; 
a' = zt*d'-(1+l*d)*c'-b';

m = a'+b'+c'+d';
vol = (a')*(d')*wt+(b')*(c')*d;
eq = vol - m^2;

eq = eq * c;

f0 = sub(eq//(c')^2, {d' => 0});
f1 = sub((eq-f0*(c')^2)//(c'*d'), {d' => 0});
f2 = sub((eq-f0*(c')^2-f1*(c'*d'))//(d')^2, {d' => 0});

-- FIRST ROOT
-- side lengths up to scaling
d'' = f1-tt;
c'' = -2*f2;
a'' = sub(a', {d' => d'', c' => c''});
b'' = sub(b', {d' => d'', c' => c''});

giveCoordinatesHelper(zt,wt,a'',b'',c'',d'');

-- SECOND ROOT
-- side lengths up to scaling
d'' = f1+tt;
c'' = -2*f2;
a'' = sub(a', {d' => d'', c' => c''});
b'' = sub(b', {d' => d'', c' => c''});

giveCoordinatesHelper(zt,wt,a'',b'',c'',d'');}
\end{verbatim}
} 
\end{tbox}

Combining this with the square roots of our discriminants, we run the following to
get a list of quadrilaterals, which we test on \textsc{Magma}.

\begin{tbox} {\footnotesize \begin{verbatim}
R = QQ[x,y,z,w,s,a',b',c',d',c,d,l,t,e];

arr = [[1, 2, 5, -1, 36, 121, 996], [1, 2, 5, 0, 18, 59, 420], 
[1, 7, 0, -1, 6, 41, 264], [1, 7, 0, 0, 3, 23, 112], 
[2, 1, 5, -2, 12, 31, 204], [2, 1, 5, -1, 12, 29, 132], 
[2, 1, 5, 19, 156, 319, 1284], [2, 5, 1, -2, 4, 19, 156], 
[4, 0, 4, -2, 10, 17, 40], [4, 0, 4, -1, 30, 43, 90], 
[4, 0, 4, 0, 105, 137, 290], [4, 4, 0, -2, 2, 7, 20], 
[7, 1, 0, -2, 7, 11, 16], [7, 1, 0, -1, 6, 13, 8],
[7, 1, 0, 19, 70, 143, 136]]

for elem in arr do{
giveCoordinates(elem);} 
\end{verbatim}
} 
\end{tbox}

We can test all of these quadrilaterals on \textsc{Magma} using the \texttt{checkAll} function from Computation \ref{comp: A0 Case I (a,b,s)=(0,8,2)}. The code returns nothing, showing that $\Width(\Delta) \ge m$ fails in each case.

\end{computation}
\end{document}